\documentclass[12pt,oneside,english,jou]{amsart}
\usepackage[T1]{fontenc}
\usepackage[latin9]{inputenc}
\usepackage{babel}
\usepackage{verbatim}
\usepackage{mathrsfs}
\usepackage{amstext}
\usepackage{amsthm}
\usepackage{amssymb}
\usepackage{amsfonts} 	 	
\usepackage{amsmath}
\usepackage{esint}
\usepackage{enumerate}
\usepackage{graphicx}
\usepackage[unicode=true,pdfusetitle,
bookmarks=true,bookmarksnumbered=false,bookmarksopen=false,
breaklinks=false,pdfborder={0 0 1},colorlinks=false]
{hyperref}
\usepackage[margin=2.5cm, tmargin=3.7cm, bmargin=3.7cm]{geometry}
\usepackage[foot]{amsaddr}
\usepackage{mathtools}
\usepackage{ dsfont }
\usepackage[shortlabels]{enumitem}
\usepackage{bm}
\usepackage{float}
\usepackage{wrapfig}
\usepackage[usenames,dvipsnames]{xcolor}
\usepackage{layout}
\usepackage[title]{appendix}
\definecolor{blue}{rgb}{0,0,1}

\newcommand*{\e}[1]{\text{e}^{#1}}
\makeatletter
\numberwithin{equation}{section}
\numberwithin{figure}{section}
\theoremstyle{plain}
\newtheorem{thm}{\protect\theoremname}[section]
\theoremstyle{definition}
\newtheorem{rem}[thm]{\protect\remarkname}
\newtheorem{remark}[thm]{Remark}

\newtheorem{defn}[thm]{\protect\definitionname}
\theoremstyle{plain}
\newtheorem{prop}[thm]{\protect\propositionname}
\theoremstyle{plain}
\newtheorem{lem}[thm]{\protect\lemmaname}
\theoremstyle{plain}
\newtheorem{conjecture}{\protect\conjecturename}
\theoremstyle{plain}
\newtheorem{cor}[thm]{\protect\corollaryname}

\theoremstyle{definition}

\newtheorem{examplex}[thm]{\protect\examplename}

\newtheorem{lemma}[thm]{Lemma}

\newtheorem{definition}[thm]{Definition}

\newenvironment{example}
{\pushQED{\qed}\examplex}
{\popQED\endexamplex}

\usepackage{tikz}
\usetikzlibrary{shapes,arrows}
\usepackage{verbatim}
\usepackage{amsthm}
\usepackage{amstext}

\DeclareMathOperator{\diam}{diam}

\DeclareMathOperator{\supp}{supp}

\newcommand{\ind}{\mathds{1}}

\newcommand{\R}{\mathbb R}
\newcommand{\Z}{\mathbb Z}
\newcommand{\N}{\mathbb N}

\newcommand{\eps}{\varepsilon}

\newcommand{\mF}{\mathcal{F}}

\newcommand{\tc}{\tilde{c}}

\newcommand{\mE}{\mathcal{E}}

\makeatother

\providecommand{\conjecturename}{Conjecture}
\providecommand{\corollaryname}{Corollary}
\providecommand{\definitionname}{Definition}
\providecommand{\examplename}{Example}
\providecommand{\lemmaname}{Lemma}
\providecommand{\problemname}{Problem}
\providecommand{\propositionname}{Proposition}
\providecommand{\remarkname}{Remark}
\providecommand{\theoremname}{Theorem}
\providecommand{\taskname}{Task}

\def\diam{{\rm diam}}

\def\supp{{\rm supp}}

\def\Mk{{\mathcal M}}
\def\Dh{\dim_{\rm H}}

\def\half{\frac{1}{2}}

\newcommand{\lam}{\lambda}

\def\Lam{\Lambda}
\newcommand{\gam}{\gamma}
\newcommand{\om}{\omega}

\newcommand{\Sig}{\Sigma}

\newcommand{\bi}{{\bf i}}
\newcommand{\bj}{{\bf j}}

\def\N{{\mathbb N}}

\def\Ak{{\mathcal A}}

\def\Sk{{\mathcal S}}
\def\Ik{{\mathcal I}}
\def\Jk{{\mathcal J}}
\def\Lk{{\mathcal L}}

\def\Vk{{\mathcal V}}

\def\be{\begin{equation}}
	\def\ee{\end{equation}}

\newcommand{\Ek}{{\mathcal E}}
\newcommand{\Fk}{{\mathcal F}}
\def\Gk{{\mathcal G}}

\def\ov{\overline}

\newcommand{\const}{{\rm const}}

\def\what{\widehat}
\def\wt{\widetilde}
\def\spt{{\rm spt}}

\def\shalf{{\textstyle{\half}}}

\definecolor{brick}{HTML}{FF0800}

\definecolor{myred1}{RGB}{255, 0, 0}
\definecolor{myyellow1}{RGB}{255, 255, 219}
\definecolor{mygreen1}{RGB}{0, 255, 0}
\definecolor{mygreen2}{RGB}{0, 126, 0}
\definecolor{myblue1}{RGB}{0, 0, 255}

\begin{document}	

	\title{Typical dimension and absolute continuity for classes of dynamically defined measures, Part II : exposition and extensions}
	
	\author{Bal\'azs B\'ar\'any$^1$}
	\address{$^1$Department of Stochastics,	Institute of Mathematics, Budapest University of Technology and Economics, M\H{u}egyetem rkp. 3, H-1111 Budapest, Hungary}
	\email{balubsheep@gmail.com}
	
	\author{K\'aroly Simon$^{1,2}$}
	\address{$^2$Budapest University of Technology and Economics, Department of Stochastics, MTA-BME Stochastics Research Group, P.O.Box 91, 1521 Budapest, Hungary}
	\email{simonk@math.bme.hu}
	
	\author{Boris Solomyak$^3$}
	\address{$^3$Bar-Ilan University, Department of Mathematics, Ramat Gan, 5290002 Israel}
	\email{bsolom3@gmail.com}

	\author{Adam \'Spiewak$^4$}
	\address{$^4$Institute of Mathematics, Polish Academy of Sciences, ul.~\'Sniadeckich 8, 00-656 Warsaw, Poland}
	\email{ad.spiewak@gmail.com}
	
	\date{\today}

	\thanks{BS was supported in part by the Israel Science Foundation grant \#1647/23. A\'S was partially supported by the National Science Centre (Poland) grant 2020/39/B/ST1/02329.}
	
	\begin{abstract}
		This paper is partly an exposition, and partly an extension of our work \cite{BSSS} to the multiparameter case. 
		We consider certain classes of parametrized dynamically defined measures. 
		These are push-forwards, under the natural projection, of ergodic measures for parametrized families of smooth iterated function systems (IFS) on the line. Under some assumptions, most crucially,
		a transversality condition, we obtain formulas for the Hausdorff dimension of the measure and absolute continuity for almost every parameter in the appropriate  parameter 
		region. The main novelty of \cite{BSSS} and the present paper is that not only the IFS, but also the ergodic measure in the symbolic space, whose push-forward we consider, depends on the parameter. 
		This includes many interesting families of measures, in particular, invariant measures for IFS's with place-dependent probabilities and natural (equilibrium) measures for smooth IFS's. One of the goals of this paper is to present an exposition of \cite{BSSS} in a more reader-friendly way, emphasizing the ideas and proof strategies, but omitting the more technical parts. This exposition/survey is based in part on the series of lectures by K\'aroly Simon at the Summer School ``Dynamics and Fractals'' in 2023 at the Banach Center, Warsaw. The main new feature, compared to \cite{BSSS}, is that we consider multi-parameter families; in other words, the set of parameters is allowed to be multi-dimensional. This broadens the scope of applications. A new 
		application considered here is to a class of Furstenberg-like measures, see Section~\ref{Furst-like}.
	\end{abstract}
	
	
	
	\maketitle
	\tableofcontents

	
	\section{Introduction}
	\subsection{Hyperbolic iterated function systems on the line}
	For the entire note we fix
	an $m \geq 2$ and a compact non-degenerate interval $X\subset \mathbb{R}$. We write
	$\mathcal{A}:=\left\{1,\dots  ,m\right\}$.
	\begin{wrapfigure}[10]{l}{0.5\textwidth}
		\begin{center}
			\includegraphics[width=0.5\textwidth]{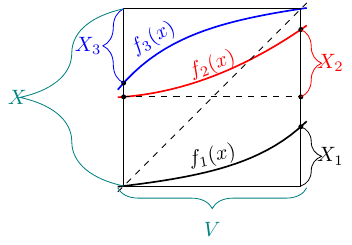}
		\end{center}
		\vspace{-10pt}
		\label{z95}
		\vspace{-10pt}
	\end{wrapfigure}
	
	\begin{definition}\label{R99}
		For an $r>1$
		we say that \texttt{
			$\mathcal{F}=\left\{f_i\right\}_{i\in\mathcal{A}}$ is
			a $\mathcal{C}^r$-smooth hyperbolic IFS on $X$} (see Figure \ref{z95}) if the following two assumptions hold:
		
		\begin{enumerate}
			\item There exists an open set $V\supset X$ such that for every $i\in \mathcal{A}$,
			$f_i:V\to f_i(V)\subset V$ is a $\mathcal{C}^r$ diffeomorphism ,
			
			\item  $f_i(X)\subset  X$ and $0<|f'_i(x)|<1$  for all $i\in\mathcal{A}$ and $x\in X$.
		\end{enumerate}
	\end{definition}
	Throughout this paper we only consider $\mathcal{C}^r$-smooth hyperbolic IFS for some $r>1$ and for the main part consider $r>2$.
	The attractor $\Lambda=\Lambda ^{\mathcal{F}} $  is the unique non-empty compact set satisfying the self-conformality equation
	\begin{equation}
		\label{R98}
		\Lambda =\bigcup _{i\in\mathcal{A}} f_i(\Lambda ).
	\end{equation}
	The level-$1$ cylinders are $X_i:=f_i(X)$ and the level-$n$ cylinders are
	$X_{i_1\dots   i_n}:=f_{i_1\dots   i_n}(X)$ for all $(i_1,\dots  ,i_n)\in\mathcal{A}^n$, where we used the shorthand notation
	$f_{i_1\dots   i_n}:=f_{i_1}\circ\cdots\circ f_{i_n}$.
	It is easy to see that
	\begin{equation}
		\label{R97}
		\Lambda =\bigcap _{n=1}^{\infty   }\bigcup\limits_{(i_1,\dots  ,i_n)\in \mathcal{A}^n} X_{i_1\dots  i_n}.
	\end{equation}
	Thus, it follows from the second part of Definition \ref{R99} that
	the sequence $\Bigl\{\bigcup\limits_{(i_1,\dots  ,i_n)\in \mathcal{A}^n} X_{i_1\dots  i_n}\Bigr\}_{n=1}^{\infty   }$ is a nested sequence of compact sets and for every $n$,  $\left\{X_{i_1\dots  i_n}\right\}_{(i_1,\dots  ,i_n)\in\mathcal{A}^n}$ is a cover of $\Lambda $. So,
	\begin{equation}
		\label{R96}
	\Bigl\{\bigcup\limits_{\mathbf{i}\in \mathcal{A}^n} X_{\mathbf{i}}\Bigr\}_{n=1}^{\infty   }
	\end{equation}
	can be considered as the \texttt{natural covering system}
	of the attractor. If the level-$1$ cylinders are disjoint:
	\begin{equation}
		\label{R93}
		X_i\cap X_j=\emptyset, \quad \text{for all } i\ne j,
	\end{equation}
	then the dimension theory of $\Lam$ is well understood.
	Observe that in this case $\Lambda $ is the repeller of an expanding map $\phi :\bigcup _{i=1}^{m }X_i\to X$ and $\phi (x):=f^{-1}_i(x)$ if $x\in X_i$.
	\begin{definition}\label{R61}
		Let $\Theta$ be the set of all $C^{r }$-smooth  hyperbolic IFS on $X$ consisting of $m$ functions:  $\mathcal{F}=(f_1,\dots  ,f_m)$.
		For $\Fk \in \Theta$, let
		$
		L(\Fk) = \sup\limits_{i\in\mathcal{A}} \|f''_i\|_{\infty  }
		$.
		Let $0<\gamma _1< \gamma _2 <1$. We introduce
		\begin{equation}
			\label{P95}
			\Theta _{\gamma _1,\gamma _2}
			=
			\left\{
			\mathcal{F}\in \Theta:
			\gamma _1 \leq  |f'_i(x) | \leq \gamma _2,\quad
			\forall i\in\mathcal{A},\quad x\in X
			\right\}.
		\end{equation}
		For a  $1<q \leq  r$
		and  $h\in\mathcal{C}^r(X)$ we define
		\begin{equation}
			\label{R58}
			\|h\|_q:=
			\left\{
			\begin{array}{ll}
				\sum_{k=0}^{ q} \|h^{(k)}\|_\infty
				,&
				\hbox{if $q\in\mathbb{N}$ ;}
				\\
				\sum_{k=0}^{ \lfloor q\rfloor} \|h^{(k)}\|_\infty
				+
				\sup\limits_{x\ne y\in X}
				\frac{|h^{(\lfloor q\rfloor)}(x)-h^{(\lfloor q\rfloor)}(y)  |}
				{|x-y|^{q-\lfloor q\rfloor}}
				,&
				\hbox{if $q\not\in \mathbb{N}$ .}
			\end{array}
			\right.
		\end{equation}
		Since we will consider families of IFS's it is useful to
		define the distance between IFS's: let $1<q \leq  r$  and let $\mathcal{F}=\left\{f_1,\dots   f_m\right\},\mathcal{G}=\left\{g_1,\dots   g_m\right\}\in \Theta $.
		Then their $q$-distance is
		\begin{equation}
			\label{rho metric}
			\varrho_q(\mathcal{F},\mathcal{G}):=
			\max\limits_{i\in \mathcal{A}}
			\|f_i-g_i\|_q.
		\end{equation}

	\end{definition}

	\begin{lem}[Bounded Distortion Property]\label{R60}
		Let $r = 1 + \delta > 1$ with $\delta \in (0,1)$ and consider $\mathcal{F}\in\Theta_{\gamma_1, \gamma_2} $.
		\begin{enumerate}
			[label=\emph{(\alph*)}]
			\item There exist constants $c_1,c_2>0$ such that for all $n$ and $\pmb{\omega}\in\mathcal{A}^n$ and for all $x,y\in X$,
			\begin{equation}
				\label{R59}
				c_1<\frac{|f'_{\pmb{\omega}}(x)|}{|f'_{\pmb{\omega}}(y)|}<c_2.
			\end{equation}
			\item There exists a constant $c_3>0$ such that
			for all $\mathcal{G}\in \Theta_{\gamma_1, \gamma_2}$ with $\varrho_{1 + \delta}(\Fk, \Gk) \leq 1$ and $n\in \mathbb{N}$, $\pmb{\omega}\in \mathcal{A}^n$,
			\begin{equation}
				\label{R57}
				\exp \left[-n c_3 \varrho_{1 + \delta}(\mathcal{F},\mathcal{G})^{\delta} \right]
				<
				\frac{|f'_{\pmb{\omega}}(0)|}{|g'_{\pmb{\omega}}(0)|}
				<
				\exp \left[-n c_3 \varrho_{1 + \delta}(\mathcal{F},\mathcal{G})^\delta \right]
			\end{equation}
		\end{enumerate}
		\end{lem}
		The proof is available in \cite[Section 14]{ourbook}.

	\subsection{Dimension of the attractor of a hyperbolic IFS. The non-overlapping case}
	Let $\mathcal{F}\in \Theta $.
	Our objective in this paper is to get a better understanding of what happens in the overlapping case, like in Figure \ref{z95} 
	where cylinder intervals $X_2$ and $X_3$ overlap. However, first we give a brief summary of the dimension theory in the non-overlapping case.
	Recall the definition of the  Hausdorff dimension: 
	\begin{equation}
		\label{R90}
		\dim_{\rm H}  \Lambda =\inf\left\{t \geq 0:
		\forall \varepsilon >0,\
		\exists
		\left\{A_i\right\}_{i=1}^{\infty},\ A_i\subset \mathbb{R} \text{ such that }
		\colorbox{lightgray}
		{$\sum\limits_{i}^{\infty  }|A_i|^t \leq \varepsilon$},\
		\Lambda \subset \bigcup_{i=1}^{\infty } A_i
		\right\},
	\end{equation}
	where $|\cdot|$ denotes the diameter. If \eqref{R93} holds, then
	the system of covers $\Bigl\{\bigcup\limits_{\mathbf{i}\in \mathcal{A}^n} X_{\mathbf{i}}\Bigr\}_{n=1}^{\infty   }$
	can serve as the system of most optimal covers $\left\{A_i\right\}_i$
	in \eqref{R90}.
	Hence, the  Hausdorff dimension  of $\Lambda $ is given by
	\begin{equation}
		\label{R95}
		\dim_{\rm H}  \Lambda =s^{\mathcal{F}},
	\end{equation}
	where
	\begin{equation}
		\label{R94}
		s^{\mathcal{F}}:=\lim\limits_{n\to\infty}
		s_n,\quad \text{ and $s_n$ is the solution of   }\quad
		\sum\limits_{\mathbf{i}\in\mathcal{A} }|X_{\mathbf{i}}|^{s_n}=1.
	\end{equation}
	(See \cite[Chapter 5]{falcbook3} and
	\cite[Theorem 14.2.2]{ourbook}.)
	We call $s^{\mathcal{F}}$ \texttt{the conformal similarity dimension of $\mathcal{F}$}, and we can characterize it
	as the root of the pressure function, which is
	\begin{equation}
		\label{R92}
		P_{\mathcal{F}}(t):=\lim\limits_{n\to\infty}
		\frac{1}{n}\log
		\sum_{\mathbf{ i}\in\mathcal{A}^n} |X_{\mathbf{i}}|^t=
		\lim\limits_{n\to\infty}
		\frac{1}{n}\log
		\sum_{\mathbf{ i}\in\mathcal{A}^n}
		\|f'_{\mathbf{i}}\|^t,
	\end{equation}
	where $\|\cdot\|$  is the supremum norm on $X$. The second equality follows from the Bounded Distortion Property \eqref{R59}. It is immediate from the definition that the pressure function $P_{\mathcal{F}}(\cdot)$ is  continuous, strictly decreasing, $P_{\mathcal{F}}(0)=\log m>0$ and
	$\lim_{t\to\infty}P_{\mathcal{F}}(t)=-\infty$, so it has a unique zero.   It is easy to see (see \cite[Chapter 14]{ourbook}) that
	\begin{equation}
		\label{R89}
		s^{\mathcal{F}}=P_{\mathcal{F}}^{-1}(0).
	\end{equation}
	Formula \eqref{R92} means that in a very loose sense, $\sum_{\mathbf{ i}\in\mathcal{A}^n} |X_{i_1\dots  i_n}|^t\approx \exp{(n P_{\mathcal{F}}(t))}$. Using this, if we accept that
	\eqref{R93} implies that
	$\Bigl\{\bigcup\limits_{\mathbf{i}\in \mathcal{A}^n} X_{\mathbf{i}}\Bigr\}_n$
	is "the most optimal covering system of $\Lambda $" in the sense specified above, then we get from the definition of the Hausdorff dimension that if the first cylinders do not intersect (i.e.  \eqref{R93} holds), then
	\begin{equation}
		\label{R91}
		\dim_{\rm H}  \Lambda
		=
		P_{\mathcal{F}}^{-1}(0)=s^{\mathcal{F}}.
	\end{equation}
	If we drop the assumption \eqref{R93}, then the { last formula does not necessarily} remain valid, but we always have the inequality
	\[ \Dh\Lam \leq s^\Fk. \]
	For example, if
	\begin{equation}
		\label{R88}
		\mathcal{F}=\left\{f_1(x)=x/3,f_2(x)=(x+1)/3,f_3(x)=x/3+1\right\},
	\end{equation}
	then
	\begin{equation}
		\label{R87}
		\dim_{\rm H}  \Lambda^{\mathcal{F}} \approx 0.876036<1=s^{\mathcal{F}},
	\end{equation}
	see \cite[Section 4.3]{ourbook}.
	This is due to the fact that there is an exact overlap:
	$f_1\circ f_3=f_2\circ f_1$. In general, we say there is an exact overlap for the IFS $\mathcal{F}\in \Theta $ if
	\begin{equation}
		\label{R49}
		\exists\, \mathbf{i}, \mathbf{j}\in \bigcup_n  \mathcal{A}^n,\
		\mathbf{i}\ne\mathbf{j},
		\quad
		\text{ such that }\quad
		f_{\mathbf{i}}|_{\Lambda }\equiv f_{\mathbf{j}}|_{\Lambda }.
	\end{equation}

	The so-called \texttt{Exact Overlap Conjecture} (see
	\cite[Conjecture 14.3.7]{ourbook} ) states:
	\begin{conjecture}[Exact Overlap Conjecture]\label{R86}
		{If $\dim_{\rm H}  \Lambda < \min\left\{1,s^{\mathcal{F}}\right\}$, then $\mathcal{F}$ has an exact overlap.}
	\end{conjecture}

	\subsection{Dimension of  invariant measures for a hyperbolic IFS. The non-overlapping case}
	The definition of the  Fourier transform, and definitions of various kinds of dimensions of measures and  connections between them can be found in the Appendix on page \pageref{R79}.
	Let $X$ be a compact interval
	and
	$\mathcal{F}=\left\{ f_i \right\}_{i=1}^{m}\in \Theta$.  Recall that we write $\mathcal{A}:=\left\{1,\dots  ,m\right\}$, and let $\mathcal{A}^*:=\bigcup \limits_{n=1}^{\infty   }\mathcal{A}^n$ be the set of finite words above the alphabet $\mathcal{A}$.
	We define a convenient metric (adapted to the IFS $\mF$) on the symbolic space $\Sigma: =\mathcal{A}^{\mathbb{N}}$ as follows: the distance between $\mathbf{i},\mathbf{j}\in\Sigma $ with $\mathbf{i} \neq \mathbf{j}$ is
	\begin{equation}
		\label{R78}
		{ d (\mathbf{i},\mathbf{j}) = d_{\Fk} (\mathbf{i},\mathbf{j}) }:=
		|f_{\mathbf{i}\wedge\mathbf{j}}(X)|,\ \ \ \text{ where }
		\mathbf{i}\wedge\mathbf{j} \text{ is the common prefix of }
		\mathbf{i} \text{ and } \mathbf{j}.
	\end{equation}
	The natural projection $\Pi :\Sigma \to\mathbb{R}$ is
	\begin{equation}
		\label{R64}
		\Pi (\mathbf{i}) = \Pi^\mF (\mathbf{i}) :=\lim\limits_{n\to\infty}f_{\mathbf{i}|_n}(x),
	\end{equation}
	where $x\in X$ is arbitrary and
	$\mathbf{i}|_n:=(i_1,\dots  ,i_n)$ for $\mathbf{i}=(i_1,i_2,\dots  )\in \Sigma $.
	We write $\sigma $  for the left shift on $\Sigma$, defined as  $\sigma (\mathbf{i}):=(i_2,i_3,i_4,\dots  )$.
	A Borel probability measure $\mu $ on $\Sigma $ is \texttt{invariant}
	if for every Borel set $H\subset \Sigma $ we have
	$
	\mu (H)=\mu (\sigma ^{-1}H).
	$
	An invariant probability measure $\mu $ on $\Sigma $  is \texttt{ergodic}  if
	$$
	\sigma ^{-1}(H)=H \Longrightarrow \text{ either }
	\mu (H)=0 \text{ or } \mu (H)=1.
	$$
	We write $\mathcal{E}_{\sigma }(\Sigma ) $ for the collection of ergodic shift invariant measures on $\Sigma $.
	Let $\mu\in\mathcal{E}_\sigma (\Sigma ) $. Then the entropy $  h_\mu =h_\mu (\sigma )$ is (roughly speaking) the   exponential growth rate of the measure of a typical $n$-cylinder. That is, let $\mathbf{i}|_n:=(i_1,\dots  ,i_n)$ for an $\mathbf{i}\in\Sigma $ and let the corresponding $n$-cylinder be
		\begin{equation}
			\label{R69}
			[\mathbf{i}|_n]:=\left\{ \mathbf{j}\in\Sigma :
			i_k=j_k,\ \forall k\leq n
			\right\}.
		\end{equation}
	Then for $\mu $-a.e. $\mathbf{i}\in\Sigma $, the following limit exists and equals to a constant, which is called \texttt{entropy} (see \cite[Corollary 9.5.4]{ourbook}):
	\begin{equation}
		\label{R11}
		-\lim\limits_{n\to\infty} \frac{1}{n}\log \mu ([\mathbf{i}|_n])=:h_\mu.
	\end{equation}
The Lyapunov exponent
of $\mu $ with respect to  (w.r.t.) the IFS $\mathcal{F}=\left\{ f_i \right\}_{i=1}^{ m}$ is the number $\chi _\mu  $ such that
	\begin{equation}
	\label{R53}
	\chi  _\mu(\Fk) :=-\int
	\log|f'_{i_1}(\Pi (\sigma \mathbf{i}))|d\mu (\mathbf{i})
	=
	-\lim\limits_{n\to\infty}
	\frac{1}{n}\log |f_{j_1\dots  j_n}(X)|\quad
	\text{ for $\mu $-a.e. } \mathbf{j}\in\Sigma
	.
\end{equation}
This follows by applying Birkhoff's Ergodic Theorem to $\mathbf{i} \mapsto \log |f'_{i_1}(\Pi(\sigma \mathbf{i}))|$, together with the Chain Rule and the Bounded Distortion Property (see \cite[Section 14.2.3]{ourbook}).

Roughly speaking,
\begin{equation}\label{eq:E}
	{\mu ([\mathbf{i}|_n])\approx \e{-n\cdot h_\mu   }}, \text{ for a $\mu $-typical $\mathbf{i}\in\Sigma $},
\end{equation}
and
\begin{equation}\label{eq:LE}
	|f_{i_1\dots  i_n}(X)|\approx \e{-n \chi_{\mu } } \quad
	\text{ for a $\mu $-typical $\mathbf{i}\in\Sigma $.}
\end{equation}
	
 We want to study push-forward measures
	$$\nu: = \Pi_* \mu, \text{ i.e. } \nu(H) = \mu(\Pi^{-1}(H)) \text{ for Borel } H \subset \R.$$
	Clearly, $\nu$ is supported on the attractor $\Lambda$; it often exhibits a complicated fractal structure. Let us begin with the study of the Hausdorff dimension of $\nu$.\\
	
	\textbf{Heuristics when \eqref{R93} holds}: Assume first that $  X_k \cap X_{\ell } = \emptyset$ for $k \neq \ell $.  Then for a   $\mu$-typical $\mathbf{i} \in \Sigma $
	and large $n$, using \eqref{eq:E}, \eqref{eq:LE}, and the fact that \eqref{R93} implies $\nu(f_{i_1\ldots i_n}(X)) = \mu([i_1, \ldots, i_n])$, we have
	\[  \Dh\nu  \approx  \frac{\log \nu(f_{i_1\ldots i_n}(X))}{\log |f_{i_1\ldots i_n}(X)|} =  \frac{\log \mu([i_1, \ldots, i_n])}{\log
		|f_{i_1\ldots i_n}(X)|}\approx \frac{\log e^{-nh_\mu}}{ \log e^{-n\chi_\mu}} =\frac{h_\mu}{\chi_\mu(\Fk)},\]
	where $\dim_{\rm H}  \nu $ is defined in \eqref{R83}. Note that if \eqref{R93} fails, then we only have $\nu(f_{i_1\ldots i_n}(X)) \geq \mu([i_1, \ldots, i_n])$, leading to a bound
	\begin{equation}\label{eq:hdim upper bound}\Dh\nu  \leq \frac{h_\mu}{\chi_\mu(\Fk)},
	\end{equation}
	valid for arbitrary $C^{1+\delta}$-smooth systems.
	\subsection{Gibbs measures}\label{sec:gibbs}
	Let $\phi :\Sigma \to \mathbb{R}$ be a continuous function (for the metric introduced in \eqref{R78}). Such functions will be called \texttt{potentials}. We define
	\begin{equation}
		\label{R71}
		\text{var}_k\phi :=
		\sup\left\{
		|\phi (\mathbf{i})-\phi (\mathbf{j})|:
		|\mathbf{i}\wedge\mathbf{j}| \geq k
		\right\}.
	\end{equation}
	We say that $\phi $ is  a \texttt{H\" older continuous potential}
	if there exist $\alpha \in(0,1)$
	and $b>0$ such that
	\begin{equation}
		\label{R72}
		\text{var}_k\phi \leq b\alpha ^k,\quad \text{for all }k >0.
	\end{equation}
	The set of H\" older continuous potentials   is denoted by  $\mathscr{H}$.
	For $\phi \in\mathscr{H}$,  $\mathbf{i }\in\Sigma$, and
	$\pmb{\omega}=(\omega _1,\dots  ,\omega _n)$ we write
	\begin{equation}
		\label{R70}
		S_n\phi (\mathbf{i}):=
		\sum_{\ell =0}^{n-1 }\phi (\sigma ^{\ell }\mathbf{i}),\qquad
		S_n\phi (\pmb{\omega}):=
		\sup\left\{
		S_n\phi (\mathbf{i}):
		\mathbf{i}\in[\pmb{\omega}]
		\right\}.
	\end{equation}
	Fix an arbitrary $\mathcal{F}\in\Theta $.
	The \texttt{pressure of the potential $\phi \in\mathscr{H}$}
	is defined by
	\begin{equation}
		\label{R68}
		P(\phi ):=\lim\limits_{n\to\infty}
		\frac{1}{n}\log \left(
		\sum_{\pmb{\omega}\in \mathcal{A}^n}
		\e{S_n\phi (\pmb{\omega})}
		\right).
	\end{equation}
	An invariant measure $\mu $ is called an \texttt{equilibrium state} for $\phi$ if
	\begin{equation}
		\label{R54}
		P(\phi )=
		h_\mu +\int\phi d\mu.
	\end{equation}
	For a H\" older potential $\phi \in\mathscr{H}$ the unique equilibrium state is the Gibbs measure (see \cite[Theorem 1.22]{Bow} ), whose existence is guaranteed by  the following theorem:
	\begin{thm}\label{R67}
		Let $\phi \in\mathscr{H}$. Then there exists a unique $\mu \in \mathcal{E}_\sigma (\Sigma )$, the \texttt{Gibbs measure for the potential $\phi $}, for which there exist constants 	
		$c_1,c_2>0$ such that for all $n $ and $\mathbf{i}\in\Sigma $  we have
		\begin{equation}
			\label{R66}
			c_1 \leq
			\frac{\mu ([\mathbf{i}|_{n }])}
			{\exp(
				-n P(\phi )+S_{n }\phi (\mathbf{i})
				)}
			\leq c_2.
		\end{equation}
	\end{thm}
	For the proof of this theorem see \cite[Chapter 1]{Bow}.
	Among all potentials the so-called geometric potential  will be the most important one for us. To define it first we introduce the potential  $\phi_{\mathcal{F}} ^s$ for all $s \geq  0$:
	\begin{equation}
		\label{R65}
		\phi_{\mathcal{F}} ^s(\mathbf{i}):=\log |f'_{i_1}(\Pi ( \sigma \mathbf{i}))|^s.
	\end{equation}
	Observe that by the Chain Rule and the Bounded Distortion Property,
	\begin{equation}
		\label{R63}
		\exp \left(S_n\phi_{\mathcal{F}}^s (\mathbf{i})\right)=|f'_{i_1\dots   i_n}(\Pi ( \sigma ^n\mathbf{i}))|^s\sim
		|X_{i_1\dots  i_n}|^s,
	\end{equation}
	where $a_n\sim b_n$ means that there exists a constant $C \in (0,\infty)$  such that $C^{-1} \leq \frac{a_n}{b_n} \leq C$ for all $n$. This implies that
	\begin{equation}
		\label{R26}
		P(\phi _{\mathcal{ F}}^{s })=P_{\mathcal{F}}(s),
	\end{equation}
	where the pressure function  $P_{\mathcal{F}}(\cdot)$ was defined in \eqref{R92}.
	Now we define the \texttt{geometric potential} as $\phi _{\mathcal{F}}:=\phi _{\mathcal{ F}}^{s_{\mathcal{F}} }$, where $s_{\mathcal{F}}$ was defined in	\eqref{R89} as the zero of $P_{\mathcal{F}}$, so that
	\begin{equation}
		\label{R56}
		P(\phi _{\mathcal{F}})=0,
	\end{equation}
	Let \texttt{$\mu _{\mathcal{F}}$ be the Gibbs measure for the geometric potential} $\phi _{\mathcal{F}}$.
	Putting together \eqref{R11}, \eqref{R53}, Theorem \ref{R67}, \eqref{R63}, and \eqref{R56}, we obtain:
	\begin{cor}\label{R55}
		For every  $\mathcal{F}\in \Theta (X)$
		there exist constants $c_4,c_5>0$ such that for all $\pmb{\omega}\in \mathcal{A}^*$ we have
		\begin{equation}
			\label{eq:geom pot s}
			s_{\mathcal{F}}=\frac{h_{\mu_{\mathcal{F}} }}{\chi_{\mu_{\mathcal{F}}}(\Fk) },
			\qquad\text{ and }\qquad
			c_4<
			\frac{\mu _{\mathcal{F}}([\pmb{\omega}])}{|X_{\pmb{\omega}}|^{s_{\mathcal{F}}}}
			<c_5.
		\end{equation}  	
	\end{cor}
	
	\begin{definition}\label{R25}
		The \texttt{natural measure} of the IFS $\mathcal{F}$ is defined by $\nu _{\mathcal{F}}:=\Pi _*(\mu _{\mathcal{F}}).$
	\end{definition}
	The upper bound \eqref{eq:hdim upper bound} gives
	\[	\dim_{\rm H} \nu _ {\mathcal{F}} \leq  s_{\mathcal{F}},\]
	with equality if \eqref{R93} holds. Therefore, we have the following consequence of \eqref{eq:geom pot s} (and inequality $\Dh \nu_{\mF} \leq \Dh \Lam_{\mF}$), which explains the significance of natural measures:
	\begin{equation}
		\label{R47}
		\dim_{\rm H}  \nu _{\mathcal{F}}=
		\min\left\{1,
		\frac{h_{\mu_\Fk}}{\chi_{\mu_\Fk}(\Fk)}
		\right\}
		\Longrightarrow
		\dim_{\rm H} \Lambda  _{\mathcal{F}}=\min\left\{
		1,s_{\mathcal{F}}
		\right\}.
	\end{equation}
	In order words, if there is no dimension drop for the natural measure, then there is no dimension drop for the attractor.
	\bigskip

	\subsection{The self-similar case}
	In this subsection we always assume that the IFS $\mathcal{F}$
	is self-similar, that is,
	\begin{equation}
		\label{R52}
		\mathcal{F}=\left\{ f_i(x)=r_ix+d_i \right\}_{i=1}^{m},
		\qquad
		r_i\in\left(-1,1\right)\setminus \left\{0\right\},\quad d_i\in\mathbb{R}.
	\end{equation}
	In this case $\phi _{\mathcal{ F}}^{s }(\mathbf{i})=\log |r_{i_1}|^s$ for any
	$\mathbf{i}\in\Sigma $. So,
	$P_{\mathcal{F}}(\phi _{\mathcal{ F}}^{s })=\log (|r_1|^s+\cdots + |r_m|^s)$. Thus, by \eqref{R94} we get that $s_{F}$
	is the solution of the so-called self-similarity equation:
	\begin{equation}
		\label{R51}
		\large{
			\underbrace{|r_{1}|^{s_{\mathcal{F}} }}_{p_1 }+\cdots+\underbrace{|r_{m}|^{s_{\mathcal{F}} }}_{p_m  }=1}.
	\end{equation}
	In this case the conformal similarity dimension $s_{\mathcal{F}}$ is simply called the \texttt{similarity dimension}.
	If we define the probability vector $\mathbf{p}=(p_1,\dots  ,p_m)$ by
	$p_i:=|r_i|^{s_{\mathcal{F}}}$,
	then
	$\mu _{\mathcal{F}}:=(p_1,\dots  ,p_m)^{\mathbb{N}}$ is the
	Gibbs measure for the geometric potential $\phi _{\mathcal{ F}}^{s_{\mathcal{F}} }$
	on $\Sigma =\left\{ 1,\dots  ,m \right\}^{\mathbb{N}}$. That is, for $\pmb{\omega }=(\omega _1,\dots  ,\omega _n)\in \Sigma _n$ the
	$\mu _{\mathcal{F}}$-measure of the cylinder
	$
	[\pmb{\omega }]:=\left\{ \mathbf{i}\in\Sigma :i_1=\omega _1,\dots  i_n=\omega _n \right\}
	$ is
	$
	\mu_{\mathcal{F}} ([\pmb{\omega }])=p_{\omega _1}\cdots p_{\omega _n}
	$. Then the natural measure for $\mathcal{F}$  is
	$\nu_{\mathcal{F}} =\Pi _*\mu_{\mathcal{F}} $. In this case
	\begin{equation}
		\label{R50}
		h_{\mu _{\mathcal{F}}}=
		-\sum_{k=1}^{m }p_i\log p_i=-s_{\mathcal{F}}
		\sum_{k=1}^{m }|r_i|^{s_{\mathcal{F}}}\log|r_i|,
		\quad
		\chi_{\mu _{\mathcal{F}}}=-\sum_{k=1}^{m }|r_i|^{s_{\mathcal{F}}}\log|r_i|.
	\end{equation}
	Hence,
	\begin{equation}
		\label{R48}
		s_{\mathcal{F}}=\frac{h_{\mu _{\mathcal{F}}}}{\chi_{\mu _{\mathcal{F}}}}.
	\end{equation}
	
	Self-similar measures are the measures on $\mathbb{R}$, which can be represented in the form
	$\nu_{\mathcal{F},\mathbf{p}} =(\Pi_{\mathcal{F}}) _*\mu_{\mathbf{p}} $ for a measure $\mu_{\mathbf{p}} =\mathbf{p}^{\mathbb{N}}$, where $\mathbf{p}=(p_1,\dots  ,p_m)$ is a probability vector ($p_i>0$, and $\sum_{k=1}^{m }p_i=1$). The similarity dimension of a self-similar measure  $\nu $ is $\dim_{\rm Sim}\nu_{\mathcal{F},\mathbf{p}} :=\frac{h_{\mu_{\mathbf{p}} }}{\chi_{{\mu_{\mathbf{p}} }}}  $.
	
	\subsubsection{The Exact Overlap Conjecture in the self-similar case}
	The Exact Overlap Conjecture is open even in the self-similar case. However, some breakthrough results have been obtained
	in the last decade, among which we mention two here. Suppose we are given a self-similar IFS of the form \eqref{R52}.  Let { $\Delta_n(\mathcal{F})$ be the minimum of
	$\Delta(\pmb{\omega},\pmb{\tau})$ for distinct
	$\pmb{\omega},\pmb{\tau}\in\Sigma_n$, where }
	\[
	\Delta(\pmb{\omega},\pmb{\tau})=\left\{\begin{array}{cc}
		\infty & \mbox{ if }  f_{\pmb{\omega}}'(0)\neq f_{\pmb{\tau}}'(0) \\
		\left|f_{\pmb{\omega}}(0)-f_{\pmb{\tau}}(0)\right| & \mbox{ if } f_{\pmb{\omega}}'(0)=f_{\pmb{\tau}}'(0).
	\end{array}\right.
	\]
	We say that the self-similar IFS $\mathcal{S}$ satisfies the
	\underline{\texttt{Exponential Separation Condition}} (ESC) if there exist $\varepsilon>0$ and a sequence  $n_k\uparrow \infty $ such that
	\begin{equation}\label{k78}
		\Delta_{n_k}>\varepsilon^{n_k}.
	\end{equation}
	Hochman \cite{Hoch}
	proved that the ESC holds for ``most''  self-similar IFS's.
	More precisely,
	a self-similar IFS of the form \eqref{R52} is determined by $2m$ parameters $(r_1, \ldots, r_m, d_1, \ldots, d_m)$. The set of those parameters for which the ESC does not hold form a subset of $\mathbb{R}^{2m}$ of packing dimension at most $2m-1$ (in particular: of Lebesgue measure zero).  Moreover,
	M. Hochman \cite{Hoch}
	proved the following { breakthrough} result:
	\begin{thm}[Hochman]\label{R44}
		Assume that the self-similar IFS $\mathcal{F}$ satisfies the ESC. Let $\mathbf{p}$ be a probability vector. Then
		\begin{equation}
			\label{R43}
			\dim_{\rm H}  \nu _{\mathcal{F},\mathbf{p}}=\min\left\{
			1,\dim_{\rm Sim} \nu _{\mathcal{F},\mathbf{p}}
			\right\}.
		\end{equation}
	\end{thm}
	Using this and \eqref{R47} we obtain that
	the set of parameters of those self-similar IFS's of the form \eqref{R52} for which
	$\dim_{\rm H}  \Lambda _{\mathcal{F}}\ne\min\left\{1,s_{\mathcal{F}}\right\}$
	has packing dimension at most $2m-1$.
	
	A recent result of Rapaport \cite{Rap} says that
	the Exact Overlap Conjecture holds when the contractions ratios
	$r_1,\dots  ,r_m$ are algebraic numbers.
	
	\subsection{Linear fractional IFS's}\label{sec:fracIFS} 

	Fix an $m \geq  2$ and a compact parameter interval $\mathfrak{I}\subset \mathbb{R}$. For every parameter $t\in \mathfrak{I}$
	and for every $i\in\mathcal{A}:=\left\{1,\dots  ,m\right\}$ we are given   $A_i^t=     \begin{bmatrix}
		a_i^t  & b_i^t  \\
		c_i^t  &  d_i^t
	\end{bmatrix}\in GL_2(\mathbb{R})$. For every $i\in \mathcal{A}$  the associated linear fractional mapping is $f_i^t:\mathbb{R} \cup \{ \infty \}\to \mathbb{R}\cup \{ \infty \}$, $f_i^t(x)=\frac{a_i^tx+b_i^t}{c_i^tx+d_i^t}$.
	We say that $(\mathcal{F}^t)_{t\in \mathfrak{I}}=(\left\{f_i^t\right\}_{i=1}^{m })_{t\in \mathfrak{I}}$ is a   family of linear fractional IFS
	if there exists an open bounded interval $V\subset \mathbb{R}$ and a $\gamma \in(0,1)$ such that $f_i^t(\overline{V})\subset V$ and $|(f _{i}^{t })'(x)|<\gamma $ for every $i\in\mathcal{A}$ and $t\in \mathfrak{I}$. Moreover we require that for every $i\in\mathcal{A}$ the function $\phi_i :\overline{V}\times \mathfrak{I}\to V$, $\phi _i(x,t):=f_i^t(x)$    is real analytic.
	
	
	\begin{definition}[Non-degenerate family of linear fractional IFS's ]
		Let $(\mathcal{F}^t)_{t\in \mathfrak{I}}=(\left\{f_i^t\right\}_{i=1}^{m })_{t\in \mathfrak{I}}$ be a family of linear fractional IFS, let $\Lambda ^t$ be the attractor for every $t\in \mathfrak{I}$ and let $\Sigma :=\mathcal{A}^{\mathbb{N}}$ be the symbolic space. We define the natural projection $\Pi ^t:\Sigma \to \Lambda ^t$  in the usual way $\Pi _i ^t(\mathbf{i}):=\lim\limits_{n\to\infty}f^t_{i_1\dots  i_n}(x_0)$, where $x_0\in V$ is arbitrary. We say that the family  $(\mathcal{F}^t)_{t\in \mathfrak{I}}$ is a non-degenerate
		family of linear fractional IFS's
		if
		\begin{equation}
			\label{P75}
			\mathbf{i},\mathbf{j}\in\Sigma ,\ \mathbf{i}\ne \mathbf{j},
			\implies
			\exists t_0\in \mathfrak{I},\ \Pi ^{t_0}(\mathbf{i})\ne \Pi ^{t_0}(\mathbf{j}).
		\end{equation}
	\end{definition}

	\begin{thm}[Solomyak, Takahashi \cite{SolTak}  ]
		Let $(\mathcal{F}^t)_{t\in \mathfrak{I}}$ be a non-degenerate
		family of linear fractional IFS's and let $\Lambda ^t$ be the attractor of $\mathcal{F}^t$. For every $t$ let $s_{t}$ be the root of the pressure function for the IFS $\mathcal{F}^t$.   Then for all but a set of zero Hausdorff dimension of $t\in \mathfrak{I}$ we have
		\begin{equation}
			\label{P74}
			\dim_{\rm H}  \Lambda^t =\min\left\{1,s_{t}\right\}.
		\end{equation}
	\end{thm}
	The proof is based on the adaptation of Hochman's method to the linear fractional IFS, see \cite{HS}.

	\subsection{Families of $\mathcal{C}^r$-smooth hyperbolic IFS's}
	Unfortunately, apart from the linear fractional case, results similar to Hochman's theorem do not exist for general $\mathcal{C}^r$-smooth hyperbolic IFS's. We need to confine our  attention to some families
	of $\mathcal{C}^r$-smooth hyperbolic IFS's which satisfy the so-called {\em transversality condition}. We will have assertions which claim that for a Lebesgue typical parameter the conformal similarity dimension gives the dimension of the attractor. More importantly, for all parameters $\lambda $  we consider a measure
	$\mu _\lambda $ on $\Sigma $ and their push-forward measures
	$\nu _\lambda :=(\Pi _\lambda )_*\mu _\lambda$. We study the absolute continuity and the dimension  of these measures $\nu _\lambda $  for Lebesgue typical parameter $\lambda $.
	\begin{definition}\label{R21}
		We say  that $\left\{\mathcal{F}^\lambda =\left\{f^\lambda_i \right\}_{i=1}^{m }\right\}_{\lambda \in \overline{U}}  $ is a \texttt{continuous family of $\mathcal{C}^r$-smooth hyperbolic IFS's
			on the compact interval $X$} if the parameter domain  $\overline{U}$  is the closure of the open set $U\subset \mathbb{R}^d$, there exist $0 < \gamma_1 < \gamma_2 < 1$ such that $\Fk^\lam \in \Theta_{\gamma_1, \gamma_2}$ for all $\lam \in \ov{U}$, there exists a bounded open interval
		$V\supset X$ such that for all  $i\in \mathcal{A}$ and $\lambda \in
		\overline{U}$  we have that $f _{i}^{\lam} : V \to V$ is a $\mathcal{C}^r$ diffeomorphism satisfying
		$f _{i}^{\lambda  }(X)\subset X$ and moreover, $\lam \mapsto \Fk^\lam$ is continuous in the $C^r$-topology (i.e. in the metric $\varrho_r$ as defined in \eqref{rho metric}).
	\end{definition}
	
	\begin{example}\label{R41}
		Let $\mathcal{F}=\left\{f_1,\dots  ,f_m\right\}$ be a $\mathcal{C}^r$-smooth
		hyperbolic IFS on the compact interval $X$. Using the notation of Definition
		\ref{R99} we assume that
		\begin{equation}
			\label{R16}
			|f'_i(x)|<\frac{1}{2},\quad \text{ holds for  all }i\in\mathcal{A},\
			x\in V.
		\end{equation}
		Let $\varepsilon >0$ be a sufficiently small number. Set $U:=(-\varepsilon ,\varepsilon )^m$ and  for a $\lambda=(\lambda_1,\dots  ,\lambda _m) \in \overline{U}$ let
		$$\mathcal{F}^\lambda=\left\{f _{1}^{\lambda}(x),\dots  ,f _{m}^{\lambda}(x)\right\} :=\left\{f_1(x)+\lambda_1,\dots   ,f_m(x)+\lambda _m\right\}.$$
		We say that \texttt{$\mathcal{F}^\lambda$ is a vertical translate of $\mathcal{F}$}.  It is clearly a continuous family of $C^r$-smooth hyperbolic IFS's.
	\end{example}
	\subsubsection{Principal Assumptions I}\label{R35}
	In the main part of this { survey} (relating to the problem of absolute continuity) we will assume:
	\begin{enumerate}[start=1,label={(MA\arabic*)}]
		\item\label{as:MA1} For every $j\in\mathcal{A}$ and $\lambda \in U$
		the second derivative (in $x$) of the map $f_j^\lam$ exists
		and is uniformly H\" older continuous in both $x$ and $\lambda $.
		\item The maps $\lam \mapsto f^\lam_j(x)$ are $C^{1+\delta}$-smooth on $U$ (uniformly w.r.t. $x$).
		\item For every $i,j$
		the second partial derivatives $\frac{d^2}{dxd\lam_i}f^\lam_j(x), \frac{d^2}{d\lam_i dx}f^\lam_j(x)$ are $\delta$-H\"older (uniformly, both in $\lam_i$ and $x$).
		\item\label{as:MA4} The system $\{f_j^\lam\}_{j \in \Ak}$ is {\em uniformly hyperbolic and contractive}: there exist $\gamma_1,\ \gamma_2 > 0$ such that
		\[
		0 < \gam_1 \le |(\textstyle{\frac{d}{dx}}f_j^\lam)(x)| \le \gam_2 < 1,\quad
		\forall x\in X,\  j\in \mathcal{A},\ \lambda\in \overline{U}.
		\]
	\end{enumerate}
	Note that some of the results (relating to the dimension) will hold under weaker regularity assumptions - see Sections \ref{sec:main} and \ref{sec:dim}.
	\begin{remark}\label{R15}
		Clearly, all the principal assumptions (MA1)-(MA4)  hold if we simply assume
		(MA4) and
		\begin{enumerate}[label={(MA123)}]
			\item\label{as:MA123} for every $j\in\mathcal{A}$, all of the third partial derivatives with respect to the $d+1$ variables of the map
			$(\lambda ,x) \to f _{j}^{\lambda  }(x)
			$  exist and are continuous.	
		\end{enumerate}	
	\end{remark}
	For the precise formulation of the assumptions (MA1)-(MA3)
	see Appendix \ref{R20}. In the special case
	when we have only one parameter (that is, $d=1$) the precise formulation can be found in
	\cite[Section 2]{BSSS}.

	For $\pmb{\omega}=(\omega _1,\dots \omega_n )\in \mathcal{A}^*$, $\mathbf{i}=(i_1,i_2,\dots  )\in \Sigma =\mathcal{A}^{\mathbb{N}}$  and $\lambda \in \overline{U}$  we write
	$$
	f_{\pmb{\omega}}^\lambda:=f _{\omega _1}^{\lambda  }
	\circ\cdots\circ
	f _{\omega _n}^{\lambda  },\quad
	\text{ and }
	\quad
	\Pi ^\lambda (\mathbf{i}):=\lim\limits_{n\to\infty}
	f _{i_1\dots  i_n}^{\lambda  }(x_0),
	$$
	where $x_0\in X$ is arbitrary.

	{
	
	\subsubsection{Transversality Condition}\label{R34}
	We only consider families $\left\{\mathcal{F}^\lambda \right\}_{\lambda \in \overline{U}}$ which (like the one in Example \ref{R41}) satisfy the \underline{\texttt{transversality condition}}:
	
	\begin{enumerate}[label={(MT)}]
		\item\label{as:trans mp} $\exists\,\eta>0:\ \forall\, \bi,\bj\in \Sig,\ \ i_1 \ne j_1, \ \left|\Pi^\lam(\bi) - \Pi^\lam(\bj)\right| < \eta \implies \left|\nabla(\Pi^\lam(\bi) - \Pi^\lam(\bj))\right| \ge \eta$,
	\end{enumerate}
	where
	$\nabla$ stands here for the gradient with respect to the parameter variable $\lam$ in $\R^d$.
	In the case $d=1$ this condition takes the form:
	
	\begin{equation}
		\label{R40}\tag{T}
		\exists\,\eta>0:\ \forall\, \lambda \in \overline{U},\
		\forall\  \mathbf{i},\mathbf{j}\in \Sigma ,\ \ i_1 \ne j_1, \ \left|\Pi^\lam(\mathbf{i}) - \Pi^\lam(\mathbf{j})\right| < \eta \implies \left|\textstyle{\frac{d}{d\lam}}(\Pi^\lam(\mathbf{i}) - \Pi^\lam(\mathbf{j}))\right| \ge \eta.
	\end{equation}
	
	It is easy to check that
	the transversality condition (MT) is equivalent to any of the following conditions (T2)-(T3):
	
	\begin{equation}
		\label{R39}\tag{T2}
		\exists\,\eta>0:\ \forall\, \lambda \in \overline{U},\
		\forall\  \mathbf{i},\mathbf{j}\in \Sigma ,\ \ i_1 \ne j_1, \ \Pi^\lam(\mathbf{i}) = \Pi^\lam(\mathbf{j}) \implies \left|\nabla(\Pi^\lam(\mathbf{i}) - \Pi^\lam(\mathbf{j}))\right| \ge \eta,
	\end{equation}
	and
	\begin{equation}
		\label{R38}\tag{T3}
		\exists\,C_T>0:\ \forall\, r>0,\
		\forall\  \mathbf{i},\mathbf{j}\in \Sigma ,\ \ i_1 \ne j_1, \
		\mathcal{L}^d
		\left\{\lambda \in \overline{U}:
		\left|\Pi^\lam(\mathbf{i}) - \Pi^\lam(\mathbf{j})\right| < r
		\right\}
		\leq  C_T\cdot r.
	\end{equation}
}

	\begin{figure}[H]\label{R37}
		\includegraphics[width=8.2cm]{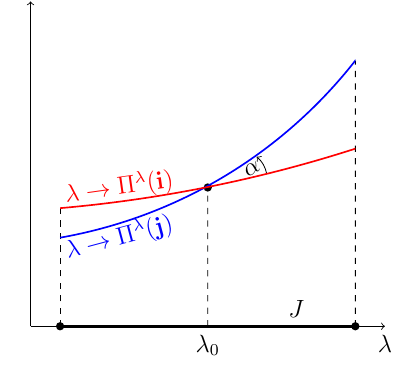}
		\includegraphics[width=8.2cm]{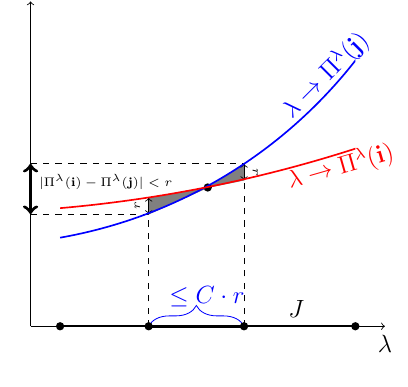}
		\caption{Conditions \eqref{R40} and \eqref{R38}. In the figure on the left, corresponding to \eqref{R40}, the angle $\alpha$ between the slopes is non-zero.}
	\end{figure}
	Transversality conditions \eqref{R40} and \eqref{R38} (for $d=1$) are visualized in Figure \ref{R37}.
	The following  was proved in \cite[Lemma 2.14]{BKRS}.
	\begin{lem}\label{R12}
		The family of vertical translates of
		a $\mathcal{C}^r$-smooth
		hyperbolic IFS on the compact interval from Example \ref{R41} is an {important} example
		where the Transversality Condition \ref{as:trans mp} holds. 	
	\end{lem}
	
	 We will also encounter transversality in another form, namely, the {\em transversality of degree (or order) $\beta$}, introduced by Peres and Schlag \cite{PS00},
	see Lemma~\ref{prop_beta trans} and Remark~\ref{rem:beta tran} below. In that condition there is no assumption $i_1\ne j_1$, which sometimes provides more flexibility.
	
	\medskip
	
	{The significance of the transversality condition lies in the fact that it allows one 
	to calculate the Hausdorff dimension for a \textit{typical} parameter $\lam$, for general overlapping non-linear systems. Below is a classical result in this direction, which considers projections of an arbitrary (but fixed!) ergodic measure on the symbolic space. See e.g. \cite[Chapter 14.4]{ourbook} for the proof.
		
		\begin{thm}\label{thm:dim ac fixed mu}
			Let $\{\mF^\lam\}_{\lam \in \ov{U}}$ be a smooth family of $C^r$-smooth hyperbolic IFS's on the compact interval with  $r>1$, satisfying the transversality condition \ref{as:trans mp}. Let $\mu$ be an ergodic shift-invariant Borel probability measure on $\Sig$ and set $\nu_\lam = (\Pi^\lam)_* \mu$. Then
			\begin{enumerate}[{\rm (1)}]
				\item\label{it: proj dim H} $\Dh \nu_\lam = \min \left\{ 1, \frac{h_{\mu}}{\chi_{\mu}(\mF^\lam)} \right\}$ for $\Lk^d$-a.e. $\lam \in U$,
				\item $\nu_\lam \ll \Lk^1$ for $\Lk^d$-a.e. $\lam \in U$ such that $\frac{h_{\mu}}{\chi_{\mu}(\mF^\lam)} > 1$.
				\end{enumerate}
Moreover, for any Borel probability measure $\mu$ on $\Sig$ the following hold (below $d_\lam$ is the metric on $\Sig$, corresponding to $\Fk^\lam$, defined in \eqref{R78}):
\begin{enumerate}[{\rm (3)}]
					\item\label{it:proj dim cor} $\dim_{cor} \nu_\lam = \min \left\{ 1, \dim_{cor} (\mu, d_\lam) \right\}$ for $\Lk^d$-a.e. $\lam \in U$,
					\item $\nu_\lam \ll \Lk^1$ with $\frac{d \nu_\lam}{d \Lk^1} \in L^2(\R)$ for $\Lk^d$-a.e. $\lam \in U$ such that $\dim_{cor} (\mu, d_\lam) > 1$.
				\end{enumerate}
		\end{thm}
		
		{Claim \ref{it:proj dim cor}  of the above theorem can be seen as a result on the correlation dimension preservation under (non-linear) projections $\Pi^\lam$. In fact, claim
		\ref{it: proj dim H} can be seen in the same way for the Hausdorff dimension, as $\frac{h_{\mu}}{\chi_{\mu}(\mF^\lam)} = \Dh(\mu, d_{\lam})$. See \cite{SolTransSurv} for a recent survey on transversality methods for IFS's, which elaborates on the connections between Theorem \ref{thm:dim ac fixed mu} and Marstrand-Mattila projection theorems for orthogonal projections.}
		
		\section{Parameter-dependent measures}
		
		Whilst powerful, Theorem \ref{thm:dim ac fixed mu} is not sufficient for all applications. This includes the scenario in which one allows the measure on the symbolic space to depend on the same parameter as the IFS, i.e. studying projections $\nu_\lam := (\Pi^\lam)_* \mu_\lam$. This requires a non-trivial extension of Theorem \ref{thm:dim ac fixed mu}, and this 
		is the main subject of this note, following \cite{BSSS}. Let us now elaborate on several situations to which this setting applies. We have already encountered the first one.
		
		\subsection{Natural measures}
		
		Recall that a natural measure for an IFS $\mF$ is $\nu_{\mF} := \Pi_* \mu_{\mF}$, where $\mu_{\mF}$ is the Gibbs measure corresponding to the geometric potential $\phi_\mF$ (see Definition \ref{R25}). Consider now a parametrized family $\mF^\lam$ and a corresponding family of natural measures $\nu_\lam := \nu_{\mF^\lam}$, which in this case are projections of Gibbs measures $\mu_\lam$ corresponding to the potentials
		\[ \phi_\lam := \phi_{\mF^\lam} = \log |(f^\lam_{i_1})'(\Pi^\lam(\sigma\mathbf{i} ))|^{s_{\mF^\lam}}. \]
		Clearly, $\phi_\lam$ depends on $\lam$ and so does $\mu_\lam$ (except for some very special cases, e.g. when every $\Fk^\lam$ is self-similar with $r_1 = \ldots = r_m$). Therefore, Theorem \ref{thm:dim ac fixed mu} cannot be directly applied in this setting (although it can be used to establish typical dimension and positive Lebesgue measure result for the attractor $\Lambda_\lam$ of $\mF^\lam$ for almost every $\lam \in U$ if transversality holds, see \cite[Theorem 14.4.1]{ourbook} and its proof).

	}

	\subsection{IFS's with place dependent probabilities}
	Let $\mathcal{F}=(f_1,\dots  ,f_m)\in\Theta $ and $\mathbf{p}=(p_1,\dots  ,p_m)$
	be a probability vector. Then there exists a unique measure
	$\nu_{\mathcal{F},\mathbf{p}} $  such that the support $\text{spt}(\nu )$ of $\nu $ satisfies  $\text{spt}(\nu )=\Lambda  $  and
	\begin{equation}
		\label{R33}
		\nu_{\mathcal{F},\mathbf{p}} =\sum_{i=1}^{m }p_i\cdot\nu_{\mathcal{F},\mathbf{p}}\circ f _{i}^{-1 }
	\end{equation}
	We call $\nu_{\mathcal{F},\mathbf{p}}$  the invariant measure corresponding to $\mathcal{F}$ and $\mathbf{p}$. We can write
	\eqref{R33} in an equivalent form:
	\begin{equation}
		\label{R32}
		\int \varphi\, d\nu_{\mathcal{F},\mathbf{p}} (x)
		=
		\sum_{i=1}^{m }
		\int
		p_i\cdot \varphi(f_i(x))\,d\nu_{\mathcal{F},\mathbf{p}} (x),\qquad \forall \varphi\in \mathcal{C}(X),
	\end{equation}
	where $\mathcal{C}(X)$ is the  set of
	continuous functions on the compact non-degenerate interval $X$.
	Place dependent invariant measures are obtained by replacing in \eqref{R32} the constant $p_i$ by  positive functions $p_i(x)$ which add up to $1$ everywhere. More precisely, for every $i\in\mathcal{A}$  let $p_i:X\to (0,1)$ be a
	H\" older continuous function which is bounded away from zero, so that
	 $\sum_{i=1}^{m }p_i(x)\equiv 1$. It was proved by Fan and Lau \cite{FL}
	that there exists a unique measure $\nu$, called the \texttt{place dependent stationary measure}, satisfying
	\begin{equation}
		\label{R31}
		\int \varphi \,d\nu (x)
		=
		\sum_{i=1}^{m }
		\int
		p_i(x)\cdot \varphi(f_i(x))\,d\nu (x),\qquad \mbox{for every}\ \varphi\in \mathcal{C}(X).
	\end{equation}
	Or equivalently,
	\begin{equation}
		\label{P81}
		\nu (B)=\sum_{i=1}^{m}\int\limits_{f _{i}^{-1 }(B)}p_i(x)d\nu (x),
		\quad \text{for every Borel set }B.
	\end{equation}
	B\'ar\'any \cite{BB15} proved that the measure $\nu $
	is actually a push-forward measure of a Gibbs measure. Namely, let
	$\varphi (\mathbf{i}):=\log p_{i_1}(\Pi (\sigma \mathbf{i}))$, where
	$\Pi :\Sigma \to \Lambda $ is the natural projection. Then $\varphi :\Sigma \to\mathbb{R}$ is a H\" older continuous potential. So, by Theorem \ref{R67} there exists a unique Gibbs measure $\mu $ on $\Sigma $ for the potential $\varphi $.     It was proved in \cite[Lemma 2.2]{BB15}
	that there exist constants $a,b>0$ such that for all $\mathbf{i}\in\Sigma$,
	\begin{equation}
		\label{R24}
		a<
		\frac{\mu \left([\mathbf{i}|_n]\right)}
		{\prod \limits_{k=1 }^{n }p_{i_k}(\Pi (\sigma ^{m+1}\mathbf{i}))}
		<b\quad
		\text{ and }\ \ 
		\nu =\Pi _*\mu .
	\end{equation}
	The equation defining the place dependent invariant measure $\nu $ can be
	described by the
	\texttt{Ruelle operator}  $T_{\mathcal{F}}:\mathcal{C}(X)\to  \mathcal{C}(X)$ defined by
	\begin{equation}
		\label{R23}
		(T_{\mathcal{F}}g)(x):=\sum_{i=1}^{m }p_i(x) g(f_i(x)).
	\end{equation}
	Then  $\nu $ is the fixed point of the adjoint operator $T_{\mathcal{F}}^*:\mathcal{C}(X)^*\to \mathcal{C}(X)^*$, that is,
	\begin{equation}
		\label{R22}
		T_{\mathcal{F}}^*\nu =\nu .
	\end{equation}
	In this case, the entropy and the Lyapunov exponent are
	\begin{equation}
		\label{R30}
		h_\nu =-\int \sum_{i=1}^{ m}p_i(x)\log p_i(x)\,d\nu (x)\quad\mbox{and}\quad
		\chi_\nu =
		-\int \sum_{i=1}^{m }p_i(x)\log|f'_i(x)|\,d\nu (x).
	\end{equation}


{We will  study overlapping cases in which the transversality condition holds, like the following ones.}

\subsubsection{Application: Place dependent Bernoulli convolutions}\label{R36}

Consider an IFS on the compact interval $X=\left[-1,1\right]$:
\begin{equation}
	\label{Q99}
	\Psi_{\lambda}=\left\{
	\psi_0^{\lambda}(x)=\lambda x-(1-\lambda),\
	\psi_1^{\lambda}(x)=\lambda x+(1-\lambda)\right\}
\end{equation}
with place dependent probabilities:
$$\left\{
p_0(x)=\frac{1}{2}+\rho x,\
p_1(x)=\frac{1}{2}-\rho x  \right\},\quad
x\in X.$$
The Ruelle operator $T$  acts on a continuous function $g\in \mathcal{C}(X)$ as follows:
\begin{equation*}
	Tg(x)=\Bigl( \half+\rho x  \Bigr)g\bigl(  \lambda x-(1-\lambda)  \bigr)+\Bigl( \frac{1}{2}-\rho x \Bigr)g\bigl( \lambda x+(1-\lambda)\bigr).
\end{equation*}
The fixed point $\nu _{\lambda ,\rho }$ of the Dual operator $T^*$ is  a place dependent Bernoulli convolution measure. Using \eqref{R30} we obtain
that the Lyapunov exponent and the entropy of this measure are
\begin{equation}
	\label{R18}
	\chi_{\nu _{\lambda ,\rho }}=-\log \lambda\quad
	\text{ and }\quad
	h_{\nu _{\lambda ,\rho }}
	=-
	\sum_{\varepsilon\in\left\{-1,1\right\}}
	\int_{\mathbb{R} }
	\left(\shalf+\varepsilon  \rho x\right)
	\log \left(\shalf+\varepsilon  \rho x\right)
	d \nu _{\lambda ,\rho }(x).
\end{equation}
If $0<\lambda <0.5$, then  the attractor of the IFS is a Cantor set of dimension less than one. So, we may assume that $0.5<\lambda <1$.
Shmerkin and Solomyak \cite{SS2} proved that for the parameter interval $$U:=\left(0.5,0.6684755\right)$$
the transversality condition holds. Using this, B\'ar\'any proved
\cite[Theorem 4.1]{BB15}
\begin{thm}\label{R17}
	Let $A:=\log2-\frac{2\rho ^2(1-\lambda )^2}{1+\lambda (
		4\rho (1-\lambda )	-\lambda )}$ and $B:=\frac{\rho ^2}{3(1-4\rho ^2)}$.
	Then
	$$
	\frac{A-B}{-\log\lambda }\leq
	\dim_{\rm H} \nu _{\lambda ,\rho }\leq
	\frac{A}{-\log\lambda  },\qquad
	\text{ for Lebesgue almost all } \qquad \lambda \in U.
	$$
	Moreover, $\nu _{\lambda ,\rho }$ is absolute continuous with respect to  the Lebesgue measure   for Lebesgue almost every $\lambda \in U$
	satisfying $\frac{A-B}{-\log\lambda }>1$.
\end{thm}
Using this theorem, based on the work of B\'ar\'any \cite{BB15},
it was obtained in \cite{BSSS}
that
{$\nu _{\lambda ,\rho }$ is absolutely continuous almost everywhere in the region marked ``abs.\ cont.'' in Figure \ref{R13}. In the region marked as ``singular'' the measure is singular everywhere; this was
shown in \cite{BB15} and follows from the fact the Hausdorff dimension of the measure is less than one.}
\begin{figure}[H]
	\includegraphics[width=80mm]{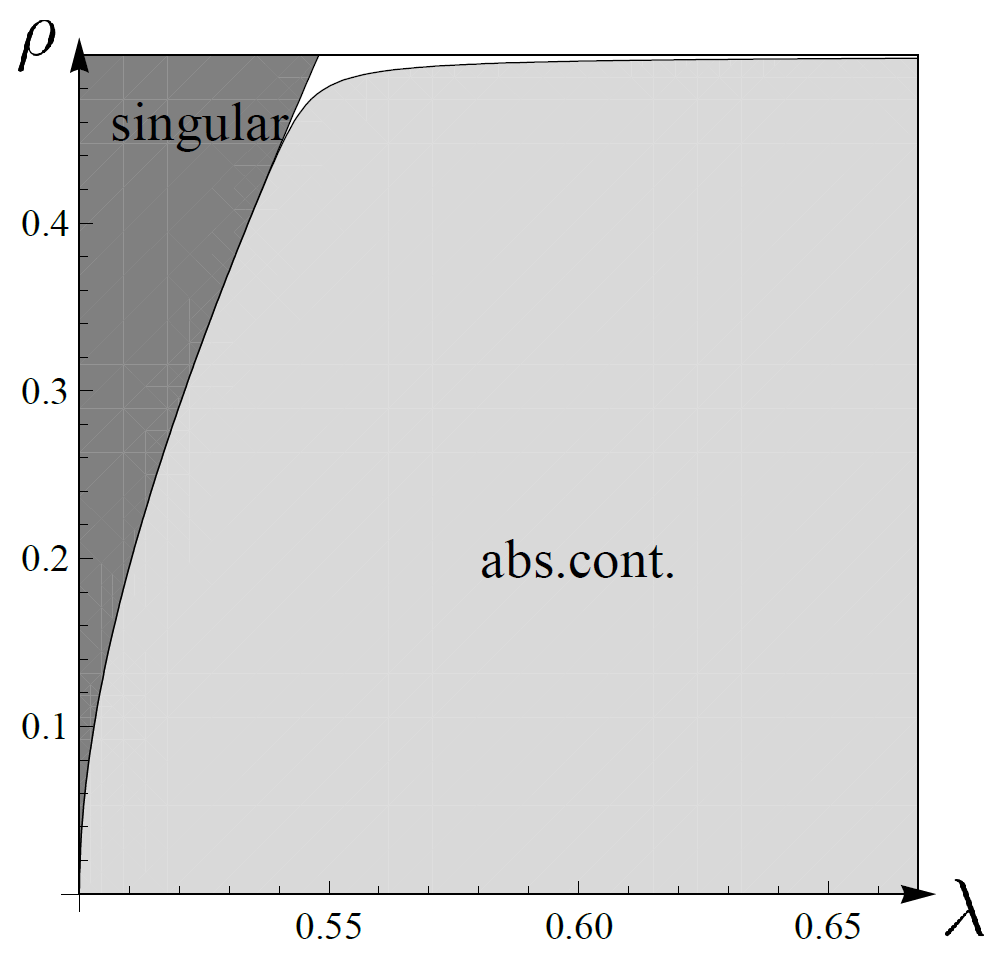}\\
	\caption{The absolute continuity and singularity regions of the measure $\nu _{\lambda ,\rho }$. }\label{R13}
\end{figure}
\subsubsection{Application: Slanted baker map}

\bigskip

Let $0<\rho <\frac{1}{2}$ and $1/2<\lambda <1$ and let us consider the following dynamical system $f_{\lambda ,\rho }:[-1,1]\times[0,1]\mapsto[-1,1]\times[0,1]$, where
\[
f_{\lambda ,\rho }(x,y)=\left\{\begin{array}{cc}
	\left(\lambda x-(1-\lambda),\frac{2y}{1+2\rho x}\right) & \text{if }0\leq y<\frac{1}{2}+\rho x \\
	\left(\lambda x+(1-\lambda),\frac{2y-2\rho x-1}{1-2\rho x}\right) & \text{if } \frac{1}{2}+\rho x\leq y\leq1.
\end{array}\right.
\]
For the action of $f_{\lambda ,\rho }$ on the rectangle $[-1,1]\times[0,1]$ see Figure~\ref{ffunction}.
\begin{figure}[H]
	\includegraphics[width=170mm]{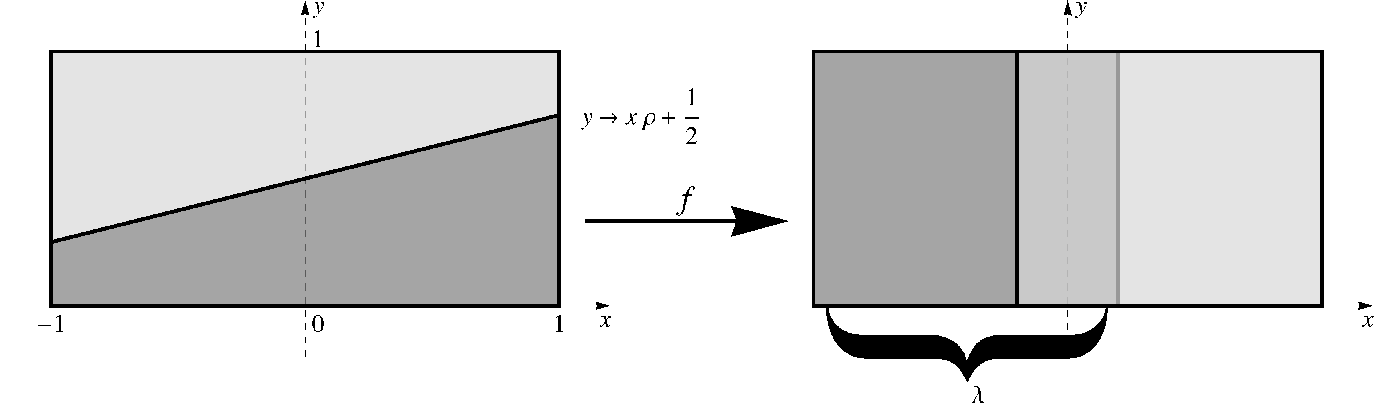}\\
	\caption{The map $f$ acting on the rectangle $[-1,1]\times[0,1]$.}\label{ffunction}
\end{figure}
It follows from \cite{pesin1992dimension} that there exists an $f_{\lambda ,\rho }$-invariant  measure $\mathfrak{m} _{\lambda ,\rho } $ called
Sinai-Bowen-Ruelle (SBR) measure  satisfying
$$
\frac{1}{n}\sum_{k=0}^{n-1 }\overline{\mathcal{L}}\circ f_{\lambda ,\rho }^{-k}\to \mathfrak{m} _{\lambda ,\rho } \quad
\text{ weakly},
$$
where $\overline{\mathcal{L}}$  is the normalized Lebesgue measure on the rectangle $[-1,1]\times [0,1]$.  Thus $\mathfrak{m} _{\lambda ,\rho }$ is absolutely continuous if and only if $\nu_{\lambda ,\rho }$ is absolutely continuous.

{\subsubsection{Furstenberg-like measures} \label{Furst-like}
	
	Another possible application is the family of Furstenberg measures, or more generally,  equilibrium measures induced by locally finite matrix cocycles. Similar measures were considered by B\'ar\'any and Rams \cite{BR18} to study the dimension of planar self-affine sets, and in particular, to provide a dimension maximizing measure. Here we consider a simple special case to demonstrate another direction of possible applications of our results.
	
	Let $\mathfrak{A}=(A_1,\ldots,A_m)$ be a tuple of $GL_2(\R)$ matrices with strictly positive entries. Similarly to Subsection~\ref{sec:fracIFS}, let us define a family of linear fractional maps induced by the matrices $A_i$. That is, let $f_i\colon[0,1]\to[0,1]$ be 
	$$
	f_i(x)=\frac{a_i x+b_i(1-x)}{(a_i+c_i)x+(b_i+d_i)(1-x)}\ \ \text{ for }\ A_i=\begin{pmatrix}
		a_i & b_i \\ c_i & d_i
	\end{pmatrix}.
	$$
	The maps $f_i$ are different from the maps defined in Subsection~\ref{sec:fracIFS}, but they have similar properties. We chose a different representation here, because it fits better with
	 the case of matrices having positive entries, since such matrices preserve the positive quadrant.
	
	It is easy to see that $\|f_i'\|=\frac{|\det(A_i)|}{\langle A_i\rangle^2}$, where $\langle A_i\rangle=\min\{a_i+c_i,b_i+d_i\}$. Let $X_\mathfrak{A}$ be the attractor and let $\Pi_\mathfrak{A}\colon\Sigma\to[0,1]$ be the natural projection of the IFS $\Phi_{\mathfrak{A}}=\{f_i\}_{i\in\mathcal{A}}$ as above. Let $\underline{v}(x)=\begin{pmatrix}
		x \\ 1-x
	\end{pmatrix}$. Then $\frac{A_i\underline{v}(x)}{\|A_i\underline{v}(x)\|_1}=\underline{v}(f_i(x))$, where $\|\cdot\|_1$ is the $1$-norm on $\R^2$.
	For every $q\in\R$, the following limit {exists:}
	$$
	P_{\mathfrak{A}}(q)=\lim_{n\to\infty}\frac{1}{n}\log\sum_{\pmb{\omega}\in\mathcal{A}^n}\|A_{\pmb{\omega}}\|^q.
	$$
	Furthermore, there exists a unique ergodic shift invariant probability measure $\mu_q$ such that for some $C>0$ and every $\pmb{\omega}\in\mathcal{A}^*$,
	$$
	C^{-1}\leq\frac{\mu_{\mathfrak{A},q}([\pmb{\omega}])}{e^{-|\pmb{\omega}|P_\mathfrak{A}(q)}\|A_{\pmb{\omega}}\|^q}\leq C,
	$$
	see Feng~\cite{Feng}. In particular, $\mu_{\mathfrak{A},q}$ is the Gibbs measure, defined in Subsection~\ref{sec:gibbs}, with respect to the potential $\bi\mapsto q\log\|A_{i_1}\underline{v}(\Pi_{\mathfrak{A}}(\sigma\bi))\|_1$.
	
	By {Oseledets'} multiplicative ergodic theorem, there exist reals $\eta_2< \eta_1$, such that
	$$
	\eta_i(\mathfrak{A})=\lim_{n\to\infty}\frac{1}{n}\log\alpha_i(A_{\bi|_n})\text{ for $\mu_{\mathfrak{A},q}$-almost every $\bi\in\Sigma$},
	$$
	which we call the Lyapunov exponents of the matrix cocycle $\mathfrak{A}$. Here $\alpha_i(A)$ denotes the $i$th singular value of $A$. Using \eqref{R54}, we get
	$$
	h_{\mu_{\mathfrak{A},q}}=P_{\mathfrak{A}}(q)-q\eta_1(\mathfrak{A})\text{ and }\chi_\mathfrak{A}=\eta_1(\mathfrak{A})-\eta_2(\mathfrak{A}),
	$$
	where $\chi_\mathfrak{A}$ denotes the Lyapunov exponent of the IFS $\Phi_{\mathfrak{A}}$.
	
	\begin{thm}
		Let $U=\left\{\mathfrak{A}=(A_1,\ldots,A_m)\in GL_2(\R_+)^m:|\det(A_i)|<\frac12\langle A_i\rangle^2\right\}$. Then for every $q\in\R$ the measures $\nu_{\mathfrak{A},q}=(\Pi_\mathfrak{A})_*\mu_{\mathfrak{A},q}$ satisfy
		\begin{enumerate}[{\rm (1) }]
			\item\label{it:app1} $\Dh \nu_{\mathfrak{A},q} = \min \left\{ \frac{P_{\mathfrak{A}}(q)-q\eta_1(\mathfrak{A})}{\eta_1(\mathfrak{A})-\eta_2(\mathfrak{A})}, 1 \right\}$ for Lebesgue-a.e. $\mathfrak{A} \in U$,
			\item\label{it:app2} $\nu_{\mathfrak{A},q} \ll \Lk^1$ for Lebesgue-a.e. $\mathfrak{A} \in U$, such that $P_{\mathfrak{A}}(q)>(q+1)\eta_1(\mathfrak{A})-\eta_2(\mathfrak{A})$.
		\end{enumerate}
	\end{thm}
	
	\begin{proof}
		The strategy of the proof is to decompose $U$ into measurable subsets on which the IFS $\Phi_\mathfrak{A}$ satisfies the transversality condition under some natural parametrization, and then to apply Theorem~\ref{thm:main_gibbs}.
		
		Let us define an equivalence relation on $GL_2(\R_+)$ as
		$$
		\begin{pmatrix}
			a & b \\ c & d
		\end{pmatrix}\sim\begin{pmatrix}
			a' & b' \\ c' & d'
		\end{pmatrix}\text{ if there exists $t\in\R$ such that }\begin{pmatrix}
			a+t(a+c) & b+t(b+d) \\ c-t(a+c) & d-t(b+d)
		\end{pmatrix}=\begin{pmatrix}
			a' & b' \\ c' & d'
		\end{pmatrix}.
		$$
		Simple algebraic manipulations show that if $A\sim A'$ then $\det(A)=\det(A')$ and $\langle A\rangle=\langle A'\rangle$. Hence, the relation can be naturally extended to $U$ by $(A_1,\ldots,A_m)\sim(A_1',\ldots,A_m')$ if $A_i\sim A_i'$ for every $i=1,\ldots,m$. Moreover, if $\mathfrak{A}\sim\mathfrak{A}'$ then if $\Phi_\mathfrak{A}=\{f_i(x)\}_{i\in\mathcal{A}}$ then there exists $(t_1,\ldots,t_m)\in\R^m$ such that $\Phi_{\mathfrak{A}'}=\{f_i(x)+t_i\}_{i\in\mathcal{A}}$. Hence, by Theorem~\ref{thm:main_gibbs}, 
		in view of the claim in Example~\ref{R41},
		for every fixed $\mathfrak{A}\in U$ and for almost every $\mathfrak{A}'\sim\mathfrak{A}$, \ref{it:app1} and \ref{it:app2} hold. The claim of the theorem then follows by Fubini's theorem.
	\end{proof}
}


{ \section{Main results: projections of parameter-dependent measures}\label{sec:main}

In this section we state extensions of Theorem \ref{thm:dim ac fixed mu} to the case of projections of parameter-dependent measures $\nu_\lam = (\Pi^\lam)_* \mu_\lam$, which  cover cases considered in the previous section. The next result is an extension of \cite[Theorems 3.1 and 3.3]{BSSS} to the multiparameter case.

\begin{thm}\label{thm:main_gibbs}
	Let $\{f_j^\lam\}_{j \in \Ak}$ be a parametrized IFS satisfying smoothness assumptions \ref{as:MA1} - \ref{as:MA4} and the transversality condition \ref{as:trans mp} on $U$. Let $\left\{ \mu_\lam \right\}_{\lam \in \ov{U}}$ be a family of Gibbs measures on $\Sig$ corresponding to a family of uniformly H\"older potentials $\phi^\lambda\colon\Sig\mapsto\R$ (i.e. there exists $0<\alpha<1$ and $b>0$ with $\sup_{\lambda\in \ov{U}}\mathrm{var}_k(\phi^\lambda)\leq b\alpha^k$), such that $\lam \mapsto \phi^\lambda$ is H\"older continuous in the supremum norm. Then the
	measures $\nu_\lam = (\Pi^\lam)_* \mu_\lam$ satisfy
	\begin{enumerate}[{\rm (1)}]
		\item $\Dh \nu_\lam = \min \left\{ \frac{h_{\mu_{\lam}}}{\chi_{\mu_\lam}(\mF^\lam)}, 1 \right\}$ for $\Lk^d$-a.e. $\lam \in U$,
		\item $\nu_\lam \ll \Lk^1$ for $\Lk^d$-a.e. $\lam \in U$ such that $\frac{h_{\mu_\lam}}{\chi_{\mu_\lam}(\mF^\lam)} > 1$.
	\end{enumerate}
\end{thm}

Recall that instead of conditions \ref{as:MA1} - \ref{as:MA4}, it suffices to assume \ref{as:MA123} and \ref{as:MA4}, see Remark \ref{R15}.

Actually, the above theorem is a consequence of results valid for general measures (not necessarily Gibbs), satisfying certain regularity properties of the dependence $\lam \mapsto \mu_\lam$. Let us introduce now those conditions and state more general versions of the theorem.

\begin{defn} Let $\left\{ \mu_\lam \right\}_{\lam \in \ov{U}}$ be a collection of finite Borel measures on $\Sig$. We define the following continuity conditions for $\mu_{\lam}$:

\medskip

\begin{enumerate}[label={(M0)}]
	\item\label{as:measure_cont} for every $\lambda_0$ and every $\eps>0$ there exist $C,\xi>0$ such that
	\[
	C^{-1} e^{-\eps|\omega|} \mu_{\lambda_0}([\omega]) \leq\mu_{\lambda}([\omega])\leq C  e^{\eps|\omega|} \mu_{\lambda_0}([\omega])
	\]
	holds for every $\om \in \Sig^*,\ |\om| \geq 1$ and $\lam \in \ov{U}$ with $|\lam - \lam_0|<\xi$;
\end{enumerate}

\medskip

\begin{enumerate}[label={(M)}]
	\item\label{as:measure} there exists $c>0$ and  $\theta \in (0,1]$ such that for all $\omega\in\Sig^*,\ |\om|\geq 1$, and all $\lambda,\lambda' \in \ov{U}$,
	\[
	e^{-c|\lambda-\lambda'|^\theta|\omega|} \mu_{\lambda'}([\omega]) \leq\mu_{\lambda}([\omega])\leq e^{c|\lambda-\lambda'|^\theta|\omega|} \mu_{\lambda'}([\omega]).
	\]
\end{enumerate}
\end{defn}
Note that \ref{as:measure} implies \ref{as:measure_cont}. The following is the dimension part of the result. It is a multiparameter version of \cite[Theorem 3.1]{BSSS}. Recall that $d_\lam$ is the metric on $\Sig$, corresponding to $\Fk^\lam$, defined in \eqref{R78}.

\begin{thm}\label{thm:main_hausdorff}
	Let $\{f_j^\lam\}_{j \in \Ak}$ be a {continuous family of $C^{1 + \delta}$-smooth hyperbolic IFS's on the compact interval X}, satisfying \ref{as:MA4} and the transversality condition  \ref{as:trans mp} on $U$. Let $\left\{ \mu_\lam \right\}_{\lam \in \ov{U}}$ be a collection of finite ergodic shift-invariant Borel measures on $\Sig$ satisfying \ref{as:measure_cont}. Then
	\begin{enumerate}[{\rm (1)}]
	\item\label{it: hdim par meas} $\Dh(\nu_\lam) = \min\left\{1, \frac{h_{\mu_\lam}}{\chi_{\mu_\lam}(\Fk^\lam)}\right\}$ for $\Lk^d$-a.e. $\lam \in U$.
	\end{enumerate}
	Moreover, the following holds for any family of finite Borel measures on $\Sig$ satisfying \ref{as:measure_cont}:
	\begin{enumerate}[{\rm (2)}]
	\item\label{it:dim cor par meas} $\dim_{cor}(\nu_\lam) = \min\left\{1, \dim_{cor}(\mu_\lam, d_\lam)\right\}$ for $\Lk^d$-a.e. $\lam \in U$.
	\end{enumerate}
\end{thm}

{Note that we assume weaker regularity assumptions on the IFS than in \cite[Theorem 3.1]{BSSS} ($C^{1 + \delta}$ instead of $C^{2 + \delta}$). Moreover, we do not assume continuity of $\lam \mapsto h_{\mu_\lam}$ and $\lam \mapsto \chi_\mu(\Fk^\lam)$.}

The most general version of the absolute continuity result is as follows. See Section \ref{R79} for the definitions of correlation and Sobolev dimensions, denoted $\dim_{cor}$ and $\dim_S$ respectively.

\begin{thm}\label{thm:main_cor_dim}	Let $\{f_j^\lam\}_{j \in \Ak}$ be a parametrized IFS satisfying smoothness assumptions \ref{as:MA1} - \ref{as:MA4} and the transversality condition \ref{as:trans mp} on $U$. Let $\left\{ \mu_\lam \right\}_{\lam \in \ov{U}}$ be a collection of finite Borel measures on $\Sig$ satisfying \ref{as:measure}. Then
	\begin{equation} \label{eq:main_cor_dim} \dim_S(\nu_\lam) \geq \min \left\{\dim_{cor}(\mu_\lam, d_\lam), 1 + \min\{\delta, \theta \} \right\}  \end{equation}
	for Lebesgue almost every $\lam \in U$, where $d_\lam$ is the metric on $\Sig$ defined in \eqref{R78} corresponding to $\Fk^\lam$ {and $\delta, \theta$ are from assumptions \ref{as:MA1} - \ref{as:MA4} and \ref{as:measure} respectively.}
\end{thm}

{
\begin{cor}\label{cor:ac}
	Let $\{f_j^\lam\}_{j \in \Ak}$ be a parametrized IFS and let $\left\{ \mu_\lam \right\}_{\lam \in \ov{U}}$ be a collection of finite Borel measures on $\Sig$ satisfying the assumptions of Theorem~\ref{thm:main_cor_dim}. Then $(\Pi_\lambda)_*\nu_\lam$ is absolutely continuous with a density in $L^2$ for Lebesgue almost every $\lam$ in the set $\{ \lam \in U : \dim_{cor}(\mu_\lam, d_\lam) > 1 \}$.
\end{cor}

The corollary follows by Theorem~\ref{thm:main_cor_dim} and the properties of the Sobolev dimension, see Lemma~\ref{R81}.
}

\section{Dimension - on the proof of Theorem \ref{thm:main_hausdorff}}\label{sec:dim}

In this section we present the proof of Theorem \ref{thm:main_hausdorff}. We shall not give all the details (these can be found in \cite{BSSS}), but discuss the main ideas. {In addition, as Theorem \ref{thm:main_hausdorff} is formally stronger than the corresponding result in \cite{BSSS}, we shall give a precise description of (minor) changes which one has to make in the proof of \cite[Theorem 3.1]{BSSS} in order to obtain a full proof of Theorem \ref{thm:main_hausdorff}.} Let us begin with a discussion of the main ideas.

Proofs of transversality-based results for dimension, like Theorem \ref{thm:dim ac fixed mu} or Theorem  \ref{thm:main_hausdorff}, usually consist of combining two main ingredients - one proves the version of the result for the correlation dimension (this is where transversality is used), and then extends it to the Hausdorff dimension by restricting measure on the symbolic space to suitable Egorov sets for convergences in \eqref{R11} and \eqref{R53}. Putting technicalities aside, the main observation behind the proof of Theorem \ref{thm:main_hausdorff} is noting that the correlation dimension behaves continuously with respect to $\mu_\lam$ whenever condition \ref{as:measure_cont} holds. The following exposition does not lead to a rigorous proof of Theorem \ref{thm:main_hausdorff}, but presents this main idea.

\begin{prop}\label{prop:dim cor cont}
Under assumptions of item \ref{it:dim cor par meas} of Theorem \ref{thm:main_hausdorff}, the following families of maps are equicontinuous:
\begin{enumerate}[{\rm (1)}]
\item\label{it:sym dim cor cont} $\left\{ \ov{U} \ni \lam \mapsto \dim_{cor}(\mu_\lam, d_\Fk) : \mF \in \{\mF^\lam\}_{\lam \in \ov{U}} \right\}$,
\item\label{it2} $\ov{U} \ni \lam \mapsto \dim_{cor}(\mu_\lam, d_\lam)$ (as a family consisting of a single map),
\item\label{it:dim cor cont} $\left\{ \ov{U} \ni \lam \mapsto \dim_{cor}((\Pi^{\mF})_* \mu_{\lam}) : \mF \in \{\mF^\lam\}_{\lam \in \ov{U}} \right\}$.
\end{enumerate}
\end{prop}

\begin{rem} Note that $\dim_{cor}((\Pi^{\mF})_* \mu_{\lam})$ is \textit{not} jointly continuous in $(\lam, \mF)$. Actually, it is not true that for a fixed ergodic measure, the map $\lam \mapsto \dim_{cor}((\Pi^{\mF^\lam})_* \mu)$ is continuous in the overlapping case { (see e.g. \cite[Section 1.6.1]{ourbook} and further discussions)}. The crucial point of item \ref{it:dim cor cont} of Proposition \ref{prop:dim cor cont} is that $\dim_{cor}((\Pi^{\mF})_* \mu_{\lam})$ is continuous when one fixes IFS $\mF$ and varies measure $\mu_\lam$ in the symbolic space. On the other hand, this issue does not appear if we consider measures in the symbolic space, hence in item \ref{it:sym dim cor cont} we actually do have a joint continuity of $(\lam, \Fk) \mapsto \dim_{cor}(\mu_\lam, d_\Fk)$ (with $\Fk$ considered with the $C^{1+\delta}$ topology), implying also continuity in item \ref{it2}.
\end{rem}

\begin{proof}[Proof of Proposition \ref{prop:dim cor cont}]
We only consider the family from item \ref{it:dim cor cont} --- the proof for \ref{it:sym dim cor cont}  and \ref{it2} is very similar. It is enough to prove the following: for every $\lam_0$ and $\eps>0$ there exists a neighbourhood $V$ of $\lam_0$ and a constants $C_1, C_2 > 0$ such that for every $\mF \in \{\mF^\lam\}_{\lam \in \ov{U}}$, setting $\Pi = \Pi^\mF$ we have
\begin{equation}\label{eq:energy comparison} C_1^{-1} \mE_{\alpha - \eps}(\Pi_*\mu_{\lam_0}) - C_2 \leq \mE_\alpha(\Pi_*\mu_\lam) \leq C_1 \mE_{\alpha+ \eps}(\Pi_*\mu_{\lam_0}) + C_2
\end{equation}
for $\lam \in V$ and $\alpha \in (0,1)$. For simplicity we assume $\diam(X) \leq 1$. First, note that for any $\om,\tau \in \Sig$ the following holds:
\begin{equation}\label{eq:energy sum}
C_\alpha^{-1}\sum \limits_{n=0}^{\infty}e^{\alpha n} \mathds{1}_{ \left\{ \left| \Pi(\om) - \Pi(\tau) \right| \leq e^{-n} \right\}} \leq \left| \Pi(\om) - \Pi(\tau) \right|^{-\alpha} \leq \sum \limits_{n=0}^{\infty}e^{\alpha(n+1)} \mathds{1}_{ \left\{ \left| \Pi(\om) - \Pi(\tau) \right| \leq e^{-n} \right\}}
\end{equation}
for a constant $C_\alpha > 0$. Indeed, if $\Pi^\lam(\om) = \Pi^\lam(\tau)$, then both the left- and right-hand sides are divergent. Otherwise, there exists $n \geq 0$ such that $e^{-(n+1)}< \left| \Pi(\om) - \Pi(\tau) \right| \leq e^{-n}$ and the upper bound in \eqref{eq:energy sum} follows immediately. The lower bound follows by noting additionally that there exists a constant $C_\alpha$ such that $\sum \limits_{j=0}^n e^{j\alpha} \leq C_\alpha e^{\alpha n}$ holds for all $n$. As by \ref{as:MA4}, the family $\{\mF^\lam\}_{\lam \in \ov{U}}$ consists of uniformly contracting systems, there exists $q \in \N$ such that
\begin{equation}\label{eq:qn length bound}
	\left| \Pi(\om) - \Pi(\tau) \right| \leq e^{-n}/2\ \text{ if } \om|_{qn} = \tau|_{qn}.
\end{equation}
Let $V$  be a neighbourhood of $\lam_0$ such that inequalities
\begin{equation}\label{eq:measure ineq} C^{-1} e^{-\eps|\omega|} \mu_{\lambda_0}([\omega]) \leq\mu_{\lambda}([\omega])\leq C  e^{\eps|\omega|} \mu_{\lambda_0}([\omega])
\end{equation}
hold for $\om \in \Sig^*$ and $\lam \in V$ (it exists by \ref{as:measure_cont}). Then by \eqref{eq:energy sum} and \eqref{eq:qn length bound} (below $1^\infty$ denotes an infinite word constantly equal to $1 \in \Ak$):
\[
\begin{aligned} \mE_\alpha(\Pi_*\mu_\lam) & = \int\limits_{\Sig} \int\limits_{\Sig} |\Pi(\om) - \Pi(\tau)|^{-\alpha}d\mu_\lam(\om) d\mu_\lam(\tau)  \leq e^\alpha \sum \limits_{n=0}^\infty e^{\alpha n} \mu_\lam \otimes \mu_\lam(\{|\Pi(\om) - \Pi(\tau)| \leq e^{-n}\}) \\
	& \leq  e^\alpha \sum \limits_{n=0}^\infty e^{\alpha n} \mu_\lam \otimes \mu_\lam( \{|\Pi(\om|_{qn}1^\infty) - \Pi(\tau|_{qn}1^\infty)| \leq 2e^{-n}\})\\
	& \leq  C e^\alpha \sum \limits_{n=0}^\infty e^{(\alpha+2q\eps) n} \mu_{\lam_0} \otimes \mu_{\lam_0}({\{|\Pi(\om|_{qn}1^\infty) - \Pi(\tau|_{qn}1^\infty)| \leq 2e^{-n}\}}) \\
	& \leq  C e^\alpha \sum \limits_{n=0}^\infty e^{(\alpha+2q\eps) n} \mu_{\lam_0} \otimes \mu_{\lam_0}({\{|\Pi(\om) - \Pi(\tau)| \leq 3e^{-n}\}})\\
	& \leq  C e^{\alpha+2(\alpha + 2q\eps)} \sum \limits_{n=0}^\infty e^{(\alpha+2q\eps) n} \mu_{\lam_0} \otimes \mu_{\lam_0}({\{|\Pi(\om) - \Pi(\tau)| \leq e^{-n}\}}) + Ce^{\alpha}(1 + e^{\alpha + 2q\eps})\\
	& \leq C_1 \mE_{\alpha+ 2q\eps}(\Pi_*\mu_{\lam_0}) + C_2,
\end{aligned}
\]
where in the $3$rd line we have applied \eqref{eq:measure ineq} to the set $\{|\Pi(\om|_{qn}1^\infty) - \Pi(\tau|_{qn}1^\infty)| \leq 2e^{-n}\}$ (which is a union of products of cylinder sets of length $qn)$, while the $5$th line uses $3e^{-n} \leq e^{-(n-2)}$. This proves the upper bound in \eqref{eq:energy comparison}. The lower bound follows in the same manner (as we have a matching lower bound in \eqref{eq:measure ineq}).
\end{proof}

With the aid of Proposition \ref{prop:dim cor cont} it is easy to deduce item \ref{it:dim cor par meas} of Theorem \ref{thm:main_hausdorff} directly from item \ref{it:proj dim cor} of Theorem \ref{thm:dim ac fixed mu}:

\begin{proof}[Proof of item \ref{it:dim cor par meas} of Theorem  \ref{thm:main_hausdorff}]
Fix $\eps >0$ and consider a countable cover $\{V_i\}_{ i \in \N}$ of $\ov{U}$ by open sets such that if $\lam, \lam_0 \in V_i$, then
\[\left| \dim_{cor}(\mu_\lam, d_\Fk) - \dim_{cor}(\mu_{\lam_0}, d_{\Fk}) \right| < \eps  \text{ for every } \Fk \in \left\{ \Fk^\lam : \lam \in \ov{U}  \right\}\]
and
\[ \left| \dim_{cor}(\Pi_*\mu_\lam) - \dim_{cor}(\Pi_*\mu_{\lam_0}) \right| < \eps \text{ for every } \Pi \in \left\{ \Pi^{\Fk^\lam} : \lam \in \ov{U}  \right\}.\]
Fix $\lam_0 \in V_i$. The above inequalities give for every $\lam \in V_i$:
\begin{equation}\label{eq:dim cor cont ineq}
\begin{aligned} \big| \dim_{cor}((\Pi^\lam)_* \mu_{\lam})   - & \min\{ 1, \dim_{cor}(\mu_{\lam}, d_{\lam}) \} \big| \leq \\
	& \big| \dim_{cor}((\Pi^\lam)_* \mu_{\lam_0}) - \min\{ 1, \dim_{cor}(\mu_{\lam_0}, d_{\lam}) \} \big| + 2\eps.
\end{aligned}
\end{equation}
By the item \ref{it:proj dim cor} of Theorem \ref{thm:dim ac fixed mu} applied to the measure $\mu = \mu_{\lam_0}$ we have
\[ \dim_{cor}((\Pi^\lam)_* \mu_{\lam_0}) = \min\{ 1, \dim_{cor}(\mu_{\lam_0}, d_{\lam}) \} \text{ for } \Lk^d\text{-a.e. } \lam \in V_i, \]
hence by \eqref{eq:dim cor cont ineq} and as $\{V_i\}_{ i \in \N}$ is a countable cover of $\ov{U}$,
\[ \big| \dim_{cor}((\Pi^\lam)_* \mu_{\lam})   - \min\{ 1, \dim_{cor}(\mu_{\lam}, d_{\lam}) \} \big| \leq 2\eps \text{ for } \Lk^d\text{-a.e. } \lam \in \ov{U}. \]
As $\eps>0$ is arbitrary, taking a countable intersection over $\eps \searrow 0$ finishes the proof.
\end{proof}

Let us now discuss how one obtains the Hausdorff dimension part of Theorem \ref{thm:main_hausdorff}. The first main ingredient is the inequality $\Dh(\mu) \geq \dim_{cor}(\mu)$, which holds for arbitrary measures. Unfortunately, in general we have $\frac{h_{\mu_\lam}}{\chi_{\mu_\lam}(\Fk^\lam)} \geq \dim_{cor}(\mu, d_\lam)$ for shift-invariant measures, and the inequality is often strict, so it is not enough to invoke item \ref{it:dim cor par meas} of Theorem \ref{thm:main_hausdorff}. In the case of Theorem \ref{thm:dim ac fixed mu}, with a fixed measure $\mu$ in the symbolic space, one restricts $\mu$ to the set
\[
\begin{aligned} A = \Big\{ \om \in \Sig : C^{-1} e^{-n(h(\mu) + \eps)} & \leq \mu([\om|_n])  \leq C e^{-n(h(\mu) - \eps)} \text{ and } \\
& C^{-1} {e^{-n(\chi_\mu(\Fk^{\lam_0})+\eps)}} \leq |f^{\lam_0}_{\om|_n}(X)| \leq C e^{-n\chi_\mu(\Fk^{\lam_0} - \eps)} \Big\}
\end{aligned}\]
(note that by the Egorov theorem, for every $\eps > 0$ we have that $\mu(A) \to 1$ as $C \to \infty$). A simple calculation shows that $\dim_{cor}(\mu|_A, d_\lam) \geq \frac{h(\mu) - \eps}{\chi_\mu(\mF^{\lam_0}) + \eps}$ and one can apply the already established item \ref{it:proj dim cor} of Theorem \ref{thm:dim ac fixed mu} to $\mu|_A$, obtaining the Hausdorff dimension part of the result by letting $\eps \to 0$ (and using continuity of $\lam \mapsto \chi_\mu(\mF^{\lam})$). The same strategy essentially works for Theorem \ref{thm:main_hausdorff} with parameter dependent measures $\mu_\lam$. {The difficulty is that now we have to consider parameter-dependent Egorov sets $A_\lam$ and study the
measures $\mu_\lam|_{A_\lam}$. In order to apply item \ref{it:dim cor par meas} of Theorem  \ref{thm:main_hausdorff} directly, we would have to choose $A_\lam$ in a fashion which guarantees that the family $\lam \mapsto \mu_\lam|_{A_\lam}$ satisfies condition \ref{as:measure_cont}. A convenient alternative solution is to fix a small neighbourhood $V$ of parameters and consider a common Egorov set $A$ for all $\lam \in V$. Then one can combine the standard transversality argument with an adaptation of the method from the proof of Proposition \ref{prop:dim cor cont} for the family $\mu_{\lam}|_A$. This leads to the following proposition, which is an adaptation of \cite[Proposition 5.1]{BSSS} to our case.}
}

{

\begin{prop}\label{prop:hdim local}
Under the assumptions of item \ref{it: hdim par meas} of Theorem \ref{thm:main_hausdorff}, there exists a number L>0 (depending only on the family $\{\Fk^\lam : \lam \in \ov{U} \}$) with the following property. Fix $\eps>0$. For every $\lam_0 \in \ov{U}$ there exists an open neighbourhood $U'$ of $\lam_0$ such that
\begin{equation}\label{eq:hdim lower bound}
\Dh(\nu_\lam) \geq \min \left\{ 1, \frac{h_{\mu_{\lam_0}}}{\chi_{\mu_{\lam_0}(\Fk^{\lam_0})}} \right\} - L\eps
\end{equation}
holds for $\Lk^d$-a.e. $\lam \in \bigl\{\lambda\in U': |h_{\mu_\lam} - h_{\mu_{\lam_0}}| < \eps \text{ and } |\chi_{\mu_\lam}(\Fk^\lam) - \chi_{\mu_{\lam_0}}(\Fk^{\lam_0})| < \eps\bigr\}$.
\end{prop}

We will not give a full proof of this proposition. Instead, we will present first a sketch explaining the main idea and then give a discussion of the precise changes one has to make in the proof of \cite[Proposition 5.1]{BSSS} in order to obtain a rigorous proof of Proposition \ref{prop:hdim local}.

{For the sketch of the method, fix $\lam_0 \in \ov{U}$ and a small $\eps>0$. For $D>0$ consider the set
\[
\begin{split} A_D := \Bigg\{  \om \in \Sig : \underset{n \geq 1}{\forall}\  D^{-1} e^{-n(h_{\mu_{\lam_0}} + 2\eps)} & \leq  \mu_{\lam_0}([\om|_n]) \leq D e^{-n(h_{\mu_{\lam_0}} - 2\eps)} \text{ and } \\
	D^{-1} e^{-n(\chi_{\mu_{\lam_0}}(\Fk^{\lam_0}) + 2\eps)} &\leq  \left| \left(f^{\lam_0}_{\om|_n}\right)' (\Pi^{\lam_0} (\sigma^n\om)) \right| \leq D e^{-n(\chi_{\mu_{\lam_0}}(\Fk^{\lam_0}) - 2\eps)}  \Bigg\}.
\end{split}
\]
By the Egorov theorem applied to the convergences in \eqref{R11} and \eqref{R53} for the measure $\mu_{\lam_0}$, we have $\lim \limits_{D \to \infty} \mu_{\lam_0}(A_D) = 1$.  Let $U'$ be a neighbourhood of $\lam_0$ such that for $\lam \in U'$ one has (by \ref{as:measure_cont}):
\begin{equation}\label{eq: measure ineq}
C^{-1} e^{-\eps|\omega|} \mu_{\lambda_0}([\omega]) \leq\mu_{\lambda}([\omega])\leq C  e^{\eps|\omega|} \mu_{\lambda_0}([\omega]),
\end{equation}
and (by the Bounded Distortion Property) for every $\lam \in U'$ and $\om \in \Sig,\ x,y \in X, n \geq 1$,
\begin{equation}\label{eq: bdp egorov} C^{-1} e^{-\eps n} \left| \left(f^{\lam}_{\om|_n}\right)' (y) \right| \leq  \left| \left(f^{\lam_0}_{\om|_n}\right)' (x) \right|  \leq C e^{\eps n} \left| \left(f^{\lam}_{\om|_n}\right)' (y) \right|
\end{equation}
for some constant $C$ (depending on $U'$). Applying Egorov theorem once more to \eqref{R11} and \eqref{R53} for $\mu_\lam$ and combining it with the above inequalities, we see that for every $\lam \in U'$ such that $|h_{\mu_\lam} - h_{\mu_{\lam_0}}| < \eps$ and $|\chi_{\mu_\lam}(\Fk^\lam) - h_{\mu_{\lam_0}}(\Fk^{\lam_0})| < \eps$, we have $\mu_\lam(A_D) > 0$ provided that $D$ is large enough (depending on $\lam$). As in order to obtain the statement of Proposition \ref{prop:hdim local} we are allowed to take countable intersections over $\lam$, in what follows we can fix a large $D>0$ and consider only $\lam \in U'_D := \{ \lam \in U' : \mu_\lam(A_D) > 0 \}$. Set $\tilde{\mu}_\lam := \mu_\lam |_{A_D}$ and $A_n = \{ u \in \Ak^n : [u] \cap A_D \neq \emptyset \}$. As $\Dh (\Pi^\lam)_* \mu_\lam \geq \Dh (\Pi^\lam)_* \tilde{\mu}_\lam \geq \dim_{cor} (\Pi^\lam)_* \tilde{\mu}_\lam$, it suffices to prove that
\begin{equation}\label{eq:dim int goal} \Ik = \int \limits_{U'_D} \int \limits_{\Sigma}  \int \limits_{\Sigma} |\Pi^\lam(\om) - \Pi^\lam(\tau)|^{-\alpha}d\tilde{\mu}_\lam(\om)d\tilde{\mu}_\lam(\tau)d\lam < \infty\ \text{ for } \alpha >  \min \left\{ 1, \frac{h_{\mu_{\lam_0}}}{\chi_{\mu_{\lam_0}(\Fk^{\lam_0})}} \right\} - L\eps.
\end{equation}
Fix $s > 0$. Splitting the double integral over $\Sigma \times \Sigma$ into cylinders corresponding to longest common prefixes and applying the definition of $A_D$ together with \eqref{eq: bdp egorov} one obtains
\[\begin{split} \mathcal{I} &= \int \limits_{U'_D} \sum \limits_{n=0}^\infty \sum \limits_{u \in A_n} \sum_{\substack{a,b\in \Ak\\ a \neq b}}\ \iint\limits_{[ua] \times [ub]} \left| f^\lam_u \left(\Pi^\lam(\sigma^n\om)\right) - f^\lam_u \left(\Pi^\lam(\sigma^n\tau)\right) \right|^{-\alpha}\,d\tilde{\mu}_\lam(\om)\,d\tilde{\mu}_\lam(\tau)\,d\lam \\
	& \leq C^\alpha D^\alpha \int \limits_{U_{\eps'}} \sum \limits_{n=0}^\infty e^{n\alpha(\chi_{\mu_{\lam_0}}(\Fk^{\lam_0}) + 3\eps)} \sum \limits_{u \in A_n} \sum_{\substack{a,b\in \Ak\\ a \neq b}}\ \iint\limits_{[ua] \times [ub]} \left| \Pi^\lam(\sigma^n\om) - \Pi^\lam(\sigma^n\tau) \right|^{-\alpha}\,d\tilde{\mu}_\lam(\om)\,d\tilde{\mu}_\lam(\tau)\,d\lam.
\end{split}
\]
Using \eqref{eq:energy sum} and applying the same argument as in the proof of Proposition \ref{prop:dim cor cont}, we obtain from \eqref{eq: measure ineq} (recall that $q \in \N$ is chosen so that \eqref{eq:qn length bound} holds) for $u \in A_n$:
\[\begin{split} \iint\limits_{[ua] \times [ub]} &  \left| \Pi^\lam(\sigma^n\om) - \Pi^\lam(\sigma^n\tau) \right|^{-\alpha}\,d\tilde{\mu}_\lam(\om)\,d\tilde{\mu}_\lam(\tau) \\
	&  \leq \sum \limits_{j=0}^{\infty}e^{\alpha j}   \iint\limits_{[ua] \times [ub]} \mathds{1}_{ \left\{ \left| \Pi(\sigma^n\om) - \Pi(\sigma^n\tau) \right| \leq e^{-j} \right\}} \,d\tilde{\mu}_\lam(\om)\,d\tilde{\mu}_\lam(\tau) \\
	& \leq \sum \limits_{j=0}^{\infty}e^{\alpha j}  \tilde{\mu}_\lam \otimes \tilde{\mu}_\lam \left( [ua] \times [ub] \cap \left\{ \left| \Pi((\sigma^n\om)|_{qj}1^\infty) - \Pi((\sigma^n\tau)|_{qj}1^\infty) \right| \leq 2e^{-j} \right\} \right) \\
	& \leq C \sum \limits_{j=0}^{\infty}e^{(\alpha + 2\eps q) j + \eps n}  \mu_{\lam_0} \otimes \mu_{\lam_0} \left( [ua] \times [ub] \cap \left\{ \left| \Pi((\sigma^n\om)|_{qj}1^\infty) - \Pi((\sigma^n\tau)|_{qj}1^\infty) \right| \leq 2e^{-j} \right\} \right) \\
	& \leq\iint\limits_{[ua] \times [ub]}   \left( C_1 e^{\eps n}  \left|\Pi^\lam(\sigma^n\om) - \Pi^\lam(\sigma^n\tau) \right|^{-(\alpha+ 2q\eps)} + C_2 \right) \,d\mu_{\lam_0}(\om)\,d\mu_{\lam_0}(\tau) + C_2.
\end{split}
\]
By \cite[Lemma 14.4.4]{ourbook}, the transversality condition guarantees that for every $(\om, \tau) \in [ua]\times[ub]$, we have
\[ \int \limits_{U'_D} \left| \Pi^\lam(\sigma^n\om) - \Pi^\lam(\sigma^n\tau) \right|^{-(\alpha+ 2q\eps)} d\lam \leq C_{\alpha+ 2q\eps} < \infty\ \text{ if } \alpha+ 2q\eps < 1. \]
Consequently, combing all the calculations so far and applying Fubini's theorem and the definition of $A_D$:
\[\begin{split}
	\Ik & \leq C^\alpha D^\alpha C_{\alpha + 2q\eps} \sum \limits_{n=0}^\infty e^{n\alpha(\chi_{\mu_{\lam_0}}(\Fk^{\lam_0}) + 3\eps)} \sum \limits_{u \in A_n} \sum_{\substack{a,b\in \Ak\\ a \neq b}} \mu_{\lam_0}([ua])\mu_{\lam_0}([ub]) \left(C_1 e^{\eps n} + C_2 \right) \\
	& \leq C^\alpha D^\alpha C_{\alpha + 2q\eps} \sum \limits_{n=0}^\infty e^{n\alpha(\chi_{\mu_{\lam_0}}(\Fk^{\lam_0}) + 3\eps)} \sum \limits_{u \in A_n} \mu_{\lam_0}([u])^2 \left(C_1 e^{\eps n} + C_2 \right) \\
	& \leq C^\alpha D^{(\alpha+1)} C_{\alpha + 2q\eps} \sum \limits_{n=0}^\infty e^{n(\alpha(\chi_{\mu_{\lam_0}}(\Fk^{\lam_0}) + 3\eps) - (h_{\mu_{\lam_0}} - 2\eps ))} \left(C_1 e^{\eps n} + C_2 \right).
\end{split}\]
The last sum is finite for every $\alpha > 0$ satisfying
\[ \alpha < \frac{h_{\mu_{\lam_0}} - 3\eps}{\chi_{\mu_{\lam_0}}(\Fk^{\lam_0}) + 3\eps}\ \text{ and }\ \alpha + 2q\eps < 1.\]
This establishes \eqref{eq:dim int goal} and finishes the sketch of the proof of Proposition \ref{prop:hdim local}.
}

\begin{proof}[Proof of Proposition \ref{prop:hdim local}]
We explain the changes one has to make in the proof of \cite[Proposition 5.1]{BSSS} in order to obtain the result in our case. First note the differences between the assumptions with respect to \cite[Proposition 5.1]{BSSS}: we assume weaker regularity conditions on the IFS and the multiparameter trasnversality condition \ref{as:trans mp}. Moreover, unlike in \cite{BSSS}, we do not assume continuity of the maps $\lam \mapsto h_{\mu_{\lam}}$ and $\lam \mapsto \chi_{\mu_{\lam}}(\Fk^\lam)$. An inspection of the proof in \cite{BSSS} shows the following.
\begin{itemize}
	\item The transversality condition with $d=1$ is used in the proof only via inequality \eqref{R38}. {In the case of a multiparameter family ($d>2$), the assumed condition \ref{as:trans mp} implies an analogous inequality
	\[ \Lk^d \left( \left\{ \lam \in \ov{U} : |\pi^\lam(\mathbf{i}) - \pi^\lam(\mathbf{j})| \leq r  \right\} \right) \leq C_T r\]
	for all $r>0$, $\mathbf{i}, \mathbf{j} \in \Sig$ such that $\mathbf{i}_1 \neq \mathbf{j}_1$, with a constant $C_T$ depending only on the system.} Therefore, the switch from $d=1$ to $d>1$ does not require any changes in the application of the transversality condition.
	\item Uniform hyperbolicity and contraction condition $(A4)$ in \cite{BSSS} is the same as our condition \ref{as:MA4}. The other conditions $(A1) - (A3)$ from \cite{BSSS} are used in the proof only via the parametric Bounded Distortion Property \cite[Lemma 4.2]{BSSS}. Weaker assumptions of Theorem \ref{thm:main_hausdorff} guarantee its weaker form \cite[Lemma 14.2.4.(ii)]{ourbook}, which is sufficient for the needs of the proof of \cite[Proposition 5.1]{BSSS}. Indeed, it is used only via inequalities (needed for $A_\lam \subset A$ and inequality $(5.4)$ in \cite{BSSS})
	\[ Ce^{-\eps|\mathbf{i}|} \leq \left| \frac{f^{\lam_0}_{\mathbf{i}}(x)}{f^\lam_{\mathbf{i}}(y)} \right| \leq Ce^{\eps|\mathbf{i}|}, \]
	which have to hold on an open neighbourhood $U'$ of $\lam_0$ (depending on $\eps$). The version of the Bounded Distortion Property from \cite[Lemma 14.2.4.(ii)]{ourbook} suffices for that purpose (while the stronger statement of \cite[Lemma 4.2]{BSSS} gives an explicit bound on how small $U'$ has to be, not needed here).
	\item Continuity of  $\lam \mapsto h_{\mu_\lam}$ and $\lam \mapsto \chi_\mu(\Fk^\lam)$ is used only to choose a neighbourhood $U'$ of $\lam_0$ in such a way that inequalities $|h_{\mu_\lam} - h_{\mu_{\lam_0}}| < \eps$ and $|\chi_{\mu_\lam}(\Fk^\lam) - h_{\mu_{\lam_0}}(\Fk^{\lam_0})| < \eps$ hold for $\lam \in U'$. In our case, we simply assume that in the statement of the proposition.
\end{itemize}
Applying the above changes to the proof of \cite[Proposition 5.1]{BSSS}, one can repeat the proof and conclude as on p.\ 17 of \cite{BSSS} that there exists an open neighbourhood $U'$ of $\lam_0$ such that
\[ \Dh\nu_\lam \geq \min \left\{ 1 - Q'\eps, \frac{h_{\mu_{\lam_0}} - 4\eps}{\chi_{\mu_{\lam_0}(\Fk^{\lam_0})} + 3\eps} \right\}, \]
for $\Lk^d$-a.e. $\lam \in U'$ satisfying $|h_{\mu_\lam} - h_{\mu_{\lam_0}}| < \eps$ and $|\chi_{\mu_\lam}(\Fk^\lam) - h_{\mu_{\lam_0}}(\Fk^{\lam_0})| < \eps$, where $Q'$ is an explicit constant depending on $\gamma_2$ in \ref{as:MA4}. This finishes the proof.
\end{proof}

Now we can finish the proof of Theorem \ref{thm:main_hausdorff}. It is convenient to do so in a slightly different manner than in \cite{BSSS}.

\begin{proof}[Proof of item \ref{it: hdim par meas} of Theorem \ref{thm:main_hausdorff}]
As the inequality $\Dh\nu_\lam \leq \min\left\{1, \frac{h_{\mu_\lam}}{\chi_{\mu_\lam}(\Fk^\lam)} \right\}$ holds for every $\lam$ (see e.g. \cite[Theorem 14.2.3]{ourbook}), it suffices to prove  the opposite inequality for almost every $\lam \in \ov{U}$. Fix $\eps>0$ and consider a cover $\{ V_{i,j} \}_{i,j=1}^\infty$ of $\ov{U}$ by sets of the form
\[V_{i,j}= \{ \lam \in \ov{U} :  i\eps \leq|h_{\mu_\lam}| \leq (i+1)\eps,\ j\eps \leq|\chi_{\mu_\lam}(\Fk^\lam)| \leq (j+1)\eps \}.\]
Furthermore, consider an open cover $\{U'(\lam_0) : \lam_0 \in \ov{U}\}$ of $\ov{U}$, where $U'(\lam_0)$ is the neighbourhood of $\lam_0$ from Proposition \ref{prop:hdim local} corresponding to $\eps$. Finally, choose a countable subcover $\{U_k(\lam_k)\}_{k=1}^\infty$ of $\{U'(\lam_0) : \lam_0 \in \ov{U}\}$ and set
\[\Vk = \{ V_{i,j} \cap U_k(\lam_k) : 1 \leq  i,j,k  <\infty \}.\]
Fix $V:= V_{i,j} \cap U_k(\lam_k) \in \Vk$. By Proposition \ref{prop:hdim local}, the inequality
\[ \Dh\nu_\lam \geq \min \left\{ 1, \frac{h_{\mu_{\lam_k}}}{\chi_{\mu_{\lam_k}(\Fk^{\lam_k})}} \right\} - L\eps \]
holds for $\Lk^d$-a.e. $\lam \in V$. As $\lam, \lam_k \in V_{i,j}$ and $\Vk$ is a countable cover of $\ov{U}$, this implies also
\[ \Dh\nu_\lam \geq \min \left\{ 1, \frac{h_{\mu_{\lam}} - \eps}{\chi_{\mu_{\lam}(\Fk^{\lam})} + \eps} \right\} - L\eps \]
for $\Lk^d$-a.e. $\lam \in \ov{U}$. As $\eps>0$ is arbitrary, this finishes the proof.
\end{proof}

}

\section{Absolute continuity - on the proof of Theorem \ref{thm:main_cor_dim}}

{At first we consider the 1-parameter case, as in \cite{BSSS}, and let $U\subset \R$ be a bounded open interval.
The multiparameter case is then deduced  by ``slicing'' the $d$-dimensional set of parameters; this is done at the end of the section.}

\subsection{First approach}
Let us explain now why the approach from the proof of Theorem \ref{thm:main_hausdorff} does not work for absolute continuity and hence the proof of Theorem \ref{thm:main_cor_dim} requires stronger assumptions (most notably: condition \ref{as:measure} instead of \ref{as:measure_cont}). The standard approach to proving typical absolute continuity of $(\Pi^\lam)_* \mu_\lam$ would be to use the following characterization of absolute continuity for a finite measure $\nu$ on $\R$ (see \cite[Theorem 2.12]{Mattila}):
\[ \nu \ll \Lk^1\ \text{ if and only if }\ \underline{D}(\nu, x) := \liminf \limits_{r \to 0} \frac{\nu(B(x, r))}{2r} < \infty \text{ for } \nu\text{-a.e. } x \in \R. \]
Therefore, in order to prove that $(\Pi^\lam)_* \mu_\lam \ll \Lk^1$ for almost every $\lam \in U$, it suffices to show
\begin{equation}\label{eq:ac int} \int \limits_{U} \int \limits_{\R} \underline{D}(\Pi^\lam_*\mu_\lam, x)\,d\Pi^\lam_*\mu_\lam(x)\, d\lam < \infty.
\end{equation}
By Fatou's lemma one has
\begin{equation}\label{eq:fatou} \int \limits_{U} \int \limits_{\R} \underline{D}(\Pi^\lam_*\mu, x)\,d\Pi^\lam_*\mu_\lam(x)\, d\lam \leq \liminf \limits_{ r \to 0} \frac{1}{2r} \int \limits_{U} \int \limits_{\R} \Pi^\lam_*\mu_\lam(B(x,r))\,d\Pi^\lam_*\mu_\lam(x)\,d\lam.
\end{equation}
If $\mu_\lam \equiv \mu$ for some fixed measure $\mu$, the classical approach is to use the transversality condition \ref{as:trans mp} and Fubini's theorem in order to show that if $\dim_{cor}(\mu, d_\lam) > 1$ on $U$, then  (see e.g. \cite[Theorem 6.6.2.(iv)]{ourbook} and its proof)
\begin{equation}\label{eq:tranvers int bound}\int \limits_{U} \int \limits_{\R} \Pi^\lam_*\mu(B(x,r))\,d\Pi^\lam_*\mu(x)\,d\lam \leq Cr,
\end{equation}
obtaining \eqref{eq:ac int}. Condition $\dim_{cor}(\mu, d_\lam) > 1$ is then improved to $\frac{h_{\mu}}{\chi_{\mu}(\mF^\lam)} > 1$ with the use of the Egorov theorem, similarly to the previous section.

In the case of the parameter dependent measure $\mu_\lam$ in the symbolic space, combining the above approach with the strategy from proof of Theorem \ref{thm:main_hausdorff} does not seem to work anymore. In particular, one could repeat the calculation from the proof of Proposition \ref{prop:dim cor cont} in order to bound integral $\int_{\R} \Pi^\lam_*\mu_\lam(B(x,r))\,d\Pi^\lam_*\mu_\lam(x)$ with the integral $\int_{\R} \Pi^\lam_*\mu_{\lam_0}(B(x,r))\,d\Pi^\lam_*\mu_{\lam_0}(x)$ for $\lam$ in a small neighbourhood $U'$ of $\lam_0$ (so that \eqref{eq:ac int} can be invoked for a fixed measure $\mu = \mu_{\lam_0}$). This, however, leads to a bound
\begin{equation}\label{eq:int reps bound} \int_{\R} \Pi^\lam_*\mu_\lam(B(x,r))\,d\Pi^\lam_*\mu_\lam(x) \leq r^{-\eps} \int_{\R} \Pi^\lam_*\mu_{\lam_0}(B(x,r))\,d\Pi^\lam_*\mu_{\lam_0}(x) + \mathrm{const}
\end{equation}
on a neighbourhood $U'$ of $\lam_0$ (depending on $\eps$). The error $r^{-\eps}$ is too large in order to combine \eqref{eq:fatou} with \eqref{eq:int reps bound} and \eqref{eq:tranvers int bound} for $\mu = \mu_{\lam_0}$ in order to obtain \eqref{eq:ac int}.

\subsection{Sobolev dimension}
A more refined approach, which leads to the proof of Theorem \ref{thm:main_cor_dim} and is the main part of \cite{BSSS}, is adapting the technique of Peres and Schlag \cite{PS00}, who worked with the Sobolev dimension $\dim_S$ rather than the correlation dimension. This is a notion of
 dimension extending the correlation dimension to values greater than $1$ (for finite measures on $\R$) with the crucial property that $\nu \ll \Lk^1$ whenever $\dim_S \nu > 1$. See Section \ref{R79} for the definition and more details. Peres and Schlag were able to prove that under the assumptions of Theorem \ref{thm:main_cor_dim}, if one considers the fixed measure $\mu_\lam \equiv \mu$ case, then
\[ \dim_S\nu_\lam \geq \min\{ \dim_{cor}(\mu_\lam, d_\lam), 1 + \delta \} \text{ for Lebesgue a.e. } \lam \in U, \]
{see \cite[Theorem 4.9 (4.22)]{PS00}}.
If one could prove an analog of Proposition \ref{prop:dim cor cont} for the Sobolev dimension, then we could repeat the proof of item \ref{it:dim cor par meas} of Theorem  \ref{thm:main_hausdorff} in order to conclude Theorem \ref{thm:main_cor_dim}. Unfortunately, here we face another complication: the map $\lam \mapsto \dim_S((\Pi^\Fk)_*\mu_\lam)$, in general, is not continuous, even under the stronger regularity condition \ref{as:measure}.

\begin{example}
Let $\Ak=\{1,2\}$ and consider an IFS $\Fk = \{f_1, f_2\}$ on $[0,1]$ where $f_1(x) = x/2$ and $f_2(x) = x/2 + 1/2$. Let $\Pi = \Pi^\Fk : \Sigma \to [0,1]$ be the corresponding natural projection map on $\Sigma = \{1,2 \}^{\N}$. For $\lam \in (0,1)$, let $\mu_\lam = (\lam, 1 - \lam)^{\N}$ be the corresponding Bernoulli measure on $\Sigma$. Let $\ov{U} \subset (0,1)$ be a compact interval containing $1/2$. It is straightforward to see that the family $\{ \mu_\lam\}_{\lam \in \ov{U}}$ satisfies \ref{as:measure} with $\theta = 1$. As $\Pi_* \mu_{1/2} = \Lk^1|_{[0,1]}$, we have $\dim_S(\Pi_*\mu_{1/2}) = 2$ (this follows directly from the definition of the Sobolev dimension and the formula $|\widehat{\Lk^1_{[0,1]}}(\xi)| = \frac{|e^{i\xi} - 1|}{|\xi|}$). On the other hand, $\Dh(\Pi_* \mu_{\lam}) \leq \frac{h_{\mu_\lam}}{\chi_{\mu_\lam}(\Fk)} = \frac{H(\lam)}{\log 2}$, where $H(\lam) = -\lam \log \lam - (1-\lam) \log (1 - \lam)$. Therefore $\dim_{cor}(\Pi_* \mu_{\lam}) \leq \Dh (\Pi_* \mu_{\lam}) < 1$ for $\lam \neq 1/2$, hence by Lemma \ref{R81} also $\dim_S \Pi_* \mu_{\lam} < 1$. This shows that $\lam \mapsto \dim_S((\Pi^\Fk)_*\mu_\lam)$ is not continuous in this case.
\end{example}

This makes it necessary for us to ``dive'' into the Peres-Schlag proof in \cite{PS00} and modify it in a way that suits our needs.
First, note that \cite{PS00} contains results in two versions: the $C^\infty$ case and the limited regularity case. It is the latter one that concerns us here.
It is treated in \cite{PS00} with less detail, often referring to a list of modifications needed, compared with the $C^\infty$ case. It is worth mentioning that
\cite{PS00} also contains results on the Hausdorff dimension of exceptional parameters for absolute continuity, which we do not address here.

\subsection{The 1-parameter case}

Theorem~\ref{thm:main_cor_dim} {in the 1-parameter case} is deduced from the following result, modelled after  \cite[Theorem 4.9]{PS00}.

\begin{thm}\label{thm:sobolev integral bound}
Let $\{f_j^\lam\}_{j \in \Ak}$ be a parametrized IFS satisfying smoothness assumptions \ref{as:MA1} - \ref{as:MA4} and the transversality condition \ref{as:trans mp} on
{$U\subset \R$, a bounded open interval}. Let $\left\{ \mu_\lam \right\}_{\lam \in \ov{U}}$ be a collection of finite Borel measures on $\Sig$ satisfying \ref{as:measure}. Fix $\lambda_0 \in U$, $\beta > 0$, $\gamma>0$, $\eps>0$ and $q>1$ such that $1+2\gamma + \eps<q<1+\min\{\delta, \theta\}$. Then, there exists a (sufficiently small) open interval $J \subset U$ containing $\lam_0$,
 such that for every smooth function $\rho$ on $\R$ with
{$0\le \rho\le 1$ and} $\supp(\rho) \subset J$ there exist constants $\widetilde{C}_1>0,\ \widetilde{C}_2>0$ such that
	$$
	\int_{J}\|\nu_\lambda\|_{2,\gamma}^2\rho(\lambda)\,d\lambda \leq \widetilde{C}_1\Ek_{q(1+a_0\beta)}(\mu_{\lambda_0}, d_{\lam_0}) + \widetilde{C}_2,
	$$
where $a_0=\frac{8+4\delta}{1+\min\{\delta, \theta\}}$.
\end{thm}

The use of a smoothing kernel $\rho$, replacing a characteristic function, is standard in harmonic analysis; there will be a few more such in the proof. On the other hand,
the parameter $\beta$ in the statement of the theorem may look mysterious; in fact, is comes from ``transversality of degree $\beta$'' introduced by Peres and Schlag \cite{PS00} and
defined in \eqref{tran beta} below.
In \cite{BSSS} it is shown that this condition follows from the ``usual'' transversality under our smoothness assumptions. {In the derivation
of Theorem~\ref{thm:main_cor_dim} from  Theorem~\ref{thm:sobolev integral bound} it will be essential that $\beta>0$ can be taken
arbitrarily small. }

\begin{lem}[Prop.\,6.1 from \cite{BSSS}]\label{prop_beta trans}
Let $\{f_j^\lam\}_{j \in \Ak}$ be a parametrized IFS satisfying smoothness assumptions \ref{as:MA1} - \ref{as:MA4} and the transversality condition \ref{as:trans mp} on
{ $U\subset \R$, a bounded open interval}. For every $\lam_0 \in U$ and $\beta>0$ there exists $c_\beta>0$ and an open neighbourhood $J$ of $\lam_0$ such that
\be \label{tran beta}
\left|\Pi^\lam(u) - \Pi^\lam(v)\right| < c_\beta\cdot d_{\lam_0}(u,v)^{1+\beta} \implies \left|\textstyle{\frac{d}{d\lam}}(\Pi^\lam(u) - \Pi^\lam(v))\right| \ge c_\beta\cdot d_{\lam_0}(u,v)^{1+\beta}.
\ee
holds for all $u, v\in \Sig$ and $\lam \in J$.
\end{lem}

The proof of the lemma is not difficult, but technical, so we leave it to the reader; full details are given in \cite[Prop.\,6.1]{BSSS}.

{\begin{remark} \label{rem:beta tran}
(i) An IFS satisfying \eqref{tran beta} is said to satisfy the transversality condition of degree $\beta$ on $J$.
 Although this is not necessary for the proof, for completeness we discuss how
our definition is related to the definition in \cite{PS00}. Incidentally, in \cite{PS00} the interval $J$ is called an interval of transversality of ``{\em of order $\beta$}'', but  we prefer
``{\em of degree $\beta$}'',
as in \cite[Definition 18.10]{Mattila Fourier}. Peres and Schlag first define $\Phi_\lam(u,v):= \frac{\Pi^\lam(u) - \Pi^\lam(v)}{d_{\lam_0}(u,v)}$ (this is not a typo, the denominator does not
depend on $\lam$). Then \cite[Definition 2.7]{PS00} says that an interval $J$ is an interval of transversality of order $\beta\in [0,1)$ if there exists $c_\beta>0$ such that for all
$\lam\in J$ and $u,v\in \Sig$,
$$
\left|\Phi_\lam(u,v)\right| \le c_\beta d_{\lam_0}(u,v)^\beta\implies \Bigl|\frac{d}{d\lam}\Phi_\lam(u,v)\Bigr| \ge c_\beta d_{\lam_0}(u,v)^\beta.
$$
This is, of course, equivalent to \eqref{tran beta}.

(ii) Note that transversality condition of degree $\beta$ implies our transversality condition \eqref{R40} for any $\beta\ge 0$, since in \eqref{R40} we require that $u_1 \ne v_1$, so that
$d_{\lam_0}(u,v)$ is bounded from below. The influence of $\beta$ matters only when the distance $d_{\lam_0}(u,v)$ may get arbitrarily small.
The most basic example of Bernoulli convolutions shows that \eqref{R40} does not imply transversality of
order $\beta=0$ in any neighborhood of $\lam_0$. In general, the length of the interval $J$ tends to zero as $\beta\to 0$.
\end{remark}
}

\begin{proof}[Derivation of Theorem~\ref{thm:main_cor_dim} assuming Theorem~\ref{thm:sobolev integral bound}]
It is enough to prove that for an arbitrary $t>0$ the set
$$
A = \bigl\{\lam\in \ov{U}: \dim_S(\nu_\lam) < \min \left\{\dim_{cor}(\mu_{\lam_0}, d_{\lam_0}), 1 + \min\{\delta, \theta \} \right\}-t\bigr\}
$$
has Lebesgue measure zero. Assuming the opposite, let $\lam_0$ be a density point of $A$. If $\dim_{cor}(\mu_{\lam_0},d_{\lam_0})\le 1$, we immediately get a contradiction by
Theorem~\ref{thm:main_hausdorff}(ii), in view of the fact that \ref{as:measure} is stronger than \ref{as:measure_cont} and the Sobolev dimension equals the correlation dimension when
the latter is less than one.
Thus we can assume that $\dim_{cor}(\mu_{\lambda_0}, d_{\lam_0}) >1$.

 Let $\eps>0$ be small enough to have
\[\gamma := \frac{\min \left\{\dim_{cor}(\mu_{\lam_0}, d_{\lam_0}), 1 + \min\{\delta, \theta \} \right\} - 4\eps - 1}{2} > 0.\]
Let $q = 1 + 2\gamma + 2\eps$. Then
\[ 1 + 2\gamma + \eps < q \leq \min \left\{\dim_{cor}(\mu_{\lam_0}, d_{\lam_0}), 1 + \min\{\delta, \theta \} \right\} - 2\eps.  \]
Let $\beta>0$ be small enough to have
\[ q(1 + a_0\beta) \leq \min \left\{\dim_{cor}(\mu_{\lam_0}, d_{\lam_0}), 1 + \min\{\delta, \theta \} \right\} - \eps, \]
where $a_0$ is as in Theorem \ref{thm:sobolev integral bound}. By Theorem \ref{thm:sobolev integral bound}, there exists a neighbourhood $J$ of $\lambda_0$ in $U$, interval $I$ containing $\lam_0$ and compactly supported in $J$ and smooth function $\rho$ with $0 \leq \rho \leq 1,\ \supp(\rho) \subset J$ and $\rho \equiv 1$ on $I$, such that
\[ \int_{I} \|\nu_\lam\|^2_{2,\gamma}\,d\lam \leq \int_{J}\|\nu_\lambda\|_{2,\gamma}^2\,\rho(\lambda)\,d\lambda\le \widetilde{C}_1\Ek_{q(1+a_0\beta)}(\mu_{\lambda_0}, d_{\lam_0}) + \tilde{C}_2 < \infty \]
as $q(1 + a_0\beta) \leq \dim_{cor}(\mu_{\lam_0}, d_{\lam_0}) - \eps$. Therefore, $\|\nu_\lam\|^2_{2,\gamma} < \infty$ for Lebesgue almost every $\lam \in {I}$, hence
\[\dim_S\nu_\lam \geq 1 + 2\gamma \geq \min \left\{\dim_{cor}(\mu_{\lam_0}, d_{\lam_0}), 1 + \min\{\delta, \theta \} \right\} - 4\eps\]
holds almost surely on ${I}$. As $\eps$ can be taken arbitrary small and the function $\lam \mapsto \dim_{cor}(\mu_\lam, d_\lam)$ is continuous by Proposition~\ref{prop:dim cor cont}(1a), we get a contradiction.
\end{proof}

The proof of Theorem~\ref{thm:sobolev integral bound} is rather long and technical, so we only sketch the key steps.
{ The notation
$
A\lesssim B
$
will mean that there  exist positive constants $C_1'$ and $C_2'$ such that $A \le C'_1 B + C'_2$. These constants will usually depend on the fixed parameters. Furthermore,
$A \asymp B$ will mean that for some $C_3'>1$ holds ${C'}_3^{-1} A \le B \le C_3'B$.
}

The first step is to ``decompose the frequency space dyadically,'' done with the
help of a Littlewood-Paley decomposition. The next result is the 1-dimensional case of \cite[Lemma 4.1]{PS00}, see also \cite[Lemma 18.6]{Mattila Fourier}.

\begin{lem} \label{lem:LP}
There  exists a Schwarz function $\psi\in \Sk(\R)$ such that
\begin{enumerate}[{\rm (i)}]
\item $\what\psi\ge 0$ and $\spt (\what\psi) \subset \{\xi:\ 1 \le |\xi|\le 4\}$;
\item $\sum_{j\in \Z} \psi(2^{-j}\xi) = 1$ for $\xi\ne 0$;
\item given any $\nu\in \Mk(\R)$ and any $\gam>0$, the following decomposition holds:
$$
\|\nu\|_{2,\gam}^2 \asymp \sum_{j\in \Z} 2^{2j\gam} \int (\psi_{2^{-j}}*\nu)(x)\,d\nu(x),
$$
where $\psi_{2^{-j}}(x) = 2^j \psi(2^j x)$.
\end{enumerate}
\end{lem}

For the proof of the lemma
take an even function $\eta \in \Sk(\R)$, non-increasing on $\R^+$, such that $0\le \eta\le 1$, it is equal to $1$ on $(-1,1)$ and is supported in $(-2,2)$. Then there exists a
function $\psi\in \Sk(\R)$ such that
$$
\what\psi(\xi) = \eta(\xi/2)-\eta(\xi),\ \ \xi\in \R.
$$
The properties (i), (ii) are easy to check, and (iii) follows from Parseval's formula in the form
$$
\int \ov f\,d\nu = \int \ov{\what f}\, \what \nu\,d\xi\ \ \ \mbox{for all}\ \ \nu\in \Mk(\R),\ f\in \Sk(\R),
$$
see \cite[(3.27)]{Mattila Fourier}, which implies
\be \label{eq:Parseval}
\int (\psi_{2^{-j}}*\nu)(x)\,d\nu(x) = \int \what\psi(2^{-j}\xi)|\,\what\nu(\xi)|^2\,d\xi \ge 0.
\ee

Schwarz functions decay faster than any power, thus for any $q>0$ there is $C_q>0$ such that
\be\label{eq:psi decay}
|\psi(\xi)| \le C_q(1 + |\xi|)^{-q}.
\ee
We will also use that
\be \label{meanzero}
\int_\R \psi(\xi)\,d\xi = \what\psi(0) = 0.
\ee
In fact, all higher moments of $\psi$ also vanish, but this will not be needed for our purposes.  As $\psi$ has bounded derivative on $\R$, there exists $L>0$ such that
\begin{equation}\label{eq:psi_lipschitz}
|\psi(x) - \psi(y)|\leq L |x - y| \text{ for all } x,y \in \R.
\end{equation}

\medskip

\subsection{Discretization and ``adjustment kernel''.}
In view of Lemma~\ref{lem:LP},
\be\label{eq:int to sum} \int_{\R}\|\nu_{\lambda}\|^2_{2,\gamma}\rho(\lambda)d\lambda \asymp \int_{\R}
\sum \limits_{j=-\infty}^\infty 2^{2j\gamma} \int_{\R} (\psi_{2^{-j}} * \nu_\lambda)(x)d\nu_\lambda(x)\rho(\lambda)d\lambda,
\ee
In order to prove Theorem \ref{thm:sobolev integral bound}, it is enough to consider in \eqref{eq:int to sum} the sum over $j\geq 0$, as $|(\psi_{2^{-j}} * \nu_\lambda)(x)|\le
2^j \|\psi\|_\infty$, hence the sum over $j < 0$ converges to a bounded function. 
By definition, for $j\ge 0$ we have
\begin{eqnarray*}
\int_{\R} (\psi_{2^{-j}} * \nu_\lambda)(x)\,d\nu_\lambda(x) & = & 2^j \int_{\R} \int_{\R} \psi(2^j(x-y))\, d\nu_\lambda(y)\,d\nu_\lambda(x)  \\
& = &
2^j \int\limits_{\Sig} \int\limits_{\Sig} \psi\bigl(2^j(\Pi^\lam(\omega_1) - \Pi^\lam(\omega_2))\bigr) \,d\mu_\lambda(\omega_1)\,d\mu_\lambda(\omega_2)
\end{eqnarray*}
 Let $\kappa = -\log_2\gam_2$; we recall that $\gam_2$ is the upper bound on the IFS contraction rates, see \ref{as:MA4}.
``Truncating'' at some level $n=\tc j$, for a suitable $\tc$, that is, replacing $\om_1$ and $\om_2$ by $\omega_1|_n 1^\infty$ and $\omega_2|_n 1^\infty$ respectively,
and estimating the error, using \eqref{eq:psi_lipschitz} and \ref{as:MA4}, yields that the last expression is
\[ \leq  2^j \sum \limits_{i \in \Ak^n} \sum \limits_{k \in \Ak^n} \psi\bigl(2^j(\Pi^\lam( i 1^\infty) - \Pi^\lam(k 1^\infty))\bigr)\,\mu_{\lambda}([i])\,\mu_{\lambda}([k]) +
L2^{2j+1-\kappa \tc j}=(*), \]
The parameter $\tc\ge 1$ will be chosen  to guarantee that $\tc\ge 4/\kappa$. Another key parameter to choose is the size of the
interval $J$ around $\lam_0$, which appears in the statement of Theorem~\ref{thm:sobolev integral bound}.
 Let $Q = \log_2 e$ and choose $\xi > 0 $ small enough to have $2(4+Qc)\xi < \eps$ and
\begin{equation}\label{eq:eta}
0<\frac{4+2\gamma}{\kappa - Q\xi} < \frac{\eps}{2(4+Qc)\xi}.
\end{equation}
Choose an open interval $J$ containing $\lam_0$ so small that $2c|J|^\theta\leq\xi$ (with $c,\theta$ as in \ref{as:measure}) and \eqref{tran beta} hold.
Then choose $\tc\geq 1$ such that
	\begin{equation}\label{eq:forn} \frac{4+2\gamma}{\kappa - Q\xi} \leq \tc \leq \frac{\eps}{2Q(2+c)\xi}\end{equation}
(it exists due to \eqref{eq:eta}).	

Now comes the crucial point, which makes our situation different from that of \cite{PS00}: we introduce a kernel $e_j$, which controls parameter dependence of $\mu_\lam$ at level
$n=\tc j$. Namely, we define a map $e_j\colon\Sig\times\Sig\times J\mapsto\R$ by
	\begin{equation}\label{eq:ej def}
	e_j(\omega_1,\omega_2,\lambda):= \begin{cases}\frac{\mu_\lambda([\omega_1|_n])\mu_\lambda([\omega_2|_n])}{\mu_{\lambda_0}([\omega_1|_n])\mu_{\lambda_0}([\omega_2|_n])}, &\text{ if } \mu_{\lambda_0}([\omega_1|_n])\mu_{\lambda_0}([\omega_2|_n]) \neq 0,\\
	1, & \text{ otherwise}.
	\end{cases}
	\end{equation}
By the property \ref{as:measure},
\begin{equation}\label{eq:errorest}
	e_j(\omega_1,\omega_2,\lambda)\leq e^{2c|\lambda-\lambda_0|^\theta n}\le 2^{Q\xi \tc j}\ \ \text{ for all $\omega_1,\omega_2$ and $\lambda\in J$.}
	\end{equation}
	Also by \ref{as:measure}, if $i \in \Sig^*$ is a fixed finite word, then  $\mu_{\lam_0}([i]) = 0$ if and only if $\mu_{\lam}([i]) = 0$ for all $\lam \in \ov{U}$; in other words: $\supp(\mu_{\lam_0}) = \supp(\mu_\lam)$. Denote $\wt\Ak^n:= \{i\in\Ak^n:\ \mu_{\lam_0}([i])\ne 0\}$.  We have, therefore, (note that now the integral is with respect to $\mu_{\lambda_0}$),
\[
\begin{split}
(*)& = 2^j \sum \limits_{i \in \wt\Ak^n} \sum \limits_{k \in \wt\Ak^n} 
\psi\bigl(2^j(\Pi^\lam(i 1^\infty) - \Pi^\lam(k 1^\infty))\bigr)\,\frac{\mu_\lambda([i])\mu_\lambda([k])}{\mu_{\lambda_0}([i])\mu_{\lambda_0}([k])}\,\mu_{\lambda_0}([i])\,\mu_{\lambda_0}([k]) + L2^{2j+1-\kappa\tc j} \\
&
= 2^{j} \int\limits_{\Sig} \int\limits_{\Sig} \psi\bigl(2^j(\Pi^\lam(\omega_1|_n 1^\infty) - \Pi^\lam(\omega_2|_n 1^\infty))\bigr)\,e_j(\omega_1,\omega_2,\lambda)\, d\mu_{\lambda_0}(\omega_1)\,d\mu_{\lambda_0}(\omega_2) + L2^{2j+1-\kappa\tc j}.
\end{split}
\]
Truncating again and estimating the error, similarly to the above, but now integrating with respect to $\mu_{\lam_0}$, finally yields:
\begin{equation}\label{eq:discretization_bound}
\begin{split}
\int_{\R} (\psi_{2^{-j}} &* \nu_\lambda)(x) \, d\nu_\lambda(x) \leq \\
&\quad \leq 2^{j} \int\limits_{\Sig} \int\limits_{\Sig} \psi\bigl(2^j(\Pi^\lam(\omega_1) - \Pi^\lam(\omega_2))\bigr)\,e_j(\omega_1,\omega_2,\lambda)\, d\mu_{\lambda_0}(\omega_1)\,d\mu_{\lambda_0}(\omega_2) + 4L 2^{(2 + Q\tc\xi-\tc\kappa)j},
\end{split}
\end{equation}
{where we leave the precise ``accounting'' in the error estimate to the reader (or see \cite[Section 7]{BSSS}). Note that the last additive term is not greater than
$4L\cdot 2^{-2j}$ by \eqref{eq:forn}.}
Now, substituting
\eqref{eq:discretization_bound} into \eqref{eq:int to sum} we obtain,  recalling that the sum over $j<0$ in \eqref{eq:int to sum} converges:
\[\begin{split}
\int_{J}\|&\nu_{\lambda}\|^2_{2,\gamma}\rho(\lambda)\,d\lambda\lesssim \int \limits_{\R} \sum \limits_{j=0}^\infty 2^{2j\gamma} \int\limits_{\R} (\psi_{2^{-j}} * \nu_\lambda)(x)\,d\nu_\lambda(x)\,\rho(\lambda)\,d\lambda\\
& \qquad \qquad \qquad = \sum \limits_{j=0}^\infty 2^{2j\gamma} \int \limits_{\R} \int\limits_{\R} (\psi_{2^{-j}} * \nu_\lambda)(x)\,d\nu_\lambda(x)\,\rho(\lambda)\,d\lambda\\
&\leq \sum \limits_{j=0}^\infty 2^{2j\gamma} \int \limits_{\R} \Bigl( 2^{j} \int\limits_{\Sig} \int\limits_{\Sig} \psi\bigl(2^j(\Pi^\lam(\omega_1) - \Pi^\lam (\omega_2))\bigr)\,e_j(\omega_1,\omega_2,\lambda)\, d\mu_{\lambda_0}(\omega_1)\,d\mu_{\lambda_0}(\omega_2)\\
&\qquad \qquad \qquad + 4L\cdot 2^{-(2\gam+2)j} \Bigr)\rho(\lambda)\,d\lambda \\
&\lesssim  \sum \limits_{j=0}^\infty 2^{j(2\gamma+1)} \int\limits_{\Sig} \int\limits_{\Sig} \left|\int_\R\psi\bigl(2^j(\Pi^\lam(\omega_1) - \Pi^\lambda (\omega_2))\bigr)\,e_j(\omega_1,\omega_2,\lambda)\,\rho(\lambda)\,d\lambda\right|\, d\mu_{\lambda_0}(\omega_1)\, d\mu_{\lambda_0}(\omega_2).
\end{split}\]
We have to be careful, since $\psi$ is not a positive function!
Exchanging summation and integration in the  2nd displayed line above is legitimate, since the inner integral is non-negative by
\eqref{eq:Parseval}. After that, exchanging the order of integration is allowed by Fubini's theorem, as $\rho$ is compactly supported, so the function is integrable.
Strictly speaking, we do not need
 the absolute value
sign outside of the inner integral in the last line above, since we know that the left-hand side is positive.

To finish the proof of Theorem \ref{thm:sobolev integral bound}, it is enough to show the following proposition (with the same notation as in Theorem \ref{thm:sobolev integral bound}).

\begin{prop}\label{prop:trans int}
	There exists {$ C_3 > 0$} such that for any distinct $\om_1,\om_2\in \Sig$, any $j\in \N$  we have 
	\begin{equation}\label{claim1}
	\left|\int_\R\psi\bigl(2^j(\Pi^\lam(\omega_1) - \Pi^\lam(\omega_2))\bigr)\,e_j(\omega_1,\omega_2,\lambda)\,\rho(\lambda)\,d\lambda\right| \leq  
	C_3\cdot \tc j2^{Q(2+c)\xi\tc j} \left( 1 + 2^j d(\om_1,\om_2)^{1+a_0\beta}\right)^{-q},\\
	\end{equation}
	where 
	$C_3$ depends only on $q, \rho$, and $\beta$, with $a_0 = \frac{8 + 4\delta}{1 + \min\{\delta, \theta\}}$ and $d(\om_1,\om_2) = d_{\lam_0}(\om_1,\om_2)$, the metric defined in \eqref{R78}.
\end{prop}

Indeed, if \eqref{claim1} holds, then, recalling the definition of energy \eqref{eq:energy}, we obtain
\[
\begin{split}
\int_{J}\|&\nu_{\lambda}\|^2_{2,\gamma}\,\rho(\lambda)\,d\lambda \\
& \lesssim
C_3\cdot \tc \sum_{j=0}^\infty j2^{j[2\gamma+ 1 + Q(2+c)\xi\tc-q]}\Ek_{q(1+a_0\beta)}(\mu_{\lambda_0}, d_{\lam_0})\\
& \leq  
C_3\cdot \tc\sum_{j=0}^\infty j2^{-\frac{\eps}{2}j}\Ek_{q(1+a_0\beta)}(\mu_{\lambda_0}, d_{\lam_0})<\infty,
\end{split}\]
and Theorem \ref{thm:sobolev integral bound} is proved. Here we used that $2\gamma+ Q(2+c)\xi\tc \le \eps/2$ by  \eqref{eq:forn} and $1 + 2\gam + \eps<q$ by the assumption of the theorem.

{\begin{remark}
The last proposition is \cite[Prop.\,7.2]{BSSS}; however, in \cite{BSSS}
we omitted the absolute value signs around the integral. The actual proof was for the absolute value. Strictly speaking,
taking the absolute value is unnecessary, since in the end we are estimating a positive quantity from above.
\end{remark}
}

\subsection{Proof sketch of Proposition \ref{prop:trans int}}

The proof is similar to that of \cite[Lemma 4.6]{PS00} in the case of limited regularity; however, some technical issues were treated in \cite{BSSS}
differently and in more detail, especially since \cite{PS00} leaves much to the reader. {(In fact, our natural projection is $1,\delta$-regular in the sense of \cite[Section 4.2]{PS00},
so we have $L=1$ in the Peres-Schlag notation. Note that equation \cite[(4.32)]{PS00} applies only when $L\ge 2$ and has a little typo; the case $L=1$ is special.)}

Fix distinct $\om_1,\om_2\in \Sig$ and denote $r = d(\om_1,\om_2)$. To simplify notation, let $e_j(\lambda):=e_j(\omega_1,\omega_2,\lambda)$.
Denote $\ov I = \supp(\rho)\subset J$. Since $J$ is open, there exists $K=K(\rho)\ge 1$ such that the $(2K^{-1})$-neighborhood of $\ov I$ is contained in $J$.

We can assume without loss of generality that $j$ is sufficiently large to satisfy $2^j r^{1+a_0\beta} > 1$ for a fixed $a_0$ (note that $r\le 1$).
Indeed, the integral in \eqref{claim1} is bounded above by $|J|\cdot \|\psi\|_\infty\cdot 2^{Q \xi \tc j}$, { in view of \eqref{eq:errorest}}, hence if $2^j r^{1+a_0\beta} \le 1$, then the inequality \eqref{claim1} holds with
$
C_3= |J|\cdot \|\psi\|_\infty \cdot 2^q$.

Let
$$
\phi\in C^\infty(\R),\ \ \ {0\le \phi\le 1},\ \ \phi\equiv 1\ \mbox{on}\ \ [-1/2,1/2],\ \ \supp(\phi)\subset (-1,1),
$$
and denote
$$\Phi_\lam = \Phi_\lam(\om_1,\om_2) := \frac{\Pi^\lam(\om_1) - \Pi^\lam(\om_2)}{d(\om_1,\om_2)}{= \frac{\Pi^\lam(\om_1) - \Pi^\lam(\om_2)}{r}.}$$
The idea, roughly speaking, is to separate the contribution of the zeros of $\Phi_\lam$, which are simple by transversality. Outside of a neighborhood of these zeros, we get an estimate using the rapid decay of $\psi$ at infinity, and near the zeros we linearize and use the fact that $\psi$ has zero mean. 
We have
\begin{eqnarray*}
	\int_\R \rho(\lam)\, \psi\!\left(2^j [\Pi^\lam(\om_1) - \Pi^\lam(\om_2)]\right)e_j(\lambda) d\lam  & = & \int \rho(\lam) \,\psi\!\left(2^j r \Phi_\lam \right)e_j(\lambda)\, \phi(K c_\beta^{-1} r^{-\beta} \Phi_\lam)\,d\lam \nonumber
	\\
	& + & \int \rho(\lam) \,\psi\!\left(2^j r \Phi_\lam\right)e_j(\lambda) \left[1-\phi(K c_\beta^{-1} r^{-\beta} \Phi_\lam)\right] d\lam  \label{integ1} \\[1.1ex]
	& =: & A_1 + A_2, \nonumber
\end{eqnarray*}
where $c_\beta$ is the constant from \eqref{tran beta}.
The integrand in $A_2$ is constant zero when $|K c_\beta^{-1} r^{-\beta} \Phi_\lam| \le \half$, hence we have $|\Phi_\lam|> \half K^{-1}C_\beta r^\beta$ in $A_2$.
The inequalities \eqref{eq:psi decay} and  \eqref{eq:errorest} yield
$$
|A_2| \le 
C_q \int |\rho(\lam)| |e_j(\lambda)| \bigl(1 + 2^j r \cdot {\textstyle{\half}} K^{-1} c_\beta r^\beta)^{-q}\, d\lam \le 
{\const\cdot} 2^{Q\xi \tc j} \bigl(1 + 2^j r^{1+\beta}\bigr)^{-q},
$$
for some constant depending on $q, \rho$ and $\beta$, and we obtained an upper bound dominated by the right-hand side of \eqref{claim1}. Thus it remains to estimate $A_1$.

Next comes a lemma where transversality is used.
 It is a variant of \cite[Lemma 4.3]{PS00} and is similar to \cite[Lemma 18.12]{Mattila Fourier}. Let $c_\beta$ be the constant from Proposition \ref{prop_beta trans}.

\begin{lem}\label{lem-trans}
	Under the assumptions and notation above, let
	$$
	\Jk:= \Bigl\{\lam\in J: \ |\Phi_\lam| < K^{-1} c_\beta r^{\beta} \Bigr\},
	$$
	which is a union of open disjoint intervals. Let $I_1,\ldots, I_{N_\beta}$ be the intervals of $\Jk$ intersecting $\ov I = \supp(\rho)$, enumerated in the order of $\R$. Then each $I_k$ contains a unique zero $\ov\lam_k$ of $\Phi_\lam$ and
	\be \label{inter0}
	[\ov\lam_k - d_\beta r^{2\beta}, \ov\lam_k + d_\beta r^{2\beta}]\subset I_k,
	\ee
	for some constant $d_\beta>0$.
	For every interval $I_k$ holds
	\be \label{inter1}
	2d_\beta\cdot r^{2\beta}\le |I_k| \le 2K^{-1},
	\ee
	hence
	\be \label{up1}
	N_\beta \le 2 + \textstyle{\half}d_\beta^{-1}|J| \cdot r^{-2\beta}.
	\ee
	Moreover,
	\be \label{low2}
	|\Phi_\lam| \le \textstyle{\half} K^{-1} c_\beta r^\beta\ \ \ \mbox{for all}\ \ \ \lam\in [\ov\lam_k - \textstyle{\half} d_\beta r^{2\beta}, \ov\lam_k + \textstyle{\half} d_\beta r^{2\beta}].
	\ee
	\end{lem}

\begin{proof}[Partial proof of  Lemma \ref{lem-trans}] This is a ``classical'' transversality argument.
Clearly,  $\lam\mapsto \Phi_\lam$ is continuous, so the intervals $I_k$ are well-defined. Since $K\ge 1$, on each of the intervals we have $|\frac{d}{d\lam}\Phi_\lam| \ge c_\beta r^{\beta}$ by the transversality condition \eqref{tran beta} of degree $\beta$.
	Thus $\Phi_\lam$ is strictly monotonic on each of the intervals. Let $\lam\in I_k {\cap} I$, where $\ov  I = \supp(\rho)$. Then $|\Phi_\lam| < K^{-1} c_\beta r^\beta$, and using the lower bound on the derivative we obtain that there exists unique
	$\ov \lam_k\in I_k$, such that $\Phi_{\ov \lam_k} = 0$, and it satisfies $|\lam - \ov \lam_k| \le K^{-1}$. For the rest of the proof, see \cite[Lemma 7.3]{BSSS}.
	The inequality \eqref{up1} follows from the lower bound in \eqref{inter1}.
	The proof of the
	claims  \eqref{inter0}, \eqref{low2} and the lower bound in \eqref{inter1} use the inequality {
	\be \label{doom}
	\left|\frac{d}{d\lam}\Phi_\lam\right|  \le C_{\beta,1} r^{-\beta} \iff \left| \frac{d}{d\lam}(\Pi^\lam(\om_1) - \Pi^\lam(\om_2))\right| \le C_{\beta,1} d_{\lam_0}(\om_1,\om_2)^{1-\beta},
	\ee
	which is somewhat technical, proven in
	\cite[Appendix C]{BSSS} (one can take $d_\beta = K^{-1} C_{\beta,1}^{-1}\cdot c_\beta$).  This is the place where it is important that $\beta>0$, since for $\beta=0$
	we cannot expect \eqref{doom} to hold in the entire neighborhood. }
	\end{proof}

In order to separate the contribution of the zeros of $\Phi_\lam$ we again use a smoothing bump function and let
 $\chi \in C^\infty(\R)$ be such that $\supp(\chi) \subset (-\half d_\beta, \half d_\beta)$, {$0\le \chi\le 1$}, and $\chi\equiv 1$ on $[-\frac{1}{4} d_\beta, \frac{1}{4} d_\beta]$.
By  Lemma~\ref{lem-trans} we can write
\begin{eqnarray*}
	A_1 & = & \int \rho(\lam) \,\psi\!\left(2^j r \Phi_\lam \right)e_j(\lambda) \phi(Kc_\beta^{-1} r^{-\beta} \Phi_\lam)\,d\lam \\
	& = & \sum_{k=1}^{N_\beta} \int \rho(\lam) \,\chi\bigl(r^{-2\beta} (\lam- \ov\lam_k)\bigr) \,\psi(2^j r\Phi_\lam)e_j(\lambda) \,\phi(Kc_\beta^{-1} r^{-\beta} \Phi_\lam)\,d\lam \\
	& + & \int \rho(\lam) \left[ 1- \sum_{k=1}^{N_\beta} \chi\bigl(r^{-2\beta} (\lam- \ov\lam_k)\bigr)\right]e_j(\lambda) \psi(2^j r\Phi_\lam) \, \phi(Kc_\beta^{-1} r^{-\beta} \Phi_\lam)\,d\lam\\
	& = & \sum_{k=1}^{N_\beta} A^{(k)}_1 + B.
\end{eqnarray*}
{The integral $B$ is estimated similarly to $A_2$ above. Using Lemma~\ref{lem-trans} and transversality, one can check that
 $|\Phi_\lam| \ge \frac{1}{4} d_\beta c_\beta r^{3\beta}$ on the support of the integrand in $B$.
It follows that on this support,
\be \label{add1}
|\psi(2^j r\Phi_\lam)| \le C_q \bigl( 1 + {(d_\beta c_\beta/4)}\cdot 2^j r^{1 + 3\beta}\bigr)^{-q},
\ee
by the rapid decay of $\psi$, and using \eqref{eq:errorest} we obtain $$|B| \le \const\cdot 2^{Q\xi \wt c j}\bigl( 1 + 2^j r^{1 + 3\beta}\bigr)^{-q}$$
for some constant depending on $q$ and $\beta$.}

It remains to estimate the integrals $A^{(k)}_1$.
Without loss of generality, we can assume $k=1$ and let $\ov\lam= \ov\lam_1$. In view of the bound \eqref{up1} on the number of intervals,
the desired inequality will follow from this.
{First one can check that, by construction, $\phi\equiv 1$ on the support of $\chi\bigl(r^{-2\beta} (\lam- \ov\lam)\bigr)$, and hence the $\phi$-term in $A_1^{(1)}$ can be
ignored, that is,}
$$
A_1^{(1)} = \int \rho(\lam) \,\chi\bigl(r^{-2\beta} (\lam- \ov\lam)\bigr)e_j(\lambda) \,\psi(2^j r\Phi_\lam) \,d\lam.
$$

It is convenient to make a change of variable, so we define a function $H$ via
\be \label{def-H}
\Phi_\lam = u \iff \lam = \ov \lam +H(u),\ \ \mbox{provided}\ \ \chi\bigl(r^{-2\beta} (\lam- \ov\lam)\bigr)\ne 0.
\ee
Note that $\chi\bigl(r^{-2\beta} (\lam- \ov\lam)\bigr)\ne 0$ implies $|\lam - \ov\lam| < \frac{1}{2} d_\beta r^{2\beta}$, so $\lam\in I_1$ by \eqref{inter0}, and by transversality,
\be\label{trans on chi}
\Bigl|\frac{d}{d\lam} \Phi_\lam\Bigr| \ge c_\beta r^\beta\ \text{ if }\ \chi\bigl(r^{-2\beta} (\lam- \ov\lam)\bigr)\ne 0.
\ee
Therefore, $H$ is well defined. We have
\begin{eqnarray*}
	A_1^{(1)} & = & \int \rho\bigl(\ov\lam + H(u)\bigr)\,\chi\bigl(r^{-2\beta} H(u) \bigr)e_j(\ov\lam+H(u)) \,\psi(2^j ru) \,H'(u)\,du \\
	& = & \int F(u)\,\psi(2^j ru) \,du,
\end{eqnarray*}
where
\be \label{def-F}
F(u) = \rho\bigl(\ov\lam + H(u)\bigr)\,\chi\bigl(r^{-2\beta} H(u) \bigr)e_j(\ov\lam+H(u))  \,H'(u).
\ee
Observe that $H'(u) = [\frac{d}{d\lam} \Phi_\lam]^{-1}$, hence \eqref{trans on chi} gives $|H'(u)|\le c_\beta^{-1} r^{-\beta}$ on the domain of $F$. Since $\rho$ and $\chi$ are bounded by 1, the inequality \eqref{eq:errorest} implies
\be \label{Fbound}
\|F\|_\infty \le {c^{-1}_\beta} \cdot r^{-\beta}2^{Q\xi \tc j}.
\ee

Recall that $\Phi_{\ov\lam} = 0$, so that $H(0)=0$.
Since $\int_\R \psi(\xi)\,d\xi = 0$ by \eqref{meanzero}, we can subtract $F(0)$ from $F(u)$ under the integral sign; we then split the integral as follows:
\begin{eqnarray}
A_1^{(1)} & = & \int [F(u) - F(0)]\, \psi(2^j ru) \,du \nonumber \\
& = & \int_{|u| < (2^j r)^{-1+\eps'}} [F(u) - F(0)]\, \psi(2^j ru) \,du + \int_{|u| \ge (2^j r)^{-1+\eps'} }[F(u) - F(0)]\, \psi(2^j ru) \,du \label{wehave} \\[1.2ex]
& =: & B_1 + B_2, \nonumber
\end{eqnarray}
where $\eps' \in (0, \frac{1}{2})$ is a small fixed number. Recall that our goal is to show
$$
|A^{(1)}_1| \le \const \cdot \tc j2^{Q(2+c)\xi\tc j}  \cdot \bigl(1 + 2^j r^{1 + a_0\beta}\bigr)^{-q},
$$
for some constants $a_0\ge 1$ and $\const$ depending only on $q$, $\rho$, and $\beta$. We can assume that $2^j r^{1 + a_0\beta} \ge 1$, otherwise, the estimate is trivial by increasing the constant. To estimate $B_2$,  note that  for any $M>0$ we have by the rapid decay of $\psi$:
$$
|\psi(2^j ru)| \le C_M\bigl( 1 + 2^j r|u|\bigr)^{-M},
$$
hence, by \eqref{Fbound},
\begin{eqnarray*}
	|B_2|  & \le & C_{\beta,M} \cdot r^{-\beta} \cdot 2^{Q\xi \tc j}(2^j r)^{-1} \int_{|x|\ge (2^j r)^{\eps'}} (1 + |x|)^{-M}\,dx  \\[1.1ex]
	& \le &  C_{\beta,M}' \cdot  r^{-\beta} \cdot 2^{Q\xi \tc j}(2^j r)^{-1} (2^j r)^{-\eps'(M-1)} \\[1.1ex]
	& \le & C_{\beta,M}''\cdot 2^{Q\xi \tc j} \cdot (2^j r^{1+2\beta})^{-q},
\end{eqnarray*}
for {$M$ such that $1 + \eps'(M-1)>q$.} Here we used that $2^j r \ge 2^j r^{1 + 2\beta} \geq 1$.

\smallskip

In order to estimate $B_1$, we show that the function $F$ from \eqref{def-F} is $\delta$-H\"older by our assumptions; we also need to estimate the constant in the H\"older bound.
We can write
$$
F(u) = \rho\bigl(\ov\lam + H(u)\bigr)\,\chi\bigl(r^{-2\beta} H(u) \bigr)e_j(\ov\lam+H(u))  \,H'(u) =: F_1(u) F_2(u)F_3(u) H'(u),
$$
and then
\begin{multline*}
F(u) - F(0) = \bigl(F_1(u) - F_1(0)\bigr) F_2(u)F_3(u) H'(u) + F_1(0) \bigl(F_2(u) - F_2(0)\bigr)F_3(u) H'(u) \\
+ F_1(0)F_2(0) \bigl(F_3(u) - F_3(0)\bigr) H'(u) + F_1(0) F_2(0)F_3(0) \bigl(H'(u) - H'(0)\bigr).
\end{multline*}
We have
$$
|F_1(u) - F_1(0)| = |\rho\bigl(\ov\lam + H(u)\bigr) - \rho\bigl(\ov\lam + H(0 )\bigr)| \le \|\rho'\|_\infty \cdot |H(u) - H(0)|,
$$
and then
\be \label{eta1}
|H(u) - H(0)| = |H(u)| = |\lam - \ov\lam| \le c^{-1}_\beta r^{-\beta} |\Phi_\lam-\Phi_{\ov\lam}| = c^{-1}_\beta r^{-\beta} |\Phi_\lam| = c^{-1}_\beta r^{-\beta} |u|,
\ee
by transversality, which applies since $\supp(F)\subset I_1$.
Similarly,
\be \label{eta2}
|F_2(u) - F_2(0)| \le \|\chi'\|_\infty \cdot r^{-2\beta} |H(u) - H(0)| \le C^{-1}_\beta \|\chi'\|_\infty \cdot r^{-3\beta} |u|.
\ee

For $F_3$ it is enough to assume that $\mu_{\lam_0}([\omega_1|_{\tc j}])\mu_{\lam_0}([\omega_2|_{\tc j}]) \neq 0$ (hence the same is true for $\mu_{\ov\lam}$ by \ref{as:measure}), as otherwise $e_j \equiv 1$ and then \eqref{eq:F_3 hoelder}, which is the goal of the calculation below, holds trivially.
{After some calculations, where we use the full strength of \ref{as:measure}), we obtain
\begin{equation}\label{eq:F_3 hoelder}
|F_3(u)-F_3(0)|\leq 2 {c_3} j2^{{c_4} j}c_\beta^{-\theta}r^{-\theta\beta}|u|^{\theta},\ \ \ \mbox{with}\ \ c_3 = Qc\tc\ \ \mbox{and}\ \ c_4 =Q(2+c)\tc \xi.
\end{equation}
For the details the reader is referred to \cite{BSSS}.
}

Finally, we need to estimate the term $|H'(u) - H'(0)|$. We have $H'(u) = [\frac{d}{d\lam} \Phi_\lam]^{-1}$, {and then using $\beta$-transversality \eqref{tran beta} and
\eqref{eta1}, but also a technical inequality from \cite[Proposition 4.5]{BSSS}, we obtain
$$
|H'(u) - H'(0)| \le \wt c_\beta r^{-\beta(3 + 2\delta)} |u|^\delta,
$$
see \cite{BSSS} for details.}

{Below, writing ``$\const$'' means constants depending on $q,\rho$, and $\beta$, which may be different from line to line.}
Using all of the above and $\|H'\|_\infty \le c_\beta^{-1}\cdot r^{-\beta}$ yields
$$
|F(u) - F(0)| \le \const\cdot {c_3} j2^{{c_4} j}\cdot \left( |u|^\delta r^{-\beta(3+2\delta)} + |u| r^{-4\beta}+|u|^{\theta}r^{-\beta(1+\theta)}\right),
$$
hence by \eqref{wehave} {and recalling that $(2^j r)\ge 1$ and $r\le 1$}, we obtain
\begin{eqnarray*}
	|B_1| & \le & \const\cdot {c_3} j2^{{c_4} j} \int_{|u| < (2^j r)^{-1+\eps'}} \left( |u|^\delta r^{-\beta(3+2\delta)} + |u| r^{-4\beta}+|u|^{\theta}r^{-\beta(1+\theta)}\right)\,du \\[1.2ex]
	& \le & \const\cdot {c_3} j2^{{c_4} j}\left( r^{-\beta(3+2\delta)} (2^j r)^{-(1-\eps')(1+\delta)} + (2^j r)^{-2(1-\eps')} r^{-4\beta}+(2^j r)^{-(1-\eps')(1+\theta)}r^{-\beta(1+\theta)}\right)\\
       & \leq & \const \cdot {c_3}  j2^{{c_4} j}r^{-\beta(4+2\delta)}(2^jr)^{-(1-\eps')(1+\min\{\delta, \theta\})},
\end{eqnarray*}
as $\min\{\delta, \theta\} \leq 1$. Therefore,
$$
|B_1| \le \const\cdot {c_3} j2^{{c_4} j}\left(2^j r^{1+a_0\beta}\right)^{-(1-\eps')(1+\min\{\delta, \theta\})},
$$
for  $a_0= \frac{8+4\delta}{1+\min\{\delta, \theta\}} \geq \frac{4+2\delta}{(1-\eps')(1+\min\{\delta, \theta\})}$.

Since $\eps'>0$ can be chosen arbitrarily small, we obtain
\[|B_1| \leq \const \cdot {c_3} j2^{{c_4} j}\left(1+2^j r^{1+a_0\beta}\right)^{-q} \text{ for any } q< 1+\min\{\delta,\theta\},\]
since as already mentioned, we can assume $2^j r^{1+a_0\beta} \ge 1$ without loss of generality.
{This concludes the proof of Proposition~\ref{prop:trans int} and of Theorem \ref{thm:sobolev integral bound}. }\qed

{

\subsection{The multiparameter case}
Let us now explain how one can extend the proof of Theorem \ref{thm:main_cor_dim} from the $1$-parameter case to the multiparameter one. The main difficulty is extending Proposition \ref{prop:trans int}, which is based on a rather delicate analysis. The crucial ingredient needed for reducing it to the one-dimensional case technique is the following lemma. It will allow us to slice $d$-dimensional balls in the parameter space with one-dimensional intervals, to which the techniques from the previous sections can be applied.

\begin{lem}\label{lem: mp transversality sliced}
	Let $\{f_j^\lam\}_{j \in \Ak}$ be a parametrized IFS satisfying smoothness assumptions \ref{as:MA1} - \ref{as:MA4} and the transversality condition \ref{as:trans mp} on $U$. Then for every $\lam_0 \in U$ there exists an open ball $B(\lam_0, \eps_0) \subset U$ with the following property: for every pair $\om,\tau \in \Sig$ with $\om_1 \neq \tau_1$ there exists a unit vector $e \in \R^d$ such that for every $p \in B(0, \eps_0) \cap \mathrm{span}(e)^\perp$ the one-parameter IFS $\{f_j^{\lam_0 + p + te} : t \in J_p, j \in \Ak\}$, where $J_p = \{ t \in \R : p + te \in B(0, \eps_0) \}$, satisfies
	\[\left|\Pi^{\lam_0 + p + te}(\om) - \Pi^{\lam_0 + p + te}(\tau)\right| < \eta/2 \implies \left|\textstyle{\frac{d}{dt}}(\Pi^{\lam_0 + p + te}(\om) - \Pi^{\lam_0 + p + te}(\tau))\right| \ge \eta/2 \]
	on $J_p$.
\end{lem}

\begin{proof}
	Fix $\lam_0 \in U$ and let $\eps_0 > 0$ be small enough to ensure $B(\lam_0, \eps_0) \subset U$,
	\be\label{eq:lam approx} \left|\Pi^{\lam_0}(\om) - \Pi^{\lam}(\om)\right| < \eta/4,\ \left|\nabla \left( \Pi^{\lam_0}(\om) - \Pi^{\lam_0}(\tau)\right) - \nabla \left( \Pi^{\lam}(\om) - \Pi^{\lam}(\tau)\right)\right| < \eta/2,
	\ee
	and
	\begin{align}\label{eq:lam approx angle}
		\begin{split}
			\bigg| \bigg\langle \frac{\nabla(\Pi^{\lam}(\om) - \Pi^{\lam}(\tau))}{\left| \nabla(\Pi^{\lam}(\om) - \Pi^{\lam}(\tau))\right|}&,\frac{\nabla\left(\Pi^{\lam_0}(\om) - \Pi^{\lam_0}(\tau)\right)}{\left|\nabla\left(\Pi^{\lam_0}(\om) - \Pi^{\lam_0}(\tau)\right)\right|}  \bigg\rangle \bigg| \geq \frac{1}{2} \\
			&\text{if } \left| \nabla(\Pi^{\lam}(\om) - \Pi^{\lam}(\tau))\right| \geq \frac{\eta}{2} \text{ and } \left|\nabla\left(\Pi^{\lam_0}(\om) - \Pi^{\lam_0}(\tau)\right)\right| \geq \frac{\eta}{2}
		\end{split}
	\end{align}
	for $\lam \in B(\lam_0, \eps_0)$ and all $\om,\tau \in \Sig$. Fix $\om,\tau \in \Sig$ with $\om_1 
	\neq \tau_1$. We can assume that $\left|\Pi^{\lam}(\om) - \Pi^{\lam}(\tau)\right| < \eta/2$ for some $\lam \in B(\lam_0, \eps_0)$ (as otherwise the assertion of the lemma holds trivially with any choice of $e$). Then \eqref{eq:lam approx} implies $ \left|\Pi^{\lam_0}(\om) - \Pi^{\lam_0}(\tau)\right| < \eta$, hence by \ref{as:trans mp} we have
	$\left|\nabla\left(\Pi^{\lam_0}(\om) - \Pi^{\lam_0}(\tau)\right)\right| \geq \eta$.
	We define $e = \frac{\nabla\left(\Pi^{\lam_0}(\om) - \Pi^{\lam_0}(\tau)\right)}{\left|\nabla\left(\Pi^{\lam_0}(\om) - \Pi^{\lam_0}(\tau)\right)\right|}$. Assume now that
	\[ \left|\Pi^{\lam_0 + p + te}(\om) - \Pi^{\lam_0 + p + te}(\tau)\right| < \eta/2
	\]
	holds for some $p \in B(0, \eps_0) \cap \mathrm{span}(e)^\perp$ and $t \in J_p$. Then by \ref{as:trans mp},
	\be\label{eq:lam trans mp} \left|\nabla\left(\Pi^{\lam_0 + p + te}(\om) - \Pi^{\lam_0 + p + te}(\tau)\right)\right| \geq \eta, \ee
	and hence by \eqref{eq:lam approx} we also have $\left|\nabla\left(\Pi^{\lam_0}(\om) - \Pi^{\lam_0}(\tau)\right)\right| \geq \eta/2$. Therefore, \eqref{eq:lam trans mp} and \eqref{eq:lam approx angle} give
	\begin{align*}
		\begin{split}
			\left|\textstyle{\frac{d}{dt}}(\Pi^{\lam_0 + p + te}(\om) - \Pi^{\lam_0 + p + te}(\tau))\right| &= \left| \left\langle \nabla(\Pi^{\lam_0 + p + te}(\om) - \Pi^{\lam_0 + p + te}(\tau)), e \right\rangle \right| \\
			& = \left| \left\langle \nabla(\Pi^{\lam_0 + p + te}(\om) - \Pi^{\lam_0 + p + te}(\tau)),\frac{\nabla\left(\Pi^{\lam_0}(\om) - \Pi^{\lam_0}(\tau)\right)}{\left|\nabla\left(\Pi^{\lam_0}(\om) - \Pi^{\lam_0}(\tau)\right)\right|}  \right\rangle \right|\\
			& \geq  \frac{1}{2}\left|\nabla\left(\Pi^{\lam_0 + p + te}(\om) - \Pi^{\lam_0 + p + te}(\tau)\right)\right| \\
			& = \eta / 2.
		\end{split}
	\end{align*}
\end{proof}

Using the above lemma, one can prove multiparameter versions of each of the main steps of the proof of Theorem \ref{thm:main_cor_dim} presented in the previous section. As the proofs are essentially the same as in the $1$-dimensional case (one only has to check that the constants can be controlled uniformly with respect to $\om, \tau$ and $p$), we only present sketches commenting on the appropriate changes to be made with respect to the full proofs in \cite{BSSS}, leaving details to the reader. The first step is extending Lemma \ref{prop_beta trans} on transversality of degree $\beta$.

\begin{prop}\label{prop_beta trans mp}
	Let $\{f_j^\lam\}_{j \in \Ak}$ be a parametrized IFS satisfying smoothness assumptions \ref{as:MA1} - \ref{as:MA4} and the transversality condition \ref{as:trans mp} on $U$. For every $\lam_0 \in U$ and $\beta>0$ there exists $c_\beta>0$ and an open neighbourhood $J = B(\lam_0, \eps_0)$ of $\lam_0$ with the following property: for every $\om, \tau\in \Sig$ there exists a unit vector $e \in \R^d$ such that for every $p \in B(0, \eps_0) \cap \mathrm{span}(e)^\perp$ for all $t \in J_p = \{t \in \R : p + te \in B(0, \eps_0) \}$ the following holds:
	\begin{multline*}
	\left|\Pi^{\lam_0 + p + te}(\om) - \Pi^{\lam_0 + p + te}(\tau)\right| < c_\beta\cdot d_{\lam_0}(\om,\tau)^{1+\beta}\\
	 \implies \left|\textstyle{\frac{d}{dt}}(\Pi^{\lam_0 + p + te}(\om) - \Pi^{\lam_0 + p + te}(\tau))\right| \ge c_\beta\cdot d_{\lam_0}(\om,\tau)^{1+\beta}.
	 \end{multline*}
\end{prop}
\begin{proof}[Sketch of the proof] As we have referred to \cite{BSSS} for the proof of Lemma \ref{prop_beta trans} (\cite[Proposition 6.1]{BSSS}), we shall explain changes one has to perform in the proof of \cite[Proposition 6.1]{BSSS} in order to obtain the above statement.	The idea is to repeat the proof of \cite[Proposition 6.1]{BSSS} on each interval $J_p$, with the application of Lemma \ref{lem: mp transversality sliced} instead of \eqref{R40}. More precisely, in the point of the proof in \cite{BSSS} where \eqref{R40} is applied to the pair $\sigma^n \om, \sigma^n \tau$ with $n = |\om \wedge \tau|$, we apply Lemma \ref{lem: mp transversality sliced} instead, with the choice of $e$ as corresponding to the pair $\sigma^n \om, \sigma^n \tau$. We also use the observation that  \ref{as:MA1} -- \ref{as:MA4} imply that the one-parameter IFS $\{ f^{\lam_0 + p + te}_j : t \in J_p\}_{j \in \Ak}$ satisfies assumptions  \ref{as:MA1} -- \ref{as:MA4} with the one-dimensional parameter $t$ and constants independent of $p$ and $\om,\tau$ (we use here the fact that $e$ is a unit vector). Consequently, all the regularity lemmas of \cite[Section 4]{BSSS} hold for each such one-parameter IFS, uniformly in $p,\om,\tau$, and the proof of \cite[Proposition 6.1]{BSSS} can be applied to each of them. Therefore, even though $e$ depends on $\om$ and $\tau$, all the final constants in the proposition are independent of $\om,\tau$ and $p$.
\end{proof}

Now we are ready to explain how to modify the proof of Theorem \ref{thm:main_cor_dim} in order to obtain its multiparameter version. It suffices to prove the following multiparameter version of Proposition \ref{thm:sobolev integral bound}, which then can be used in the same manner as before to prove Theorem \ref{thm:main_cor_dim}.

\begin{prop}\label{thm:sobolev integral bound mp}
	Let $\{f_j^\lam\}_{j \in \Ak}$ be a parametrized IFS satisfying smoothness assumptions\ref{as:MA1} -- \ref{as:MA4} and the transversality condition \ref{as:trans mp} on $U \subset \R^d$. Let $\left\{ \mu_\lam \right\}_{\lam \in \ov{U}}$ be a collection of finite Borel measures on $\Sig$ satisfying \ref{as:measure}. Fix $\lambda_0 \in U$, $\beta > 0$, $\gamma>0$, $\eps>0$ and $q>1$ such that $1+2\gamma + \eps<q<1+\min\{\delta, \theta\}$. Then, there exists a (small enough) open ball $J = B(\lam_0, \eps_0) \subset U$ such that for every smooth function $\rho$ on $\R^d$ with
	{$0\le \rho\le 1$ and} $\supp(\rho) \subset J$ there exist constants $\widetilde{C}_1>0,\ \widetilde{C}_2>0$ such that
	$$
	\int_{J}\|\nu_\lambda\|_{2,\gamma}^2\rho(\lambda)\,d\lambda \leq \widetilde{C}_1\Ek_{q(1+a_0\beta)}(\mu_{\lambda_0}, d_{\lam_0}) + \widetilde{C}_2,
	$$
	where $a_0=\frac{8+4\delta}{1+\min\{\delta, \theta\}}$.
\end{prop}

To prove it, note first that we can follow the proof of Proposition \ref{thm:sobolev integral bound} exactly as before up to Proposition \ref{prop:trans int}, since only conditions \ref{as:MA1} -- \ref{as:MA4} and \ref{as:measure} are used in that part (in particular: the transversality condition is not invoked). We can define the functions $e_j$ on $\Sig \times \Sig \times U$ by the same formula \eqref{eq:ej def}. Inspecting the part of the proof following Proposition \ref{prop:trans int}, we see that the proof of Proposition \ref{thm:sobolev integral bound mp} will be concluded once the following extension of Proposition \ref{prop:trans int} is established.

\begin{prop}\label{prop:trans int mp}
	There exists {$ C_7 > 0$} such that for any distinct $\om_1,\om_2\in \Sig$, any $j\in \N$  we have 
	\begin{equation}\label{claim1 mp}
		\left| \int\limits_{\R^d}\psi\bigl(2^j(\Pi^\lam(\omega_1) - \Pi^\lam(\omega_2))\bigr)\,e_j(\omega_1,\omega_2,\lambda)\,\rho(\lambda)\,d\lambda \right| \leq  
		C_7\cdot \tc j2^{Q(2+c)\xi\tc j} \left( 1 + 2^j d(\om_1,\om_2)^{1+a_0\beta}\right)^{-q},\\
	\end{equation}
	where 
	$C_7$ depends only on $q, \rho$, and $\beta$, and $a_0 = \frac{8 + 4\delta}{1 + \min\{\delta, \theta\}}$, and $d(\om_1,\om_2) = d_{\lam_0}(\om_1,\om_2)$ is the metric defined in \eqref{R78}.
\end{prop}

\begin{proof}[Sketch of the proof] Fix distinct $\om_1, \om_2 \in \Sig$, and let $e$ be the corresponding unit vector from Proposition \ref{prop_beta trans mp}. We can decompose the integral as follows (recall that $J=B(\lam_0,\eps_0)$ and $\supp(\rho) \subset J$):
	\begin{align*}
		\begin{split} \int\limits_{\R^d}&\psi\bigl(2^j(\Pi^\lam(\omega_1) - \Pi^\lam(\omega_2))\bigr)\,e_j(\omega_1,\omega_2,\lambda)\,\rho(\lambda)\,d\lambda \\
			&= \int \limits_{B(0, \eps_0) \cap \mathrm{span}(e)^\perp} \int \limits_{J_p} \psi\bigl(2^j(\Pi^{\lam_0 + p + te}(\omega_1) - \Pi^{\lam_0 + p + te}(\omega_2))\bigr)\,e_j(\omega_1,\omega_2,\lambda_0 + p + te)\,\rho(\lam_0 + p + te)dt dp.
		\end{split}
	\end{align*}
	Note that \ref{as:measure} implies that the family of measures $t \mapsto \mu_{\lam_0 + p + te}$ satisfies \ref{as:measure} on the interval $J_p$, with constants independent of $p$. This fact, together with Proposition \ref{prop_beta trans mp}, allows us to repeat the proof of Proposition \ref{prop:trans int} on each interval $J_p$ and obtain a corresponding upper bound for it. The crucial observation is that we obtain the upper bound \eqref{claim1 mp} for the inner integral above for every fixed $p \in B(0, \eps_0) \cap \mathrm{span}(e)^\perp$, with constants independent of $p$. Integrating with respect to $p$ finishes the proof (note that even though the direction $e$ in which we "slice" the ball $B(\lam_0, \eps_0)$ depends on $\om_1, \om_2$, the final upper bound on the integral in \ref{claim1 mp} does not, as the constants in Proposition \ref{prop_beta trans mp} do not depend on $\om, \tau$).
\end{proof}

}

{
	\section{Families of Gibbs measures have property \ref{as:measure} -- on the proof of Theorem~\ref{thm:main_gibbs}}	
	
	Finally, we will show how Theorem~\ref{thm:main_gibbs} follows from Theorem~\ref{thm:main_hausdorff} and Theorem~\ref{thm:main_cor_dim}. More precisely, we will give a sketch of the proof how Gibbs measures with parameter dependent potential satisfy \ref{as:measure}.
	
	Let $U\subset\R^d$ be an open and bounded set, and let $\phi^\lambda\colon\Sigma\to\R$ be a family of H\"older-continuous potentials with the following properties:
	\begin{enumerate}[(H1)]
		\item\label{it:pot1} there exists $0<\alpha<1$ and $b>0$ such that
		$$\sup_{\lambda\in\overline{U}}\sup_{\omega,\tau:|\omega\wedge\tau|=k}|\phi^\lambda(\omega)-\phi^\lambda(\tau)|\leq b\alpha^k;$$
		\item\label{it:pot2} there exists $c>0$ and $0<\theta<1$ such that
		$$
		\sup_{\omega\in\Sigma}|\phi^\lambda(\omega)-\phi^{\lambda'}(\omega)|\leq c|\lambda-\lambda'|^\theta\text{ for every }\lambda,\lambda'\in\overline{U}.
		$$
	\end{enumerate}
	
	Clearly, for every $\lambda\in\overline{U}$ there exists a unique Gibbs measure $\mu_\lambda$ satisfying \eqref{R66}. It is an easy exercise to show that the assumptions \ref{it:pot1}--\ref{it:pot2} imply \ref{as:measure_cont}, and so the dimension part of Theorem~\ref{thm:main_gibbs} easily follows by Theorem~\ref{thm:main_hausdorff}.
	
	It is reasonably more challenging to show that the family of measures $\mu_\lambda$ also satisfies \ref{as:measure} for some $0<\theta'<\theta$ and $c'>0$; however, it still uses standard methods from operator theory. But this is not the main difficulty here. In general, the correlation dimension $\dim_{cor}(\mu_\lambda,d_\lambda)$ is significantly smaller than the ratio $\frac{h_{\mu_\lambda}}{\chi_{\mu_\lambda}(\mathcal{F}^\lambda)}$. Hence, to show the absolute continuity part of Theorem~\ref{thm:main_gibbs} by applying Theorem~\ref{thm:main_cor_dim}, one needs to restrict $\mu_\lambda$ to an appropriate subset of $\Sigma$ of large measure. However, with such a restriction we might loose the property \ref{as:measure}. The main proposition of this section says that this can be done in a careful way, so that continuity properties are not harmed.
	
	\begin{prop}\label{prop:maingibbs}
		Let $\{f_j^\lam\}_{j \in \Ak}$ be a parametrized IFS satisfying smoothness assumptions \ref{as:MA1} - \ref{as:MA4}. Let $\{\mu_\lambda\}_{\lambda\in\overline{U}}$ be a family of shift-invariant Gibbs measures with potentials $\phi^\lambda\colon\Sigma\to\R$ satisfying \ref{it:pot1} and \ref{it:pot2}. Then for every $\lambda_0\in U$, $\varepsilon>0$, $\delta>0$ and $\theta'\in(0,\theta)$ there exist $\xi=\xi(\lambda_0,\varepsilon,\delta)>0$ and $c>0$ and a subset $A\subseteq\Sigma$ such that for every $\lambda\in B_\xi(\lambda_0)$:
		\begin{enumerate}[{\rm (1)}]
			\item\label{it:prop1} $\mu_\lambda(A)>1-\delta$;
			\item\label{it:prop2} $\dim_{cor}(\mu_\lambda|_A,d_\lambda)\geq\frac{h_{\mu_\lambda}}{\chi_{\mu_\lambda}(\mathcal{F}^\lambda)}-\varepsilon$;
			\item\label{it:prop3} and for every $\omega\in\Sigma^*$,
			$$
			e^{-c|\lambda-\lambda_0|^{\theta'}|\omega|}\mu_{\lambda}|_A([\omega])\leq\mu_{\lambda_0}|_A([\omega])\leq e^{c|\lambda-\lambda_0|^{\theta'}|\omega|}\mu_{\lambda}|_A([\omega]).
			$$
		\end{enumerate}
	\end{prop}
	
	First, we sketch the proof of Theorem~\ref{thm:main_gibbs}, {assuming the proposition}. For every $\lambda_0\in U$, $\varepsilon>0$ and $\delta>0$, one can apply Theorem~\ref{thm:main_hausdorff} and Corollary~\ref{cor:ac} for the measure $\mu_\lambda|_A$ on $U=B_\xi(\lambda_0)$, where the set $A$ is given by Proposition~\ref{prop:maingibbs}, and so, for almost every $\lambda\in B_\xi(\lambda_0)$,
	$$\frac{h_{\mu_\lambda}}{\chi_{\mu_\lambda}(\mathcal{F}^\lambda)}\geq\Dh\mu_\lambda\geq\dim_{cor}(\mu_\lambda|_A,d_\lambda)\geq\frac{h_{\mu_\lambda}}{\chi_{\mu_\lambda}(\mathcal{F}^\lambda)}-\varepsilon$$ and
	$$ \mu_\lambda|_A\ll\mathcal{L}\ \ \text{ for Lebesgue a.e.}\ \ \lambda\in B_\xi(\lambda_0)\cap\left\{\lambda: h_{\mu_\lambda}>\chi_{\mu_\lambda}(\mathcal{F}^\lambda)(1+\varepsilon)\right\}.$$
	Since $\lambda_0\in U$, $\varepsilon>0$ and $\delta>0$ were arbitrary, one can finish the proof of Theorem~\ref{thm:main_gibbs} by a standard density theorem argument.
	
	\smallskip
	
	Now, let us turn to the discussion of the proof of Proposition~\ref{prop:maingibbs}. First, we need to get more involved into \cite[Chapter~1]{Bow} for a more sophisticated version of Theorem~\ref{R67}. Let $L_\lambda$ be the Perron operator on the Banach space of continuous real-valued maps on $\Sigma$, defined by
	$$
	(L_\lambda h)(\omega)=\sum_{i\in\mathcal{A}}e^{\phi^\lambda(i\omega)}h(i\omega).
	$$
	Let 
	$$\Lambda=\left\{f\colon\Sigma\to\R_+:f(\omega)\leq \exp\bigl(\textstyle{\sum_{k=|\omega\wedge\tau|+1}^\infty 2b\alpha^{k}}\bigr)f(\tau)\text{ for every }\omega,\tau\in\Sigma\right\}.$$
	 Then for every $\lambda\in\overline{U}$ there exist a unique function $h_\lambda\in\Lambda$ and a unique probability measure $\nu_\lambda$ on $\Sigma$ such that
	$$
	L_\lambda h_\lambda=e^{P(\phi^\lambda)}h_\lambda,\ L_\lambda^*\nu_\lambda=e^{P(\phi^\lambda)}\nu_\lambda,\text{ and }\int h_\lambda(\omega)d\nu_\lambda(\omega)=1,
	$$
	where $P(\phi^\lambda)$ is the pressure defined in \eqref{R68}. Furthermore, there exist $0<\beta<1$ and $A>0$ such that for every $g\colon\Sigma\to\R$ with $\mathrm{var}_r(g)=0$ for some $r\geq0$ and for every $\lambda\in\overline{U}$,
	\begin{equation}\label{eq:unifconv}
		\Bigl\|e^{-(n+r)P(\phi^\lambda)}L_\lambda^{n+r}g-\int gd\nu_\lambda\cdot h_\lambda\Bigr\|\leq A\beta^n\int gd\nu_\lambda.
	\end{equation}
	These claims follow by \cite[Chapter~1]{Bow} with the uniform  bound $\alpha$ in \ref{it:pot1}.
	
	It clearly follows from the definition of the pressure and \ref{it:pot2} that for every $\lambda,\lambda'\in U$,
	\begin{equation}\label{eq:presscont}
		|P(\phi^\lambda)-P(\phi^{\lambda'})|\leq c|\lambda-\lambda'|^\theta.
	\end{equation}
	
	\begin{lem}\label{lem:1}
		For every $0<\theta'<\theta$ there exists $c_{\theta'}>0$ such that for every $\lambda,\lambda'\in U$
		$$
		\frac{h_\lambda(\omega)}{h_{\lambda'}(\omega)}\leq e^{c_{\theta'}|\lambda-\lambda'|^{\theta'}} \ \text{ for every $\omega\in\Sigma$ and }\ \frac{\nu_{\lambda}([\omega])}{\nu_{\lambda'}([\omega])}\leq e^{c_{\theta'}|\lambda-\lambda'|^{\theta'}|\omega|}\ \text{ for every }\ \omega\in\Sigma^*.
		$$
	\end{lem}
	
	The first part of the lemma follows by \eqref{eq:unifconv}, that is, the eigenfunction $h_\lambda$ can be uniformly approximated by $L_\lambda^n1$. Similarly, the second part follows {from the observation} that $\nu_\lambda([\omega])$ can be uniformly approximated by $L_\lambda^{n+|\omega|}\ind_{[\omega]}$. For details, see \cite[Section~8.1]{BSSS}. 
	
	Now, the Gibbs measure $\mu_\lambda$ (defined in Theorem~\ref{R67}) is $\mu_\lambda(A)=\int_Ah_\lambda d\nu_\lambda$. A simple consequence of Lemma~\ref{lem:1} is that for every $0<\theta'<\theta$ there exists $c_{\theta'}>0$ such that for every $\omega\in\Sigma^*$,
	\begin{equation}\label{eq:contmu}
		\frac{\mu_{\lambda}([\omega])}{\mu_{\lambda'}([\omega])}\leq e^{c_{\theta'}|\lambda-\lambda'|^{\theta'}|\omega|}.
	\end{equation}
	
	Now, we wish to show that there exists a common Egorov set $A$ in a sufficiently small neighbourhood of every $\lambda_0\in U$ to verify Proposition~\ref{prop:maingibbs}. So, let $\lambda_0\in U$ and $\varepsilon>0$ be arbitrary but fixed. By the large deviation principle, see \cite[Theorem~6]{Young90}, there exist $C>0$ and $s>0$ such that
	\begin{equation}
		\mu_{\lambda_0}\left(\left\{\omega:\left|\frac1n S_n\phi_{\lambda_0}(\omega)-\int\phi_{\lambda_0} d\mu_{\lambda_0}\right|>\varepsilon/4\right\}\right)\leq Ce^{-ns}.
	\end{equation}
	Observe that $\lambda\mapsto\int\phi^\lambda d\mu_\lambda$ is $\theta''$-H\"older with some $0<\theta''<\theta'$ and with some constant $c_{\theta''}>0$. Let $\xi>0$ be such that $\alpha e^{c_{\theta'}\xi^{\theta'}}<1$, $c\xi^\theta+c_{\theta''}\xi^{\theta''}<\varepsilon/4$ and $c_{\theta'}\xi^{\theta'}<s/2$. Our first claim is that for every $\lambda\in B_\xi(\lambda_0)$:
	\begin{multline}\label{eq:foregorov1}
		\mu_\lambda\left(\left\{\omega:\left|\frac1n S_n\phi_\lambda(\omega)-\int\phi_\lambda d\mu_\lambda\right|>\varepsilon\right\}\right)\\
		\leq C'e^{c_{\theta'}|\lambda-\lambda_0|^{\theta'}n}\mu_{\lambda_0}\left(\left\{\omega:\left|\frac1n S_n\phi_{\lambda_0}(\omega)-\int\phi_{\lambda_0} d\mu_{\lambda_0}\right|>\varepsilon/4\right\}\right)\leq C''e^{-sn/2}
	\end{multline}
	for every $n\geq1$. Indeed, let $\tau\in\Sigma$ be arbitrary but fixed. Then choosing $k\geq1$ such that $b\alpha^k<\varepsilon/4$, we have
	\[\begin{split}
		&\mu_\lambda\left(\left\{\omega:\left|\frac1n S_n\phi_\lambda(\omega)-\int\phi_\lambda d\mu_\lambda\right|>\varepsilon\right\}\right)\\
		&\qquad\leq\mu_\lambda\left(\left\{\omega:\left|\frac1n S_n\phi_{\lambda_0}(\omega)-\int\phi_{\lambda_0}d\mu_{\lambda_0}\right|>3\varepsilon/4\right\}\right)\text{ (by \ref{it:pot2} and $c\xi^\theta+c_{\theta''}\xi^{\theta''}<\varepsilon/4$)}\\
		&\qquad\leq
		\sum_{|\omega|=n+k}\mu_{\lambda}([\omega])\ind\left\{\left|\frac1n S_n\phi_{\lambda_0}(\omega\tau)-\int\phi_{\lambda_0}d\mu_{\lambda_0}\right|>2\varepsilon/4\right\}\text{ (by \ref{it:pot1} and $b\alpha^k<\varepsilon/4$)}\\
		&\qquad\leq e^{c_{\theta'}|\lambda-\lambda_0|^{\theta'}(n+k)}\sum_{|\omega|=n+k}\mu_{\lambda_0}([\omega])\ind\left\{\left|\frac1n S_n\phi_{\lambda_0}(\omega\tau)-\int\phi_{\lambda_0}d\mu_{\lambda_0}\right|>2\varepsilon/4\right\}\text{ (by \eqref{eq:contmu})}\\
		&\qquad\leq e^{c_{\theta'}\xi^{\theta'}k}e^{c_{\theta'}|\lambda-\lambda_0|^{\theta'}n}\mu_{\lambda_0}\left(\left\{\omega:\left|\frac1n S_n\phi_{\lambda_0}(\omega)-\int\phi_{\lambda_0}d\mu_{\lambda_0}\right|>\varepsilon/4\right\}\right)\text{ (by \ref{it:pot1} and $b\alpha^k<\varepsilon/4$)}.
	\end{split}\]
	
	There are actually two potentials which play a role in the dimension: $\phi^\lambda$ and $\varphi_\lambda(\omega)=-\log|(f_{\omega_1}^\lambda)'(\Pi_\lambda(\sigma\omega))|$,
	in view of \eqref{R54} and $\chi_{\mu_\lambda}(\mathcal{F}_\lambda)=\int \varphi_\lambda d\mu_\lambda$. Observe that we may assume that $\varphi_\lambda$ also satisfies \ref{it:pot1} and \ref{it:pot2}, and so both $\lambda\mapsto h_{\mu_\lambda}$ and $\lambda\mapsto\chi_{\mu_\lambda}(\mathcal{F}_\lambda)$ are $\theta''$-H\"older with $0<\theta''<\theta$, with some constant $c_{\theta''}>0$. Moreover, \eqref{eq:foregorov1} also holds for $\varphi^\lambda$.
	
	Now, we construct the common Egorov set as follows: Let
	$$
	\Omega_n^c:=\left\{\omega\in\mathcal{A}^n:\text{ there exists }\tau\in[\omega]\text{ such that }\left|\frac1n S_n\phi_{\lambda_0}(\tau)-\int\phi_{\lambda_0} d\mu_{\lambda_0}\right|>4\varepsilon\right\}.
	$$
	If $\omega\in\Omega_n^c$ then for every $\tau\in[\omega]$,
	$$\left|\frac1n S_n\phi_\lambda(\tau)-\int\phi_\lambda d\mu_\lambda\right|>\varepsilon\text{ (by the choice $\lambda\in B_\xi(\lambda_0)$)},
	$$
	and so $\mu_{\lambda}(\Omega_n^c)\leq Ce^{-ns/2}$ for every $\lambda\in B_\xi(\lambda_0)$ and $n\geq1$.
	Let $n_k:=k$ and $m_k=n_1+\ldots+n_k$. Finally, for every $K\geq1$ let
	$$
	A_K=\Omega_{m_K}\times\Omega_{n_{K+1}}\times\Omega_{n_{K+2}}\times\cdots\subset\Sigma.
	$$
	Clearly, by \eqref{eq:foregorov1} and the shift invariance of $\mu_\lambda$ we have
	$$
	\mu_\lambda(A_K^c)\leq\mu_{\lambda}(\Omega_{m_K}^c)+\sum_{k=1}^\infty\mu_{\lambda}(\Omega_{m_K}\times\Omega_{n_{K+1}}\times\cdots\times\Omega_{n_{K+k}}^c)\leq Ce^{-m_Ks/2}+\sum_{k=1}^\infty Ce^{-n_{K+k}s/2}\to0,
	$$	
	as $K\to\infty$, uniformly in $\lambda\in B_\xi(\lambda_0)$, which shows \ref{it:prop1} in Proposition~\ref{prop:maingibbs}. On the other hand, for every $\omega\in A_K$ and for every $n\geq m_K$ and every $\lambda\in B_\xi(\lambda_0)$,
	\begin{equation}\label{eq:egorovprop}
		\left|\frac{1}{n}S_n\phi_\lambda(\omega)-\int\phi^{\lambda} d\mu_{\lambda}\right|\leq 6\varepsilon.
	\end{equation}
	Hence, by \eqref{R54}, \eqref{R66} and \eqref{eq:egorovprop}
	\[\begin{split}
		\mathcal{E}_\gamma(\mu_\lambda|_{A_K},d_\lambda)&=\sum_{n=0}^\infty\sum_{|\omega|=n}\sum_{i\neq j\in\mathcal{A}}|f_{\omega}^\lambda(X)|^{-\gamma}\mu_{\lambda}|_{A_K}([\omega i])\mu_{\lambda}|_{A_K}([\omega j])\\
		&\leq \sum_{n=0}^\infty\sum_{|\omega|=n}e^{-n(h_{\mu_\lambda}-6\varepsilon-\gamma(\chi_{\mu_\lambda}+6\varepsilon))}\mu_{\lambda}|_{A_K}([\omega])<\infty
	\end{split}\]
	if $\frac{h_{\mu_\lambda}-6\varepsilon}{\chi_{\mu_\lambda}+6\varepsilon}>\gamma$. This completes the proof of \ref{it:prop2} in Proposition~\ref{prop:maingibbs}.
	
	Finally, let $\omega\in\mathcal{A}^{m_L}$ for some $L\geq K$. We may assume without loss of generality that $\omega\in\Omega_{m_K}\times\Omega_{n_{K+1}}\times\cdots\times\Omega_{n_{L}}$. Then
	$$
	\mu_{\lambda}([\omega]\cap A_K)=\mu_{\lambda}([\omega])-\sum_{p=L+1}^\infty\sum_{\tau\in\Omega_{m_{K}}\times\cdots\times\Omega_{n_{p}}^c}\mu_{\lambda}([\omega\tau]).
	$$
	For short, let $b_{p}(\lambda):=\frac{1}{\mu_{\lambda}([\omega])}\sum_{\tau\in\Omega_{m_{K}}\times\cdots\times\Omega_{n_{p}}^c}\mu_{\lambda}([\omega\tau])$. By the quasi-Bernoulli property of the Gibbs measures and $\mu_{\lambda_0}(\Omega_n^c)\leq Ce^{-ns/2}$, we have that $b_p(\lambda_0)\leq e^{-n_{p}s/2}$. Hence,
	\[\begin{split}
		\frac{\mu_{\lambda}([\omega]\cap A_K)}{\mu_{\lambda_0}([\omega]\cap A_K)}&=\frac{\mu_{\lambda}([\omega])}{\mu_{\lambda'}([\omega])}\cdot\frac{1-\sum_{p=L+1}^\infty b_{p}(\lambda)}{1-\sum_{p=L+1}^\infty b_{p}(\lambda_0)}\\
		&\leq e^{c_{\theta'}|\lambda-\lambda_0|^{\theta'}m_L}\frac{1-\sum_{p=L+1}^\infty e^{-c_{\theta'}|\lambda-\lambda_0|^{\theta'}(m_p+m_L)}b_{p}(\lambda_0)}{1-\sum_{p=L+1}^\infty b_{p}(\lambda_0)}\text{ (by \eqref{eq:contmu})}\\
		&\leq e^{c_{\theta'}|\lambda-\lambda_0|^{\theta'}m_L}\frac{1-\sum_{p=L+1}^\infty e^{-c_{\theta'}|\lambda-\lambda_0|^{\theta'}(m_p+m_L)}e^{-n_ps/2}}{1-\sum_{p=L+1}^\infty e^{-n_ps/2}}\\
		&\leq e^{c_{\theta'}|\lambda-\lambda_0|^{\theta'}m_L}e^{\frac{1-\sum_{p=L+1}^\infty (m_p+m_L)e^{-n_ps/2}}{1-\sum_{p=L+1}^\infty e^{-n_ps/2}}c_{\theta'}|\lambda-\lambda_0|^{\theta'}}\text{ (by the Mean Value Theorem)}\\
		&\leq e^{c_{\theta'}|\lambda-\lambda_0|^{\theta'}m_L}e^{m_L\frac{1}{1-\sum_{p=1}^\infty e^{-n_ps/2}}c_{\theta'}|\lambda-\lambda_0|^{\theta'}}=:e^{c'|\lambda-\lambda_0|^{\theta'}m_L}.
	\end{split}\]
	For general $\omega\in\Sigma^*$ with $|\omega|\geq m_K$, taking $L\geq K$ such that $m_L<|\omega|\leq m_{L+1}$, we get
	\[\begin{split}
		\mu_\lambda([\omega]\cap A_K)&=\sum_{|\tau|=m_{L+1}-|\omega|}\mu_{\lambda}([\omega\tau]\cap A_K)\\
		&\leq e^{c'|\lambda-\lambda_0|^{\theta'}m_{L+1}}\sum_{|\tau|=m_{L+1}-|\omega|}\mu_{\lambda_0}([\omega\tau]\cap A_K)\\
		&\leq e^{c'|\lambda-\lambda_0|^{\theta'}\frac{m_{L+1}}{m_L}|\omega|}\mu_{\lambda_0}([\omega]\cap A_K),
	\end{split}\]
	which completes the proof of \ref{it:prop3} in Proposition~\ref{prop:maingibbs}. \qed
}

\appendix

\section{Various notions of dimension}
In this paper we focus on the absolute continuity and Hausdorff dimension of families of invariant measures. To do so we frequently use two other dimensions of a measure: the correlation and Sobolev dimensions.
\begin{definition}\label{R77}
Let $(X,\rho )$ be a complete metric space. For a Borel set $A \subset X$ we write $\mathcal{M}(A)$ for the
collection of Borel measures $\mu $ satisfying
\begin{enumerate}
	\item The support  $\mathrm{spt}(\mu)\subset A$,
	\item $\mathrm{spt}(\mu)$ is compact, and
	\item $0<\mu(A)<\infty $.
\end{enumerate}
Moreover,
we denote the set of Borel probability measures on $A$ by $\mathcal{P}(A)$.	
\end{definition}
\subsection{The local and Hausdorff dimensions of a measure}
Let $\mu \in \mathcal{M}(X)$, where $(X,\rho )$ is a complete metric space. The Hausdorff dimension of the measure $\mu $ is
\begin{equation}
\label{R83}
\Dh\mu:=\inf\left\{\dim_{\rm H} A:\mu( X \setminus A)=0\right\},\text{ and }
\underline{\dim}_H\mu:=\inf\left\{\dim_{\rm H} A:\mu(A)>0\right\}.
\end{equation}
This implies that $\dim_{\rm H}  \mu \leq \dim_{\rm H}  A$.
One way to give effective bounds for $\Dh\mu$ is to estimate the local dimensions of $\mu $.
The lower  local dimension of $\mu$ at  $x\in A$ is:
    \begin{equation}\label{O63}
 \underline{\dim}(\mu,x):=  \liminf\limits_{n\to\infty} \frac{\log \mu(B(x,r))}{\log r}\,.
    \end{equation}
	The upper local dimension $\overline{\dim}(\mu,x)$ of $\mu $  at $x$,  is defined in an analogous way.
    We say that  the measure $\mu$ is exact dimensional  if the limit
    $\lim\limits_{r\downarrow 0} \frac{\log \mu(B(x,r))}{\log r}$ exists and is constant $\mu$-almost surely. This constant is denoted by $\dim\mu$.
	A well known characterization of $\Dh\mu$ is as follows:
	\begin{equation}
	\label{R84}
	\Dh\mu=\underset{x\sim \mu}{\mathrm{ess sup}}\ 
    \underline{\dim}(\mu,x)
     \ \  \text{ and }\ \ 
	\underline\dim_{\rm H}\mu= \underset{x\sim \mu}{\mathrm{ess inf}}\ 
    \underline{\dim}(\mu,x).
	\end{equation}
Another effective way to give lower bound on $\Dh\mu$ is to estimate the so-called correlation dimension of $\mu $.

\subsection{Correlation and Sobolev dimensions of a measure}\label{R79}
Let $(X,\rho )$ be a complete metric space, let $\mu $ be a Borel measure on $X$,  and $\alpha >0$.
Define the \texttt{$\alpha$-energy} as
\begin{equation}\label{eq:energy} \Ek_{\alpha}(\mu, d) = \iint \rho (x,y)^{-\alpha}d\mu(x)d\mu(y). \end{equation}
Define the \texttt{correlation dimension} of $\mu$ with respect to the metric $\rho $ as
\[ \dim_{cor}(\mu, \rho ) = \sup \{ \alpha > 0 : \Ek_{\alpha}(\mu, d) < \infty \}. \]
If $\mu $ is a Borel measure on $\mathbb{R}^d$ {with the Euclidean metric}, then the correlation dimension and the $\alpha $-energy  of $\mu $ are denoted by $\dim_{cor}(\mu)$ and $\mathcal{E}_\alpha (\mu )$. In this case we have
\begin{equation}
\label{R82}
\dim_{cor}(\mu)\leq \underline{\dim}_{\rm H}\mu  \leq \dim_{\rm H}  \mu.
\end{equation}
Moreover, for a Borel set $A\subset \mathbb{R}^d$ we have
\begin{equation}
\label{R80}
\dim_{\rm H}  A=
\sup\left\{s: \exists \ \mu \in\mathcal{M}(A),\
\mathcal{E}_s(\mu )<\infty
\right\}.
\end{equation}
From now on we assume that $\nu $ is a finite Borel measure on $\mathbb{R}$.
To define the Sobolev dimension of $\nu$,  first we recall that the Fourier transform of $\nu$ is defined by
$\widehat{\nu}(\xi)=\int e^{i\xi x}d\nu(x).$
The \texttt{homogenous Sobolev norm} of a finite Borel measure $\nu$, for $\gamma \in \R$, is
\[ \| \nu \|^2_{2,\gamma} = \int_{\R}|\widehat{\nu}(\xi)|^2|\xi|^{2\gamma}d\xi. \]
If $\| \nu \|^2_{2,\gamma}<\infty$ then we say that
\texttt{$\nu $ has a fractional derivative of order $\gam$ in $L^2$}.
The  \texttt{Sobolev dimension} is defined as follows:
\begin{equation}
\label{R27}
\dim_S\nu := \sup \left\{ \alpha \in \R : \int_\R |\widehat{\nu}(\xi)|^2(1+|\xi|)^{\alpha - 1}d\xi < \infty \right\}.
\end{equation}
It is well known (see \cite[Section 5.2]{Mattila Fourier})
 that if $0 \leq \dim_S\nu \leq \infty$, for $\alpha > 0$, then we have
\begin{equation}
\label{R28}
\int_\R |\widehat{\nu}(\xi)|^2(1+|\xi|)^{\alpha - 1}d\xi < \infty\  { \iff}\  \int_\R |\widehat{\nu}(\xi)|^2|\xi|^{\alpha - 1}d\xi = \| \nu \|_{2,\frac{\alpha - 1}{2}}^2 < \infty.
\end{equation}
The \texttt{Sobolev energy} of the measure $\nu $ of degree $\alpha $  is
	\begin{equation}
	\label{R76}
	\mathcal{I}_\alpha :=\int_\R |\widehat{\nu}(\xi)|^2|\xi|^{\alpha - 1}d\xi.
	\end{equation}
	Then by \eqref{R27} and \eqref{R28} we have $\dim_S\nu = \sup
	\left\{
s:\ \mathcal{I}_s(\nu )<\infty
	\right\}$,
	see \cite[p.74]{Mattila Fourier}.
	The connection between the Sobolev energy $\mathcal{I}_s(\nu )$  and the $\alpha $-energy defined in \eqref{eq:energy}  is as follows:  If $s\in(0,1)$ then there exists a constant $\gamma =\gamma (s)>0$  such that
	for every finite Borel measure $\nu $ on $\mathbb{R}$ we have
	(see \cite[Theorem 3.10]{Mattila Fourier})
	\begin{equation}
	\label{R73}
	\mathcal{E}_s(\nu )=\gamma \cdot \mathcal{I}_s(\nu ).
	\end{equation}
	This identity does not extend  to $s=1$ (see
	\cite[p.74]{Mattila Fourier} ).

\begin{lemma}\label{R81} Let $\nu $ be a finite Borel measure on $\mathbb{R}^1$.
\begin{enumerate}
	\item If $0<\dim_{\rm S}\nu <1  $ then $\dim_{\rm S}\nu =\dim_{cor}(\nu)  $.
	\item  If $\dim_{\rm S}\nu=\sigma  >1  $ then
\begin{enumerate}
	\item
	 $\nu$ is absolutely continuous with Radon-Nikodym derivative in $L^2(\R)$,
\item $\nu$ has fractional derivatives
of order $\frac{\sigma -1}{2}>0$
in $L^2$, see \cite[Theorem 5.4]{Mattila Fourier}.
\end{enumerate}
\end{enumerate}
\end{lemma}

\section{The precise statements of our assumptions}\label{R20}

\begin{enumerate}[start=1,label={(MA\arabic*)}]
	\item\label{as:C2 mp} the maps $f_j^\lam$ are $C^{2+\delta}$-smooth on $X$ with $M_1 = \sup \limits_{\lam \in U} \sup \limits_{j \in \Ak} \left\{ \left\| \frac{d^2}{dx^2}f^\lam_j \right\|_{\infty} \right\}  < \infty $ and there exist constants $C_1, C_2 > 0$ such that
	\[ \left| \frac{d^2}{dx^2}f^\lam_j(x) - \frac{d^2}{dx^2}f^\lam_j(y) \right| \leq C_1|x-y|^\delta \text{ and } \left| \frac{d^2}{dx^2}f^{\lam}_j(x) - \frac{d^2}{dx^2}f^{\lam'}_j(x) \right| \leq C_2|\lam-\lam'|^\delta  \]
	hold for all $x,y \in X,\ j \in \Ak,\ \lam, \lam' \in U$.
	
	\bigskip
	
	\item\label{as:lam hoelder mp} the maps $\lam \mapsto f^\lam_j(x)$ are $C^{1+\delta}$-smooth on $U$ and there exists a constant $C_3 > 0$ such that
	\[\left|\frac{\partial}{\partial \lam_i} f_j^{\lam}(x) - \frac{\partial}{\partial \lam_i} f_j^{\lam'}(x)\right| \leq C_3 | \lam - \lam'|^\delta\] holds for all $x \in X,\ j \in \Ak,\ \lam, \lam' \in U,\ 1 \leq i \leq d$.
	
	\bigskip
	
	\item\label{as:dxdlam mp} the second partial derivatives $\frac{\partial^2}{\partial x \partial\lam_i}f^\lam_j(x), \frac{\partial^2}{\partial\lam_i \partial x}f^\lam_j(x)$ exist and are continuous on $U \times X$ (hence equal) with $M_2 = \sup \limits_{1 \leq i \leq d} \ \sup \limits_{j \in \Ak}\ \sup \limits_{\lambda \in U} \left\| \frac{\partial^2}{\partial \lam_i \partial x} f^\lam_j(x)\right\|_{\infty} < \infty$, and there exist constants $C_4, C_5>0$ such that
	\[ \left| \frac{\partial^2}{\partial x \partial \lam_i}f^{\lam}_j(x) - \frac{\partial^2}{\partial x \partial\lam_i}f^{\lam}_j(y) \right| \leq C_4 |x-y|^\delta \text{ and } \left| \frac{\partial^2}{\partial x \partial \lam_i}f^{\lam}_j(x) - \frac{\partial^2}{\partial x\partial\lam_i}f^{\lam'}_j(x) \right| \leq C_5 |\lam - \lam'|^\delta \]
	hold for all $x,y \in X,\ j \in \Ak,\ \lam, \lam' \in U$ and $1 \leq i \leq d$.
\end{enumerate}

\end{document}